\documentclass[11pt,reqno]{amsart}

\usepackage[T1]{fontenc}
\usepackage{lmodern}
\usepackage{microtype}
\usepackage{amsmath,amssymb,amsthm,mathtools}
\usepackage{enumitem}
\usepackage{xcolor}
\usepackage[colorlinks=true,linkcolor=blue!55!black,citecolor=blue!55!black,urlcolor=blue!55!black]{hyperref}

\newtheorem{theorem}{Theorem}[section]
\newtheorem{proposition}[theorem]{Proposition}
\newtheorem{lemma}[theorem]{Lemma}
\newtheorem{corollary}[theorem]{Corollary}
\theoremstyle{definition}
\newtheorem{definition}[theorem]{Definition}
\newtheorem{question}[theorem]{Question}
\theoremstyle{remark}
\newtheorem{remark}[theorem]{Remark}

\newcommand{\Z}{\mathbb Z}
\newcommand{\N}{\mathbb N}
\newcommand{\Rcal}{\mathcal R}
\newcommand{\ndiam}{\operatorname{ndiam}}
\newcommand{\rank}{\operatorname{rank}}
\newcommand{\HS}{\operatorname{HS}}
\newcommand{\inideal}{\operatorname{in}}

\title[Iterated-sumset spectra]{Iterated-sumset spectra: The complete exponent law and its rank geometry}
\author{Henry Shin}
\date{}
\hypersetup{
 pdftitle={Iterated-sumset spectra: The complete exponent law and its rank geometry},
 pdfauthor={Henry Shin},
 pdfsubject={Iterated-sumset spectra, active-relation rank, and finite-observation geometry},
 pdfkeywords={iterated sumsets, sumset cardinalities, additive combinatorics,
Bh-sets, generalized Golomb rulers, finite Freiman models,
Hilbert functions of projective monomial curves, Cohen-Macaulay toric rings,
hyperplane arrangements, Ehrhart quasipolynomials}
}

\subjclass[2020]{Primary 11B13; Secondary 11P70, 05E40, 13F65,
52C35, 52B20, 11L07}
\keywords{iterated sumsets, sumset cardinalities, additive combinatorics,
$B_h$-sets, generalized Golomb rulers, finite Freiman models,
Hilbert functions of projective monomial curves, Cohen--Macaulay toric rings,
hyperplane arrangements, Ehrhart quasipolynomials}

\begin{document}

\begin{abstract}
For integers $h,k\geq1$, let $hA$ be the $h$-fold sumset of $A$ and put
$\Rcal(h,k)=\{|hA|:A\subset\mathbb Z,\ |A|=k\}$.  Previously, the
fixed-cardinality exponent law was known only for $k\leq3$; every fixed
$k\geq4$ remained open.  We settle the problem in full by determining the complete fixed-cardinality exponent law:
\[
 |\Rcal(h,k)|=
 \begin{cases}
  1,&k\leq2,\\
  h,&k=3,\\
  h^{k-1+o_k(1)},&k\geq4,
 \end{cases}
\]
where $o_k(1)\to0$ as $h\to\infty$ with $k$ fixed.  More sharply, for fixed
$k\geq4$, an interval of length $\Theta_k(h^{k-1})$ contains at least
$h^{k-1-o_k(1)}$ attainable values.  At $k=4$ we prove
$|\Rcal(h,4)|=\Theta(h^3)$ with positive lower density in its ambient interval,
disproving Nathanson's proposed $o(h^3)$ and $O(h^2)$ bounds.

One bounded addition-table geometry drives these results, coupling
Hilbert--energy amplification to optimal finite-observation compression.
Every ordered real $k$-set ($k\geq2$) has an integer model in
$[0,O_k(h^{k-2})]$ preserving every sum equality and strict comparison
through degree $h$; the exponent $k-2$ is sharp.  The universal
label-realization length is therefore $\Theta_k(h^{k-2})$, one power sharper
than Nathanson's $O_k(h^{k-1})$ bound.  For $h\geq2$ and $k\geq3$, minimum
active rank equals realization-frequency codimension, exponent-shape
codimension, and sampling-rarity exponent; a full-exponent family has
maximal-rank witnesses with Cohen--Macaulay toric coordinate rings.

At rank zero, for $h\geq2$, it proves the conjectural OEIS A227589 formula
$\binom{h+2}{2}+\mathbf1_{\{2\nmid h\}}$ for the least normalized diameter
of a four-point $B_h$-set.  It also gives exact fixed-$(h,k)$ popularity
laws for $k$-subsets of $\{1,\ldots,q\}$ as $q\to\infty$, resolving
Nathanson's Problems~9 and~10.
\end{abstract}

\maketitle

\clearpage
\thispagestyle{plain}
\enlargethispage{4\baselineskip}
\begingroup
\small
\makeatletter
\renewcommand{\@pnumwidth}{2em}
\def\l@part{\@tocline{-1}{8pt plus1pt}{0pt}{}{\bfseries}}
\makeatother
\tableofcontents
\endgroup
\clearpage

\section{Introduction}

For a finite set $A$ in an additive abelian group $G$ and an integer
$h\geq1$, its $h$-fold sumset is
\[
 hA=\{a_1+\cdots+a_h:a_1,\ldots,a_h\in A\}.
\]
We use the degree-zero convention $0A=\{0_G\}$, the singleton containing the
empty sum, so $|0A|=1$.
For a positive integer $t$, we distinguish the scalar dilation
\[
 t\mathbin{\cdot}A=\{ta:a\in A\}
\]
from the fold sumset $tA$.  We also write $[m]=\{1,\ldots,m\}$ for
$m\geq1$ and $[0]=\varnothing$.
The classical theory largely studies the structure of $A$ when $|hA|$ is
close to one of its two extremes.  A different problem, initiated by
Nathanson, is to understand the full spectrum
\[
 \Rcal(h,k)=\{|hA|:A\subset\Z,\ |A|=k\}.
\]
For an additive abelian group $G$, we similarly write
\[
 \Rcal_G(h,k)=\{|hA|:A\subset G,\ |A|=k\}.
\]
For $A\subset\mathbb Z$ with $|A|\geq2$, put
\[
 g(A)=\gcd\{a-a':a,a'\in A\},\qquad
 \ndiam(A)=\frac{\max A-\min A}{g(A)}.
\]
We call $A$ \emph{primitive} when $g(A)=1$.
Thus $\ndiam(A)$ is invariant under
$A\mapsto\{ta+s:a\in A\}$ for $t\in\mathbb Z\setminus\{0\}$ and
$s\in\mathbb Z$.
For fixed $k$, every member of $\Rcal(h,k)$ lies in the interval
\begin{equation}\label{eq:ambient}
 hk-h+1\leq |hA|\leq \binom{h+k-1}{k-1}.
\end{equation}
The lower endpoint is attained by arithmetic progressions.  A $B_h$-set is a
set for which every element of $hA$ has a unique nondecreasing
representation as a sum of $h$ elements of $A$; these sets attain the upper
endpoint.  Erd\H{o}s and Szemer\'{e}di observed that for $h=2$ every
integer between the endpoints occurs \cite{ErdosSzemeredi1983}.  For
 $h\geq3$, gaps appear in general, already for $k\geq3$.

The question most relevant here is the following problem of Nathanson
\cite[Problem~3]{NathansonTriangular}.

\begin{question}[Nathanson]\label{q:nathanson}
For fixed $k\geq4$, is
\[
 |\Rcal(h,k)|=o_k(h^{k-1})?
\]
More strongly, is $|\Rcal(h,k)|=O_k(h^{k-2})$?
\end{question}

The surrounding literature separates four different asymptotic questions.
Khovanskii and the Hilbert-function literature fix one set $A$ and let the
order $h$ grow \cite{Khovanskii1992,Khovanskii1995,EliahouMazumdar2022};
Nathanson--Ruzsa prove corresponding eventual polynomiality for finite
subsets of arbitrary abelian semigroups and for multivariate sums
\cite{NathansonRuzsa2002}.  Nathanson posed finite-window comparison questions
\cite{NathansonInverse2026}; Kravitz and, subsequently,
Fox--Kravitz--Zhang prescribe comparisons or exact differences among several
iterated-sumset profiles on such a window
\cite{Kravitz2025,FoxKravitzZhang2026}.  Rajagopal instead fixes $h$ and
determines the spectrum when $k$ is sufficiently large
\cite{Rajagopal2026}.  Here $k$ is fixed, $h\to\infty$, and we count the
attainable values of $|hA|$ as $A$ ranges over the $k$-element subsets of
$\mathbb Z$.  Nathanson
initiated this last problem, proved the exact three-point formula, and posed
Question~\ref{q:nathanson}; related work treats lower-edge inverse problems,
forbidden intervals, local four-point classification, and relation lattices
\cite{NathansonProblems,NathansonExplicit2025,NathansonTriangular,
TangXing2021,MohanPandey2023,Schinina2026,Jung2026,OBryant2025}.

\subsection*{The complete fixed-cardinality exponent law}

The interval in \eqref{eq:ambient} contains $\Theta_k(h^{k-1})$ integers.
Our first theorem shows that this ambient exponent $k-1$ is attained at
every fixed cardinality from the first open case onward.

\begin{theorem}[Complete fixed-cardinality exponent law]
\label{thm:endpoint-spectrum}
For every fixed integer $k\geq4$, there are constants
$0<\alpha_k<\beta_k$ such that, as $h\to\infty$,
\begin{equation}\label{eq:endpoint-local}
 \left|\Rcal(h,k)\cap
 [\alpha_kh^{k-1},\beta_kh^{k-1}]\right|
 \geq
 h^{k-1}\exp\!\left(
  -O_k\!\left(\frac{\log h}{\log\log h}\right)\right).
\end{equation}
Consequently
\begin{equation}\label{eq:ambient-exponent}
 \lim_{h\to\infty}\frac{\log|\Rcal(h,k)|}{\log h}=k-1,
 \qquad\text{equivalently}\qquad
 |\Rcal(h,k)|=h^{k-1+o_k(1)}.
\end{equation}
Here $o_k(1)\to0$ as $h\to\infty$; the subscript records that $k$ is fixed.
Together with the elementary cases, the complete fixed-cardinality exponent
law is
\begin{equation}\label{eq:phase-diagram}
 \lim_{h\to\infty}\frac{\log|\Rcal(h,k)|}{\log h}=
 \begin{cases}
  0,&k=1,2,\\
  1,&k=3,\\
  k-1,&k\geq4.
 \end{cases}
\end{equation}
\end{theorem}

At the first open cardinality the subpolynomial loss disappears, and the
spectrum occupies a positive proportion of its ambient interval.

\begin{theorem}[Positive-density four-point spectrum]\label{thm:main}
There is an absolute constant $c>0$ such that, for all sufficiently large
$h$,
\begin{equation}\label{eq:quantmain}
 \left|\left\{|hA|:A\subset\Z,\ |A|=4,
 \ \ndiam A<\frac{13}{20}h^2\right\}\right|\geq ch^3.
\end{equation}
More locally,
\begin{equation}\label{eq:localmain}
 \left|\Rcal(h,4)\cap
 \left[\binom{h+3}{3}-\frac{h^3}{250},
       \binom{h+3}{3}-\frac{h^3}{400}\right]\right|\geq ch^3,
\end{equation}
where the endpoints may be rounded to integers.  Consequently
\[
 \boxed{|\Rcal(h,4)|=\Theta(h^3)},
 \qquad
 \liminf_{h\to\infty}
 \frac{|\Rcal(h,4)|}{|[3h+1,\binom{h+3}{3}]\cap\Z|}>0.
\]
The values counted in \eqref{eq:localmain} may be chosen with primitive
representatives having two independent homogeneous relations of degree at
most $h$.
\end{theorem}

The exponent law is accompanied by the following structural refinement.  At $k=4$,
Theorem~\ref{thm:main}, together with the construction underlying
Theorem~\ref{thm:maximal-rank-saturation}\textup{(i)}, shows that the
localized bound \eqref{eq:localmain} is realized by maximal-rank weak
addition-table types at the sharp relation-birth and compression scale, with
explicit Cohen--Macaulay type-two witnesses.  For every fixed $k\geq4$,
part~\textup{(i)} gives a full-exponent family with the same structural
properties in a macroscopic interval.  In particular, for every
$\varepsilon>0$ and fixed $k\geq4$,
\begin{equation}\label{eq:endpoint-epsilon-form}
 |\Rcal(h,k)|\gg_{k,\varepsilon}h^{k-1-\varepsilon}.
\end{equation}
For $k\geq5$, the theorem determines the exponent but does not assert
positive density.

Thus both proposed conclusions in Question~\ref{q:nathanson} fail already
at $k=4$.  For every fixed $k\geq5$, Theorem~\ref{thm:endpoint-spectrum}
separately rules out the stronger $O_k(h^{k-2})$ proposal and determines the
ambient polynomial exponent.

\paragraph{One family and two complementary analyses.}
The \emph{Hilbert--energy construction} builds a four-parameter chart of
four-point sets.  Suspension propagates witnesses and Hilbert profiles;
cubic energy controls fixed-tuple collisions, and prime coding bounds
cross-tuple overlap.
The \emph{finite-observation analysis} interprets the bounded
addition tables of the same witnesses as rational sweep faces, yielding
rank, compression, shape, type, and frequency.
Theorem~\ref{thm:maximal-rank-saturation} combines the two: it produces a
full-exponent family contained in one intrinsic minimum-rank layer, while
every label in that family retains an explicit maximal-rank arithmetically
Cohen--Macaulay (ACM) witness at the optimal compression scale.  The
subsequent rigidity, transfer, and popularity theorems extract sharp
consequences of the same structure.

The change at $k=4$ is structural: three-point projective toric ideals are
principal, whereas the four-point case admits a four-parameter
codimension-two Hilbert--Burch family.  Suspended three-point hypersurfaces give
$h^{k-2-o(1)}$ distinct labels on the $h^{k-1}$-scale with
complete-intersection, Gorenstein
type-one witnesses, while the four-point seed gives $h^{k-1-o(1)}$ labels
with type-two witnesses; see
Corollary~\ref{cor:complete-intersection-spectrum-tower} and
Theorem~\ref{thm:maximal-rank-saturation}.

\subsection*{The common finite observation}

For integers $j\geq0$ and $k\geq1$, put
\[
 \mathcal U_{j,k}=\{u\in\mathbb Z_{\geq0}^k:\textstyle\sum_i u_i=j\},
 \qquad
 M_{j,k}=|\mathcal U_{j,k}|=\binom{j+k-1}{k-1}.
\]
For a real number $x$, write $x^+=\max(x,0)$ and
$x^-=\max(-x,0)$.  Thus, for every real zero-sum vector $z$,
\[
 \rho(z):=\sum_i z_i^+=\sum_i z_i^-=\frac12\lVert z\rVert_1.
\]
We call $\rho(z)$ the \emph{relation degree} of $z$.
For an integer $h\geq1$ and an ordered real $k$-tuple
$A=(a_1<\cdots<a_k)$, put
\[
 \begin{aligned}
 L_h(A)&=\operatorname{span}_{\mathbb R}
 \left\{z\in\mathbb Z^k:
 z\cdot A=0,\ \sum_i z_i=0,\ \rho(z)\leq h\right\},\\
 r_h(A)&=\dim L_h(A).
 \end{aligned}
\]
We call $r_h(A)$ the degree-$h$ active relation rank.  For
$1\leq i\leq r_h(A)$, put
\[
 p_i(A)=\min\{p\geq1:\dim L_p(A)\geq i\}.
\]
These are the successive relation-birth degrees.
For a primitive integer $k$-set we use the same definition for every
$1\leq i\leq k-2$; all such birth degrees are finite because the full
integral homogeneous relation lattice has rank $k-2$.

The finite additive observation common to all parts of the paper is
\[
 \mathsf O_h(A)=
 \left(
  \operatorname{sgn}\bigl((u-v)\cdot A\bigr)
 \right)_{\substack{1\leq j\leq h\\u,v\in\mathcal U_{j,k}}}.
\]
Its zero coordinates encode all weak addition tables through degree $h$; at
degree $h$ their classes are the elements of $hA$, and their relation
directions span $L_h(A)$.  Sweep faces classify this observation, compression
gives a short integral representative, frequency counts its integer fibre,
and the Hilbert function records its class-count label.  Cyclic suspension
transports the Hilbert profile and relation-birth filtration between the two
analyses.

Active rank organizes the structural side of the paper.  Rank zero is the
$B_h$ regime; at active rank one, the visible degree-$h$ collision piece is
principal; and at maximal rank $k-2$, the full homogeneous relation space is
already visible by degree $h$.
These three distinguished ranks yield complementary exact conclusions.  At
rank zero the
four-point endpoint has an optimal finite model, settling the conjectural
four-mark $B_h$-ruler formula recorded in OEIS A227589
\cite{OEISA227589}.  At rank one each
representation fibre
is a lattice chain, so one primitive direction determines the complete
collision-multiplicity histogram.  At maximal rank the endpoint construction
supplies the full spectral exponent together with rigid ACM witnesses.

Following Nathanson, let $N(h,k)$ be the least $N$ such that every value in
$\Rcal(h,k)$ has a representative in $[0,N-1]\cap\mathbb Z$.  The
weak $h$-addition-table type of an ordered set $A$ is the equivalence relation
on $\mathcal U_{h,k}$ given by
\[
 u\sim_A v\quad\Longleftrightarrow\quad u\cdot A=v\cdot A.
\]
 Write $\mathfrak T_{h,k}^{(r)}$ for the weak types of active rank $r$ realized
 by strictly increasing integer tuples.  The signed observation, or ordered
 face type, refines this weak equality type, which in turn refines the single
 cardinality label.  The complete table is canonically the label of a face of a
type-$A$ root-sweep zonotope: face dimension is active rank, and the relative
interior of its normal cone is the space of weights inducing that table.
 For $k\geq2$, translation and positive dilation are removed by the open
 normalized exponent-shape simplex
\[
 \mathcal S_k=\{x\in\mathbb R^k:0=x_1<x_2<\cdots<x_k=1\}.
\]
For a weak type $\mathsf T$, its shape locus is the set of points of
$\mathcal S_k$ inducing $\mathsf T$.  It is a finite union of relatively open
ordered sweep cells, one for each complete signed/ordered face component.
The shape locus of a Hilbert value is the finite union over the types carrying
that value.
This identifies the arrangement, Hilbert function, shape space, frequency,
and compression problem before any counting begins.
Theorem~\ref{thm:active-structure-intro} quantifies filtered compression, the
optimal universal scale, type diversity by rank, and exact
realization-frequency degrees.  We first introduce the label-level notation
needed to combine these conclusions with the endpoint family.

For $n\in\Rcal(h,k)$, define its minimum active rank by
\[
 r_{h,k}(n)=\min\{r_h(A):A\subset\mathbb Z,\ |A|=k,\ |hA|=n\},
\]
and put
\[
 \begin{aligned}
 \Rcal^{[r]}(h,k)&=\{n\in\Rcal(h,k):r_{h,k}(n)=r\},\\
 F^{(k)}_{h,n}(q)&=
 \#\{A\in\tbinom{[q]}k:|hA|=n\}.
 \end{aligned}
\]

\subsection*{The full-exponent spectrum--geometry theorem}

Put
\[
 \mathfrak T^{\max}_{h,k}=\mathfrak T_{h,k}^{(k-2)},
 \qquad
 \Lambda^{\max}_{h,k}
 =\{ |\mathcal U_{h,k}/\mathsf T|:
       \mathsf T\in\mathfrak T^{\max}_{h,k}\}.
\]
We use the following standard algebraic shorthand in the statement below.
Over an arbitrary field $K$, after translating a primitive integer
representative to $A=\{0=a_0<\cdots<a_{k-1}=D\}$, let
\[
 \begin{aligned}
  \varphi_A&:K[x_0,\ldots,x_{k-1}]\longrightarrow K[u,v],
  &x_i&\longmapsto u^{D-a_i}v^{a_i},\\
  I_A&=\ker\varphi_A,
  &R_A&=K[x_0,\ldots,x_{k-1}]/I_A.
 \end{aligned}
\]
The attached projective monomial curve is $\operatorname{Proj}R_A$, and
$\xi$ records homological degree in its Betti polynomial.
\begin{theorem}[Full-exponent spectrum--geometry synthesis]
\label{thm:maximal-rank-saturation}
Fix $k\geq4$.
\begin{enumerate}[label=\textup{(\roman*)}]
\item \emph{Structured label saturation.}
There are constants \(0<\alpha_k<\beta_k\) such that, for every sufficiently
large $h$, there is a set
\[
 \mathcal L^{\mathrm{end}}_{h,k}\subseteq
 \Lambda^{\max}_{h,k}\cap[\alpha_kh^{k-1},\beta_kh^{k-1}]
\]
with
\begin{equation}\label{eq:maximal-rank-label-saturation}
 |\mathcal L^{\mathrm{end}}_{h,k}|
 \geq h^{k-1}
 \exp\!\left(-O_k\!\left(\frac{\log h}{\log\log h}\right)\right).
\end{equation}
Each $n\in\mathcal L^{\mathrm{end}}_{h,k}$ has a primitive representative
$A$ satisfying
\[
 \begin{gathered}
  r_h(A)=k-2,\qquad
  p_i(A)\asymp_k h\quad(1\leq i\leq k-2),\\
  \ndiam(A)\asymp_k\prod_{i=1}^{k-2}p_i(A)\asymp_k h^{k-2}.
 \end{gathered}
\]
Its defining ideal is a prime toric almost complete intersection, while its
coordinate ring is arithmetically Cohen--Macaulay of type two, with total
Betti polynomial $(1+2\xi)(1+\xi)^{k-3}$ and graded Betti tower
\eqref{eq:endpoint-betti-tower}.

\item \emph{One intrinsic rank layer.}
For every sufficiently large $h$ there are
\[
 r^\ast_{h,k}\in\{2,\ldots,k-2\}
 \quad\text{and}\quad
 \mathcal E_{h,k}\subseteq
 \mathcal L^{\mathrm{end}}_{h,k}\cap\Rcal^{[r^\ast_{h,k}]}(h,k)
\]
such that every label in $\mathcal E_{h,k}$ retains a representative from
part~\textup{(i)} and
\begin{equation}\label{eq:dominant-rank-layer-count}
 |\mathcal E_{h,k}|\geq h^{k-1}
 \exp\!\left(-O_k\!\left(\frac{\log h}{\log\log h}\right)\right).
\end{equation}
For fixed $h$, the aggregate realization count
\[
 F^{(k)}_{h,\mathcal E}(q)
 =\#\{A\in\tbinom{[q]}k:|hA|\in\mathcal E_{h,k}\}
\]
is an exact quasipolynomial in $q-1$ of degree $k-r^\ast_{h,k}$.  If $A_q$
is uniform in $\binom{[q]}k$, then
\begin{equation}\label{eq:dominant-rank-layer-rarity}
 \mathbb P\bigl(|hA_q|\in\mathcal E_{h,k}\bigr)
 =\Theta_{h,k}(q^{-r^\ast_{h,k}})
 \qquad(q\to\infty),
\end{equation}
and, for every $n\in\mathcal E_{h,k}$, the normalized shape locus of $n$
has dimension $k-2-r^\ast_{h,k}$.
Consequently
\begin{equation}\label{eq:dominant-intrinsic-rank-exponent}
 \lim_{h\to\infty}
 \frac{\log\max_{2\leq r\leq k-2}|\Rcal^{[r]}(h,k)|}{\log h}=k-1.
\end{equation}
When $k=4$, necessarily $r^\ast_{h,4}=2$ and
$|\Rcal^{[2]}(h,4)|=\Theta(h^3)$.

\end{enumerate}
\end{theorem}

Part~\textup{(i)} concerns witnesses; part~\textup{(ii)} concerns labels.
The selected rank may vary with $h$ and need not equal $k-2$.  Together, the
two parts combine full-exponent abundance and maximal-rank rigidity with
intrinsic codimension, compression, and sampling rarity, while keeping
witness rank distinct from minimum label rank.

\paragraph{Dependency map.}
The two branches analyze the same bounded addition-table data.  The analytic
moment estimate and divisor recovery feed the Hilbert--energy construction;
compression and frequency arise in parallel from the finite-observation
geometry.

\enlargethispage{\baselineskip}
\begin{center}
\begingroup
\setlength{\tabcolsep}{0pt}
\setlength{\fboxsep}{4pt}
\setlength{\fboxrule}{0.4pt}
\newcommand{\depbox}[1]{%
 \fcolorbox{black!35}{black!3}{%
  \parbox[c][.72in][c]{.37\textwidth}{\centering\small #1}}}
\newcommand{\joinbox}[1]{%
 \fcolorbox{black!45}{black!5}{%
  \parbox[c][.52in][c]{.64\textwidth}{\centering\small #1}}}
\begin{tabular}{@{}c@{\hspace{.025\textwidth}}c@{}}
 \textbf{Hilbert--energy construction}&\textbf{Finite-observation geometry}\\[4pt]
 \depbox{Four-point Hilbert--Burch\\seed}
 &\depbox{The complete finite table\\$\mathsf O_h(A)$}\\[-1pt]
 $\downarrow$&$\downarrow$\\[-1pt]
 \depbox{Cyclic suspension and prime-coded\\amplification, with cubic energy}
 &\depbox{Hilbert-labelled root-sweep faces,\\filtered compression, and Ehrhart counting}\\[-1pt]
 $\downarrow$&$\downarrow$\\[-1pt]
 \depbox{Full-exponent label abundance\\with maximal-rank ACM witnesses}
 &\depbox{Sharp compression and\\rank--shape--frequency geometry}\\[-1pt]
 \multicolumn{2}{c}{$\searrow\hspace{.36\textwidth}\swarrow$}\\[-2pt]
 \multicolumn{2}{c}{\joinbox{\textbf{Full-exponent spectrum--geometry synthesis}\\[-1pt]
 {\footnotesize Theorem~\ref{thm:maximal-rank-saturation}}}}
\end{tabular}
\endgroup
\end{center}

Cyclic suspension transports the Hilbert profile, relation-birth filtration,
resolution, diameter, and arithmetic label, thereby placing the witnesses
constructed on the left in the strata analyzed on the right;
Part~\ref{part:rigidity} proves this compatibility.  Separately, the
active-rank filtration organizes three exact conclusions: the rank-zero
$B_h$ endpoint and its optimal ruler, rank-one lattice-chain representation
fibres, and the maximal-rank full-exponent family.  The endpoint and collision
theorems give the corresponding exact statements at these distinguished ranks.

The following theorem collects the finite-observation consequences used in
the synthesis.

\begin{theorem}[Finite-observation master theorem]
\label{thm:active-structure-intro}
The following statements hold.
\begin{enumerate}[label=\textup{(\roman*)}]
\item \emph{Filtered compression.}
Let $h\geq1$, $k\geq2$, and let $A$ be an ordered real $k$-set of
active rank $r=r_h(A)$.  With the successive birth degrees defined above,
put $d_A=k-1-r$.  There is an integer set
\[
 \begin{aligned}
 A^*&=\{0=a_0^*<a_1^*<\cdots<a_{k-1}^*\},\\
 a_{k-1}^*
 &\leq d_Ah^{d_A-1}\prod_{i=1}^r p_i(A)
 \leq(k-1-r)h^{k-2}
 \end{aligned}
\]
with the empty product understood as one, and with exactly the same equality
and strict-order pattern among all $j$-term
sums as $A$, simultaneously for every $1\leq j\leq h$.

\item \emph{Optimal universal scale and maximal-rank boundary.}
For all $h\geq1$ and $k\geq2$,
\begin{equation}\label{eq:compression-intro}
\begin{aligned}
 1+\left\lceil\frac{M_{h,k}-1}{h}\right\rceil
 &\leq N(h,k)
 \leq1+\max\left\{(k-2)h^{k-2},\sum_{j=0}^{k-2}h^j\right\},\\
 N(h,k)&=\Theta_k(h^{k-2}).
\end{aligned}
\end{equation}
If $h\geq2$ and $k\geq3$, then for primitive integer $k$-sets,
\begin{equation}\label{eq:maximal-rank-core-intro}
 \max_{r_h(A)=k-2}\ndiam(A)=h^{k-2}
 \qquad(h\geq2,k\geq3).
\end{equation}
More generally, for a primitive integer $k$-set the normalized diameter is
bounded by the product of its first $k-2$ relation degrees.

\item \emph{Rank-stratified type diversity.}
For fixed $k\geq3$ and $0\leq r\leq k-2$,
\begin{equation}\label{eq:rank-diversity-intro}
 |\mathfrak T_{h,k}^{(r)}|
 =\Theta_{k,r}\!\left(h^{(k-1)r}\right)
 \qquad(h\to\infty).
\end{equation}

\item \emph{Type frequencies and shape dimension.}
For fixed $h\geq2$, $k\geq3$, and a fixed weak type $\mathsf T$ of rank
$r$, its number of
realizations in $\binom{[q]}k$ is an exact quasipolynomial in $q-1$ of degree
$k-r$, with positive rational leading coefficient.  After translation and
dilation are removed, its normalized shape locus is a finite union of
relatively open ordered sweep cells, each of dimension $k-2-r$.

\item \emph{Label frequencies and minimum rank.}
For fixed $h\geq2$ and $k\geq3$, and every $n\in\Rcal(h,k)$, the
cardinality frequency $F^{(k)}_{h,n}(q)$ is an exact quasipolynomial in
$q-1$ of degree
\[
 \deg F^{(k)}_{h,n}=k-r_{h,k}(n),
\]
with positive rational leading coefficient.  Thus the possible realization
exponents are exactly controlled by minimum active rank.

\item \emph{Exact rank-one representation fibres.}
Let $h\geq2$, $k\geq3$, and let $A\subset\mathbb Z$ be an ordered
$k$-set with $r_h(A)=1$.  For $n\in hA$, let
\[
 r_{A,h}(n)=\#\{u\in\mathcal U_{h,k}:u\cdot A=n\}.
\]
If its primitive active direction has degree $r$, then,
with $M_{s,k}=0$ for $s<0$,
\begin{equation}\label{eq:rankone-multiplicity-intro}
 \#\{n\in hA:r_{A,h}(n)=j\}
 =M_{h-(j-1)r,k}-2M_{h-jr,k}+M_{h-(j+1)r,k}
 \qquad(j\geq1).
\end{equation}
Thus the same primitive direction that determines the principal Hilbert
label determines the complete representation-multiplicity histogram.
\end{enumerate}
\end{theorem}

\begin{proof}
Parts~\textup{(i)}--\textup{(vi)} follow, respectively, from
Theorem~\ref{thm:compression-master};
Corollary~\ref{cor:optimal-finite-modeling} together with
Theorem~\ref{thm:compression-master};
Theorem~\ref{thm:rank-diversity-frequency};
Corollary~\ref{cor:addition-type-frequency} together with the
translation--dilation normalization in the proof of
Corollary~\ref{cor:hilbert-strata};
Theorem~\ref{thm:rank-frequency}; and
Theorem~\ref{thm:rankone-multiplicity}.
\end{proof}

The last three items form a three-scale counting law.  Sweep-face Ehrhart
theory counts configurations of a fixed type; minimum active rank controls
the aggregate frequency of a Hilbert label; and on the principal skeleton
the primitive direction resolves that label into its individual
representation multiplicities.

Corollary~\ref{cor:sharp-type-compression} further shows that the maximal-rank
boundary is sharp at the level of addition tables: one weak type forces every
primitive realization to have normalized diameter $h^{k-2}$.

Corollary~\ref{cor:hilbert-strata} sharpens item~\textup{(v)} to the four-way
identity \eqref{eq:four-way-rank-identity}: minimum active rank is
simultaneously frequency codimension, exponent-shape codimension, and the
sampling-rarity exponent of the label.

For each witness $A$, the same geometry yields a collision exact sequence
whose dimension formula decomposes the label defect.
Proposition~\ref{prop:collision-exact-sequence} defines
$\operatorname{cnul}_h(A)$ as the kernel dimension of the
collision-to-relation boundary map and, with
$\Delta_h(A)=M_{h,k}-|hA|$, gives
\begin{equation}\label{eq:defect-rank-nullity-intro}
 \Delta_h(A)=r_h(A)+\operatorname{cnul}_h(A).
\end{equation}
Thus every active-rank-$r$ realization satisfies $|hA|\leq M_{h,k}-r$.
For $h\geq2$ and $k\geq3$,
Theorem~\ref{thm:rank-defect-frontier} proves that this bound is attained for
every $0\leq r\leq k-2$.  Consequently the largest label realized at active
rank $r$ is exactly $M_{h,k}-r$.  Active rank fixes this extremal boundary
level, and collision nullity is the additional deficit below it.  At the
largest nonprincipal label the rank statement is intrinsic: for $h\geq2$ and
$k\geq4$,
Corollary~\ref{cor:largest-nonprincipal-value} proves that $M_{h,k}-2$ has
minimum active rank two and is the largest label outside the universal
principal skeleton.

Thus $h^{k-2}$ is both the optimal-order universal compression scale and the
exact intrinsic boundary for maximal active rank.  The filtered factor
\[
 d_Ah^{d_A-1}\prod_i p_i(A)
\]
records the dimension of the translation-normalized sign cone and the
successive degrees at which its equality space appears.  The cone dimension
$d_A$ is one more than the normalized shape dimension and one less than the
frequency degree below.  This is the geometric link between compression and
type-frequency duality.

The compression statement is, in fact, independent of the ambient ordered
group.  For $k\geq2$, every equality and strict comparison among the $j$-term sums,
$1\leq j\leq h$, of an ordered $k$-tuple in an arbitrary linearly ordered
abelian group has an integer shadow satisfying the same filtered bound.
Consequently every such $k$-element subset of a torsion-free abelian group has an
order-$h$ Freiman model in $[0,O_k(h^{k-2})]$, with optimal exponent.  An
algebraic specialization gives a second sharp-order statement.  If
$W_k^{\mathrm{hom}}(h)$ is the least $W$ such that every monomial order
admits some $w\in\mathbb Z_{\geq0}^k$ with $\min_iw_i=0$ and
$\max_iw_i\leq W$ representing all comparisons $x^\alpha\prec x^\beta$ with
$|\alpha|=|\beta|\leq h$, then
\begin{equation}\label{eq:monomial-weight-intro}
 1+h+\cdots+h^{k-2}
 \leq W_k^{\mathrm{hom}}(h)
 \leq(k-1)h^{k-2}.
\end{equation}
Thus the same finite-observation theorem gives the optimal coefficient
exponent for truncated monomial orders, including equality-sensitive
monomial weights and their homogeneous initial-ideal pieces.
Corollary~\ref{cor:finite-profile-transfer} gives a simultaneous exact
universality statement: for each finite cutoff $H$, truncated sumset profiles
in torsion-free groups, product-set profiles of positive integers, and Hilbert
profiles of projective monomial curves are identical.  After an order is
chosen on the finitely generated subgroup, the complete ordered table has an
integral exponent shadow preserving active rank and every relation-birth
degree visible in the truncation, with optimal universal height
$\Theta_k(H^{k-2})$.
At the spectrum level, Corollary~\ref{cor:lattice-transfer} gives the
universality
\[
 \Rcal_G(h,k)=\Rcal(h,k)
\]
for every nontrivial torsion-free abelian group $G$.
For general abelian groups, Corollary~\ref{cor:group-exponent-dichotomy}
combines the exponent law with Nathanson's embedding argument: for fixed
$k\geq4$, every abelian group of unbounded exponent satisfies
$|\Rcal_G(h,k)|=h^{k-1+o_k(1)}$, whereas $mG=0$ implies
$|\Rcal_G(h,k)|\leq m^{k-1}$ uniformly in $h$.

Throughout the four-element discussion put
\[
 M_j=\binom{j+3}{3}\qquad(j\geq0),
\]
 and set $M_j=0$ for $j<0$ whenever shifted indices occur.

\subsection*{The Hilbert--energy construction}

We now describe the construction represented by the left branch of the
dependency map.  It produces witnesses whose sumset cardinalities form the
desired family, together with their Hilbert profiles, relation-birth
filtrations, toric resolutions, and normalized diameters.

Concretely, if $A$ is primitive, $\ell\geq2$ is an integer, $c$ is an
$\ell$-term sum from $A$, and $(c,\ell)=1$, then, for every integer $q\geq0$,
\[
 \Sigma_{\ell,c}A=\ell\mathbin{\cdot}A\cup\{c\},\qquad
 |q\Sigma_{\ell,c}A|
 =\sum_{j=0}^{\min(q,\ell-1)}|(q-j)A|,
\]
where $\ell\mathbin{\cdot}A$ denotes scalar dilation, whereas
$q\Sigma_{\ell,c}A$ and $(q-j)A$ denote fold sumsets.  Its Hilbert series
and $K$-polynomial are multiplied by
$1+z+\cdots+z^{\ell-1}$ and $1-z^\ell$, respectively.  A new relation is
born exactly at degree $\ell$.

The prime-coded suspension amplification theorem,
Theorem~\ref{thm:prime-filter-amplification}, starts from a shift-stable
chart and admissible prime tuples.  If $N_h$ is the minimum retained
parameter-set size and $E_h$ the maximum collision energy over those tuples,
it produces at least
\[
 \left(\frac h{\log h}\right)^t\frac{N_h^2}{E_h}
\]
labels after $t$ visible prime suspensions.  In exponent notation, a
prime-energy-stable chart with $N_h=h^{a-o(1)}$ and
$E_h=h^{\mu+o(1)}$ obeys
\begin{equation}\label{eq:amplification-law-intro}
 \boxed{\quad 2a-\mu\ \longmapsto\ 2a-\mu+t\quad}.
\end{equation}
 Thus an endpoint seed retains the full ambient exponent at every fixed
 suspension depth.

The proof of Theorem~\ref{thm:main} uses an explicit four-parameter family
whose Hilbert function separates into shifted tetrahedral terms.  For positive
integers $p,q,r,s$
satisfying
\[
 1\leq q<p,\qquad 1\leq r<p,\qquad (p,q)=1,
 \qquad q+s\geq p-r+1,
\]
and $A=\{0,q,p,ps+qr\}$, one has
\begin{equation}\label{eq:defect-intro}
 |hA|=M_h-M_{h-p}-M_{h-q-s}-M_{h-r-s}
       +M_{h-p-s}+M_{h-q-r-s}.
\end{equation}

The $q=1$ slice is due independently to Chen and Yang
\cite[Corollary~1]{ChenYang2026}; the coprime $q\asymp h$ extension drives
the density argument here.

The precise algebraic statement is Theorem~\ref{thm:family}.  The associated toric
ideal has three binomial generators, an L-shaped Gr\"obner degeneration,
and a Hilbert--Burch resolution; the point here is its additive use.  With
\[
 s=u-v,\qquad p=h-u,\qquad q=h-b-s,\qquad r=h-c-s,
\]
and four variables in fixed proportional intervals, \eqref{eq:defect-intro}
becomes
\[
 M_h-|hA|=M_u-M_v+M_b+M_c.
\]
There are $\gg h^4$ primitive parameters in the box.  Vaughan's sharp cubic
eighth moment, Lemma~\ref{lem:cubiceighth}, gives an $O(h^5)$ collision bound
for this map, exactly the strength required for $\gg h^3$ distinct values.

Lemma~\ref{lem:factorial-box-divisibility} makes the product of the prime
degrees a recoverable arithmetic label.  For the endpoint theorem, apply this
general principle to the same four-point seed with $t=k-4$ cyclic suspensions of
distinct prime degrees
$\ell_i\asymp_k h$, chosen in disjoint proportional intervals, and put
$R=\prod_i\ell_i$.  The residue-layer formula sums every shifted
tetrahedral term over a box of size $R$.  For every residual $x$ occurring
here, all shifted indices below are nonnegative.  Writing
$S_1=\sum_i(\ell_i-1)$, $S_2=\sum_i(\ell_i^2-1)$, and
$Y_x=2x+4-S_1$, the exact centered identity, with factorial normalization
$48=2^3\,3!$, is
\[
 48\sum_{0\leq j_i<\ell_i}M_{x-j_1-\cdots-j_t}
 =R\bigl(Y_x^3+(S_2-4)Y_x\bigr).
\]
Thus the apparent degree-$k-1$ collision problem remains a four-variable
cubic problem with a moving linear coefficient.  The coefficient-uniform
eighth-moment estimate in Lemma~\ref{lem:uniform-moving-cubic} gives
$h^{3-o(1)}$ labels for each prime tuple.  The
divisibility $R\mid48|hA|$ recovers the tuple up to bounded multiplicity, and
the $(h/\log h)^{k-4}$ choices of prime degrees supply the remaining
exponent.  This is the Hilbert--energy proof of
Theorem~\ref{thm:endpoint-spectrum}.

\subsection*{Four-point rigidity, distribution, and transfers}

The four-point case yields exact refinements.  Theorem~\ref{thm:spectral-rank}
partitions the spectrum into one rank-zero value, $h-1$ rank-one values, and
$\Theta(h^3)$ minimum-rank-two values, with $\gg h^3$ of the last kind already
in the fixed window of Theorem~\ref{thm:main}.  Their fixed-label frequencies
are quadratic quasipolynomials and their normalized shape loci are
zero-dimensional.  Every minimum-rank-two label has all its realizations in
the quadratic diameter core, whereas the rank-zero and rank-one labels are
exactly those admitting arbitrarily large normalized diameter.  Although the
rank-zero endpoint $M_h$ has arbitrarily large models, its least normalized
diameter is
\[
 \binom{h+2}{2}+\mathbf 1_{\{2\nmid h\}}
 \qquad(h\geq2),
\]
proving the conjectural 2013 four-mark formula recorded in OEIS A227589.  The largest
rank-two label $M_h-2$ has least normalized diameter
$1+\binom{h+1}{2}$.

Theorem~\ref{thm:tail} proves
\[
 \{|hA|:\ |A|=4,\ \ndiam A>h^2\}
 =\{M_h\}\cup\{M_h-M_{h-j}:2\leq j\leq h\}.
\]
Together with the witnesses in \eqref{eq:smallwitness}, this yields
Corollary~\ref{cor:core}: every nonmaximum has a representative with
$\ndiam A\leq h^2$.
Theorem~\ref{thm:defecttwodiameter} further proves
\[
 \min_{|A|=4,\ |hA|=M_h-2}\ndiam A
 =1+\binom{h+1}{2}.
\]
Every lattice realization of the latter value is collinear, yielding the
stated failure of Nathanson's root-compression proposal in every lattice
dimension $n\geq2$.  These refinements are Theorem~\ref{thm:tail},
Corollary~\ref{cor:core}, Theorem~\ref{thm:optimal-four-mark},
Theorem~\ref{thm:defecttwodiameter}, and Corollary~\ref{cor:defecttwolattice}.

 Theorem~\ref{thm:active-structure-intro}\textup{(iii)} also answers
 O'Bryant's Question~6 on type enumeration up to constant factors in the
 fixed-$k$, $h\to\infty$ regime.  Ordered
and generic $B_h$ types and the transferred
multiplication-table types have the same total exponent
$(k-1)(k-2)$.  At four points the type-$A$ coordinate dictionary yields the
diagonal order-six Farey vertex and region bounds.  For fixed $h$, the same
faces give exact popularity laws: the maximum has frequency order $q^k$,
the other principal-skeleton values have order $q^{k-1}$, and every
remaining cardinality has order at most $q^{k-2}$.  Senger proved the coarse
$k=4$ separation \cite[Theorem~2]{Senger2025}; here exact coefficients and
residual-label laws extend to every fixed $k$.  See Sections~
\ref{sec:rank-duality}, \ref{sec:collision-synthesis}, and
\ref{sec:popularity} for the full statements and conditional limit law.
The principal skeleton also admits an exact representation-level refinement.
If an active-rank-one set has primitive relation degree $r$, then that
direction determines not only the label $M_{h,k}-M_{h-r,k}$ but the entire
representation histogram through Theorem~\ref{thm:rankone-multiplicity}.
Combined with the universal $\operatorname{Beta}(k-2,1)$ first-relation law,
this gives an explicit sharp conditional collision probability, equal to
$0.958084\ldots$ at four points, and separates the literal negative answer
to Nathanson's Problem~11 \cite[Problem~11]{NathansonTriangular} from the
universal successive-limit phase law for its conditioned non-$B_h$
formulation.
Corollary~\ref{cor:label-configuration-reversal} makes the resulting measure
reversal explicit: for every fixed $k\geq4$, principal labels are negligible
under uniform sampling of labels as $h\to\infty$; by contrast, for every fixed
$h\geq2$ and $k\geq3$, their realizations account for asymptotically almost
every configuration as $q\to\infty$.

\subsection*{Four regimes and relation to prior work}

\paragraph{Adjacent asymptotic regimes.}
For a fixed finite set $A$, classical Khovanskii theory studies the eventual
polynomial behaviour of $|hA|$ as $h\to\infty$; Nathanson--Ruzsa establish
the corresponding semigroup and multivariate polynomiality theorem
\cite{NathansonRuzsa2002}.  In the finite-window profile problem posed by
Nathanson \cite{NathansonInverse2026}, Kravitz realizes arbitrary relative
orders among several sequences, while Fox--Kravitz--Zhang realize arbitrary prescribed differences
$|hA|-|hB|$ for $1\leq h\leq H$ and begin the associated cardinality and
diameter optimization \cite{Kravitz2025,FoxKravitzZhang2026}.  Their sets may
grow with the observation window and their objective is comparison among
profiles.  Rajagopal's exact theorem instead fixes $h$ and takes $k$ large
\cite{Rajagopal2026}.  The present problem fixes $k$, lets $h\to\infty$, and
counts the attainable values of a single degree-$h$ label.  Thus these four
directions share the same iterated-sumset profile but answer genuinely
different questions.

Nathanson initiated the systematic study of the last spectrum and gave
explicit constructions \cite{NathansonProblems,NathansonExplicit2025,
NathansonTriangular}.  His exact three-point formula is
\[
 \Rcal(h,3)=
 \left\{\binom{h+2}{2}-\binom{\ell}{2}:1\leq\ell\leq h\right\}.
\]
Tang--Xing and Mohan--Pandey obtained extended inverse theorems near the lower
edge, while Schinina proved the first forbidden interval for torsion-free
abelian groups
\cite{TangXing2021,MohanPandey2023,Schinina2026}.  Theorem~
\ref{thm:endpoint-spectrum} determines the polynomial exponent with $k$ fixed
and $h\to\infty$, while Theorem~\ref{thm:main} gives positive density at the
four-point transition.  Jung,
O'Bryant, and Chen--Yang supply recent local, relation-lattice, and
explicit-family results adjacent to this regime
\cite{Jung2026,OBryant2025,ChenYang2026}.

\paragraph{Hilbert and finite-observation antecedents.}
For a fixed finite set $A$, the connection between iterated sumsets and
Hilbert functions goes back to Khovanskii
\cite{Khovanskii1992,Khovanskii1995}.  Eliahou--Mazumdar encode the full
sequence $(|hA|)_{h\geq0}$ in a standard graded algebra, while
Colarte-G\'omez--Elias--Mir\'o-Roig and Elias realize these cardinalities via
monomial projections of Veronese varieties and projective monomial curves
\cite{EliahouMazumdar2022,EliahouMazumdar2025,
ColarteGomezEliasMiroRoig2023,Elias2022}.  Gimenez--Gonz\'alez-S\'anchez
further relate eventual sumset structure to Castelnuovo--Mumford regularity
and homogeneous-semigroup data of the corresponding curve
\cite{GimenezGonzalezSanchez2023}.  The later work of
Eliahou--Mazumdar proves that the Macaulay upper bounds are best possible by
constructing extremal examples in suitable abelian semigroups using monomial
ideals and Gr\"obner deformation.  For a finite $A\subset\mathbb Z$ with
$m=|A|\geq2$, Elias's identity is made explicit as follows.  After translating, write
$A=\{0=a_0<a_1<\cdots<a_{m-1}=D\}$ and, over an arbitrary field $K$, put
\[
 \begin{aligned}
 S_A&=K[x_0,\ldots,x_{m-1}],&
 \varphi_A&:S_A\longrightarrow K[u,v],\\
 \varphi_A(x_i)&=u^{D-a_i}v^{a_i},&
 I_A&=\ker\varphi_A,
 \end{aligned}
\]
and set $R_A=S_A/I_A$, $C_A=\operatorname{Proj}R_A$, and
$H_{C_A}(q)=\dim_K(R_A)_q$.  Distinct image monomials are indexed by
$qA$, so
\begin{equation}\label{eq:eliasbridge}
 |hA|=H_{C_A}(h).
\end{equation}
Those works use Hilbert data to study fixed-set growth or to construct
extremizers.  Here, by contrast, $|A|=k$ is fixed and $A$ varies: how many
degree-$h$ Hilbert values occur?  Theorem~\ref{thm:endpoint-spectrum} proves
that the full ambient exponent is already attained within one explicit
arithmetically Cohen--Macaulay type-two family.

Fox--Kravitz--Zhang also study a different efficiency problem: the least
cardinality or diameter of pairs realizing arbitrary prescribed finite-window
sign patterns.  Our parameter $N(h,k)$ fixes the cardinality $k$ and instead
asks for one interval containing a representative of every degree-$h$ label in
$\Rcal(h,k)$.  Thus the two compression questions are complementary:
finite-window profile complexity versus universal fixed-cardinality spectrum
compression.  Theorem~\ref{thm:active-structure-intro}\textup{(ii)} determines
the latter at the optimal order $\Theta_k(h^{k-2})$.

There is also a precise methodological point of contact.
Fox--Kravitz--Zhang use bounded relation subspaces to encode finite profiles
\cite[Lemma~8]{FoxKravitzZhang2026}.  Here the signed observation $\mathsf O_h(A)$
records the complete equality-and-order pattern of one fixed-cardinality set
through degree $h$, and its sweep face is used to extract active rank, filtered
compression, and exact lattice-point frequencies.

The compression and polyhedral components refine established qualitative
frameworks.  At the additive equality level, finite-profile transfer from
torsion-free groups to $\mathbb Z$ is the classical finite Freiman-model
principle \cite[Part~II, Lemma~2.3.4]{GeroldingerRuzsa2009}.  Nathanson proved
qualitative integer realizability \cite{NathansonMSTD2018} and later the bound
$N(h,k)<4(8h)^{k-1}$ \cite{NathansonCompression};
Amirkhanyan--Bush--Croot developed short order-preserving Freiman models for
structured subsets \cite{AmirkhanyanBushCroot2018}.  Our
theorem models the whole tuple without a small-doubling hypothesis, preserves
every equality and strict comparison through degree $h$, and sharpens this to
the exact polynomial order $N(h,k)=\Theta_k(h^{k-2})$.  In the complementary
fixed-$h$ regime, Rajagopal's guest appendix to Gowers's post gives
$N(h,k)=O_h(k^{10h^3})$ \cite{GowersRajagopal2026}.  Prior work likewise
determines sweep faces, strict-comparison arrangements, and the exponent of
their unlabelled count
\cite{PadrolPhilippe2024,Vietri2002,BaranyBureauxLund2018}; our refinement
retains ties, relation-birth data, Hilbert labels, and filtered short integral
normals, and determines exact quasipolynomial weak-type frequencies and their
fixed-label aggregates.
O'Bryant's fixed-degree weak-type transfer and Elias's exact monomial-curve
Hilbert identity give two further points of contact
\cite{OBryant2025,Elias2022}.  Corollary~\ref{cor:finite-profile-transfer}
strengthens the classical additive equality transfer by incorporating the
sharp filtered ordered shadow and the multiplicative and Hilbert
correspondences.  Complete truncated profiles transfer simultaneously,
while an integral exponent representative preserves order, active rank, and
every birth visible in the truncation at the optimal universal height
$h^{k-2}$.

\paragraph{The amplification theorem and its antecedents.}
Cyclic suspension is a projective toric extension and homogeneous
one-generator gluing; its formal Hilbert and resolution transforms are
classical \cite{Sahin2011,GimenezSrinivasan2019,GimenezSrinivasan2025}.
The analytic estimates build on the cubic moment and restriction methods of
Vaughan and Hughes--Wooley \cite{Vaughan1986,HughesWooley2022}.  The new
ingredient is a finite-degree quantitative combination of these algebraic
and analytic tools: prime extensions at the observation scale provide a
recoverable divisor label; bounded overlap between prime tuples combines
that label with a coefficient-uniform cubic energy bound; and the resulting
amplification propagates one four-point Cohen--Macaulay chart to every fixed
set size $k\geq4$.  More specialized antecedents for the Farey, popularity,
lattice, and regularity consequences are recorded where those results are
proved.

\subsection*{Organization}

Part~\ref{part:spectrum}, Sections~
\ref{sec:suspension-calculus}--\ref{sec:endpoint-spectrum}, develops the
suspension and collision calculus, the four-point Cohen--Macaulay seed, and
the prime-coded amplification that constructs the endpoint family.
Part~\ref{part:structure}, Sections~
\ref{sec:root-sweep-dictionary}--\ref{sec:rank-duality}, develops the
finite-observation geometry governing active rank, optimal integral models,
shape, type, and exact frequency.  Part~\ref{part:rigidity} combines these
analyses to prove Theorem~\ref{thm:maximal-rank-saturation}.  Subsequent
sections derive the four-point rigidity, defect-two, popularity, limit-law,
and transfer consequences.

\part{Spectrum abundance and prime-coded amplification}
\label{part:spectrum}

This part constructs a chosen endpoint witness for each resulting label.
The witnesses carry bounded addition tables, relation-birth filtrations,
normalized diameters, and toric resolutions; Part~\ref{part:structure}
analyzes their table, rank, birth, and compression geometry.

\section{Cyclic suspension and collision calculus}
\label{sec:suspension-calculus}

We isolate the operation that propagates the four-point endpoint chart.
Algebraically it is a projective toric extension in the sense of \c{S}ahin
and a particularly explicit homogeneous one-generator gluing
\cite{Sahin2011,GimenezSrinivasan2019,GimenezSrinivasan2025}.  The extension
and its formal Hilbert and resolution transforms are classical.  The point
here is the exact finite-observation calculus and its later use as a
prime-coded spectrum amplifier.

Fix an arbitrary coefficient field $K$ for the homological statements below,
including those in the later amplification results.  These assertions therefore
hold separately over every field; the additive, Hilbert-function, relation, and
diameter statements are field-free.

\subsection{Exact cyclic suspension}

Let
\[
 A=\{0=a_0<a_1<\cdots<a_{m-1}\}\subset\mathbb Z_{\geq0}
\]
be primitive, with $m\geq2$.  Fix an integer $\ell\geq2$ and a representation
\begin{equation}\label{eq:suspension-representation}
 c=\sum_{i=0}^{m-1}\lambda_i a_i,\qquad
 \lambda_i\in\mathbb Z_{\geq0},\qquad
 \sum_i\lambda_i=\ell,\qquad (c,\ell)=1.
\end{equation}
Define the \emph{cyclic suspension}
\[
 \Sigma_{\ell,c}A=\ell\mathbin{\cdot}A\cup\{c\}.
\]
Every point of $\ell\mathbin{\cdot}A$ is divisible by $\ell$, whereas
$(c,\ell)=1$; thus the new point is distinct, and the same coprimality
makes the suspension primitive.  If $\ell$ is prime, a choice always exists:
some nonzero $a_j$ is not divisible by $\ell$, and one may take
$c=a_j$, $\lambda_j=1$, and $\lambda_0=\ell-1$.

\begin{theorem}[Exact cyclic-suspension calculus]
\label{thm:cyclic-suspension}
Put $C=\Sigma_{\ell,c}A$ and $D=\max A$.  For
$q\in\mathbb Z_{\geq0}$, let $L_q(A)$ be the real span of the homogeneous
relation vectors of relation degree at most $q$, and put
$r_q(A)=\dim L_q(A)$.  Then:
\begin{enumerate}[label=\textup{(\roman*)}]
\item For every $q\in\mathbb Z_{\geq0}$,
\begin{equation}\label{eq:suspension-profile}
 |qC|=\sum_{j=0}^{\min(q,\ell-1)}|(q-j)A|.
\end{equation}
Equivalently, for $X\in\{A,C\}$ put
\[
 \mathcal H_X(z)=\sum_{q\geq0}|qX|z^q.
\]
Then
\begin{equation}\label{eq:suspension-Hilbert-series}
 \mathcal H_C(z)=\mathcal H_A(z)[\ell]_z,\qquad
 [\ell]_z=1+z+\cdots+z^{\ell-1}.
\end{equation}

\item Let
\[
 \begin{aligned}
 S_A&=K[x_0,\ldots,x_{m-1}],&
 \psi_A&:S_A\longrightarrow K[u,v],\\
 \psi_A(x_i)&=u^{D-a_i}v^{a_i},&
 I_A&=\ker\psi_A,\qquad R_A=S_A/I_A.
 \end{aligned}
\]
be the standard-graded projective semigroup ring of $A$.  Put
$S_C=S_A[y]$, and let $I_C$ be the kernel of the homogeneous toric map
\[
 x_i\longmapsto u^{\ell(D-a_i)}v^{\ell a_i},\qquad
 y\longmapsto u^{\ell D-c}v^c,
\]
so that $R_C=S_C/I_C$.  Then
\begin{equation}\label{eq:suspension-toric}
 R_C\cong
 R_A[y]\big/\bigl(y^\ell-x_0^{\lambda_0}\cdots
 x_{m-1}^{\lambda_{m-1}}\bigr).
\end{equation}
The displayed polynomial is homogeneous and regular.  In particular, if
$R_A$ is Cohen--Macaulay, then so is $R_C$.

\item Define the ambient $K$-polynomials by
\[
 \mathcal K_A(z)=(1-z)^m\mathcal H_A(z),\qquad
 \mathcal K_C(z)=(1-z)^{m+1}\mathcal H_C(z).
\]
Then
\begin{equation}\label{eq:suspension-K}
 \mathcal K_C(z)=\mathcal K_A(z)(1-z^\ell).
\end{equation}
If $\beta^A_{i,j}$ and $\beta^C_{i,j}$ are the standard-graded Betti
numbers over $S_A$ and $S_C$, respectively, extended by zero outside their
natural homological and internal degree ranges, then
\begin{equation}\label{eq:suspension-Betti}
 \beta^C_{i,j}
 =\beta^A_{i,j}+\beta^A_{i-1,j-\ell}.
\end{equation}
Equivalently, with $\xi$ an indeterminate recording homological degree, put
\[
 \mathcal B_A(\xi,z)=\sum_{i,j}\beta^A_{i,j}\xi^iz^j,
 \qquad
 \mathcal B_C(\xi,z)=\sum_{i,j}\beta^C_{i,j}\xi^iz^j.
\]
Then
\[
 \mathcal B_C(\xi,z)=\mathcal B_A(\xi,z)(1+\xi z^\ell).
\]

\item Label the coordinates of $C$ by the tuple
$(\ell a_0,\ldots,\ell a_{m-1},c)$.  They need not be in increasing order;
this changes $L_q(C)$ only by a coordinate permutation.  Put
\[
 w_{\ell,c}=(\lambda_0,\ldots,\lambda_{m-1},-\ell).
\]
Under the old-coordinate embedding,
\begin{equation}\label{eq:suspension-filter}
 L_q(C)=
 \begin{cases}
  L_q(A),&q<\ell,\\[2mm]
  L_q(A)\oplus\mathbb Rw_{\ell,c},&q\geq\ell.
 \end{cases}
\end{equation}
Hence
\begin{equation}\label{eq:suspension-rank}
 r_q(C)=r_q(A)+
 \begin{cases}
  0,&q<\ell,\\
  1,&q\geq\ell,
 \end{cases}
\end{equation}
and the multiset of active-relation birth degrees is the old multiset with
$\ell$ adjoined.

\item Define the order-$q$ Freiman dimension by
\[
 \operatorname{fdim}_q(A)=m-1-r_q(A).
\]
Then
\begin{equation}\label{eq:suspension-Freiman-dimension}
 \operatorname{fdim}_q(C)=
 \begin{cases}
  \operatorname{fdim}_q(A)+1,&q<\ell,\\
  \operatorname{fdim}_q(A),&q\geq\ell.
 \end{cases}
\end{equation}

\item One has
\begin{equation}\label{eq:suspension-diameter}
 \max C=\ell D,\qquad
 \ndiam(C)=\ell\ndiam(A).
\end{equation}
\end{enumerate}
\end{theorem}

\begin{proof}
Every $q$-sum from $C$ uses some number of copies of $c$.
Replace each block of $\ell$ copies by the $\ell$ scaled old points prescribed
by \eqref{eq:suspension-representation}.  The residual layer with
$0\leq j<\ell$ copies of $c$ is
\[
 jc+\ell\mathbin{\cdot}((q-j)A).
\]
Distinct layers are disjoint modulo $\ell$, because $(c,\ell)=1$.
This proves \eqref{eq:suspension-profile} and
\eqref{eq:suspension-Hilbert-series}.

For the toric presentation, reduce a polynomial modulo
\[
 f=y^\ell-x_0^{\lambda_0}\cdots x_{m-1}^{\lambda_{m-1}}
\]
to $\sum_{e=0}^{\ell-1}y^ef_e(x)$.  Under the toric map, distinct $e$ lie
in distinct residue classes $ec\pmod\ell$; within one residue class the
remaining kernel is precisely $I_A$.  This proves
\eqref{eq:suspension-toric}.
The polynomial $f$ is monic in $y$, hence is a nonzerodivisor over $R_A[y]$.
Equations \eqref{eq:suspension-K} and \eqref{eq:suspension-Betti} follow from
the minimal mapping cone of multiplication by $f$.  The cone is minimal
because all its entries lie in the homogeneous maximal ideal.

It remains to check the filtration.  In a degree-$d\leq q$ kernel binomial
\[
 x^u y^a-x^v y^b,
\]
the residue argument gives $a\equiv b\pmod\ell$.  Write
$a=\alpha\ell+e$ and $b=\beta\ell+e$, where $0\leq e<\ell$.
Normalizing both monomials modulo $f$ and cancelling $y^e$ leaves the old
binomial
\[
 x^{u+\alpha\lambda}-x^{v+\beta\lambda}
\]
of degree $d-e\leq q$.  Its relation vector differs from the original one
by a multiple of $w_{\ell,c}$.  If $q<\ell$, the two original
$y$-exponents are below $\ell$, so they are equal and no new component
occurs.  This proves \eqref{eq:suspension-filter}; the rank and Freiman
dimension formulas follow.

Finally, \eqref{eq:suspension-representation} gives $c\leq\ell D$, while
$\ell D\in C$.  Both sets are primitive and contain zero, which proves
\eqref{eq:suspension-diameter}.
\end{proof}

\begin{remark}[Visible and free regimes]
\label{rem:suspension-visible-free}
At observation order $q$, a suspension with $\ell\leq q$ is
\emph{visible}: its new relation has been born and the order-$q$ Freiman
dimension is unchanged.  A suspension with $\ell>q$ is \emph{free at order
$q$}: it adds one Freiman dimension and no active relation.  This dichotomy
is exact, not asymptotic.
\end{remark}

\subsection{Collision directions and collision nullity}

For an ordered finite integer set
$A=\{a_0<a_1<\cdots<a_{m-1}\}$ and
$q\in\mathbb Z_{\geq0}$, put
\[
 \mathcal U_{q,m}
 =\left\{u\in\mathbb Z_{\geq0}^m:\sum_i u_i=q\right\},
 \qquad
 M_{q,m}=|\mathcal U_{q,m}|=\binom{q+m-1}{m-1}.
\]
Define
\[
 \Pi_{A,q}:\mathbb R^{\mathcal U_{q,m}}\longrightarrow\mathbb R^{qA},
 \qquad e_u\longmapsto e_{u\cdot A},
 \qquad
 \mathcal C_q(A)=\ker\Pi_{A,q}.
\]
Thus $\mathcal C_q(A)$ is the degree-$q$ collision space in the monomial
basis.  Translating $A$ merely translates every element of $qA$ by the same
integer, so $\mathcal C_q(A)$ and the nullity defined below are translation
invariant (with the induced ordering).

\begin{proposition}[Collision exact sequence]
\label{prop:collision-exact-sequence}
The map
\[
 \partial_q:\mathcal C_q(A)\longrightarrow\mathbb R^m,
 \qquad
 \partial_q\!\left(\sum_u\gamma_ue_u\right)=\sum_u\gamma_u u,
\]
has image $L_q(A)$, contained in the coordinate-sum-zero subspace.  With
\[
 \mathcal N_q(A)=\ker\partial_q,
 \qquad \operatorname{cnul}_q(A)=\dim\mathcal N_q(A),
\]
there is an exact sequence
\begin{equation}\label{eq:collision-exact-sequence}
 0\longrightarrow\mathcal N_q(A)
 \longrightarrow\mathcal C_q(A)
 \xrightarrow{\partial_q}L_q(A)
 \longrightarrow0.
\end{equation}
Consequently, if $\Delta_q(A)=M_{q,m}-|qA|$, then
\begin{equation}\label{eq:defect-rank-nullity}
 \boxed{\Delta_q(A)=r_q(A)+\operatorname{cnul}_q(A).}
\end{equation}
For $C=\Sigma_{\ell,c}A$,
\begin{equation}\label{eq:suspension-defect-transform}
 \Delta_q(C)
 =\sum_{j=0}^{\min(q,\ell-1)}\Delta_{q-j}(A)
  +\sum_{j=\ell}^{q}M_{q-j,m},
\end{equation}
where the second sum is empty for $q<\ell$.
\end{proposition}

\begin{proof}
Let $\zeta=\sum_u\gamma_ue_u\in\mathcal C_q(A)$.  The coefficient sum is
zero inside each fibre of $u\mapsto u\cdot A$, and hence
\[
 \sum_i(\partial_q\zeta)_i=q\sum_u\gamma_u=0,
 \qquad
 (\partial_q\zeta)\cdot A
 =\sum_{s\in qA}s\!\sum_{u\cdot A=s}\gamma_u=0.
\]
Moreover, $\mathcal C_q(A)$ is spanned by the fibre differences $e_u-e_v$
with $u\cdot A=v\cdot A$, and
$\partial_q(e_u-e_v)=u-v\in L_q(A)$.  Thus
$\partial_q(\mathcal C_q(A))\subseteq L_q(A)$.

Conversely, let $z$ be an integral homogeneous relation of relation degree
$\rho(z)\leq q$.  Choose $w\in\mathbb Z_{\geq0}^m$ with
$\sum_iw_i=q-\rho(z)$.  Then
$z^++w,z^-+w\in\mathcal U_{q,m}$ lie in the same fibre, and
\[
 \partial_q\bigl(e_{z^++w}-e_{z^-+w}\bigr)=z.
\]
Since such integral relations span $L_q(A)$, one has
$\operatorname{im}\partial_q=L_q(A)$, proving
\eqref{eq:collision-exact-sequence}.  Now
\[
 \dim\mathcal C_q(A)=M_{q,m}-|qA|,
\]
so dimensions give \eqref{eq:defect-rank-nullity}.

For the last formula, insert \eqref{eq:suspension-profile} into
$\Delta_q(C)=M_{q,m+1}-|qC|$ and use the hockey-stick identity
\[
 M_{q,m+1}=\sum_{j=0}^{q}M_{q-j,m}.
\]
\end{proof}

\begin{remark}\label{rem:collision-nullity-scope}
The quantity $r_q(A)$ counts independent equality directions, while
$\operatorname{cnul}_q(A)$ counts linear redundancy among the degree-$q$
collision equations.  Formula \eqref{eq:suspension-defect-transform},
together with \eqref{eq:suspension-rank}, determines
$\operatorname{cnul}_q(C)$.  In general, it is not obtained by simply
convolving the old nullities.
\end{remark}

\begin{corollary}[Iterated suspension]
\label{cor:iterated-suspension}
Let $t\in\mathbb Z_{\geq0}$ and put $A^{(0)}=A$.  For each $1\leq i\leq t$, choose an
integer $\ell_i\geq2$ and a specified $c_i$ having a length-$\ell_i$
representation by points of $A^{(i-1)}$ with $\gcd(c_i,\ell_i)=1$, and set
$A^{(i)}=\Sigma_{\ell_i,c_i}A^{(i-1)}$.  Put
$R=\prod_{i=1}^t\ell_i$, with the empty product equal to one.  Then, for
$q\in\mathbb Z_{\geq0}$,
\begin{align}
 \mathcal H_{A^{(t)}}(z)
 &=\mathcal H_A(z)\prod_{i=1}^t[\ell_i]_z,
 \label{eq:iterated-suspension-H}\\
 |qA^{(t)}|
 &=\sum_{\substack{0\leq j_i<\ell_i\\
                    j_1+\cdots+j_t\leq q}}
   |(q-j_1-\cdots-j_t)A|,
 \label{eq:iterated-suspension-profile}\\
 \mathcal K_{A^{(t)}}(z)
 &=\mathcal K_A(z)\prod_{i=1}^t(1-z^{\ell_i}),
 \label{eq:iterated-suspension-K}\\
 \mathcal B_{A^{(t)}}(\xi,z)
 &=\mathcal B_A(\xi,z)\prod_{i=1}^t(1+\xi z^{\ell_i}),
 \label{eq:iterated-suspension-Betti}\\
 \ndiam(A^{(t)})&=R\ndiam(A).
 \label{eq:iterated-suspension-diameter}
\end{align}
If every $\ell_i\leq q$, then
\[
 r_q(A^{(t)})=r_q(A)+t,\qquad
 \operatorname{fdim}_q(A^{(t)})=\operatorname{fdim}_q(A).
\]
If the starting projective coordinate ring is ACM, then the ACM property
and Cohen--Macaulay type are preserved.
\end{corollary}

\begin{proof}
Apply Theorem~\ref{thm:cyclic-suspension} successively to
$A^{(0)},\ldots,A^{(t-1)}$; its profile, Hilbert-series,
$K$-polynomial, Betti, and diameter formulas iterate to the displayed
identities.  If every $\ell_i\leq q$, its rank and Freiman-dimension
formulas apply at each step.  Its regular-extension statement preserves
the ACM property, while the Betti formula preserves the rank of the final free
module and hence the Cohen--Macaulay type.  The case $t=0$ is covered by
the stated empty-product conventions.
\end{proof}

\begin{corollary}[Critical-core conservation]
\label{cor:suspension-critical-core}
Suppose that $A$ is a primitive $m$-set of maximal active rank $m-2$ at an
integer order $q\geq0$.  Let
\[
 p_1(A)\leq\cdots\leq p_{m-2}(A)
\]
be its successive relation-birth degrees.  If $C=\Sigma_{\ell,c}A$ with
$\ell\leq q$, then $C$ has maximal active rank $m-1$, its birth-degree
multiset is
\[
 \{p_1(C),\ldots,p_{m-1}(C)\}
 =\{p_1(A),\ldots,p_{m-2}(A),\ell\},
\]
and
\begin{equation}\label{eq:critical-core-conservation}
 \frac{\ndiam(C)}{\prod_{i=1}^{m-1}p_i(C)}
 =
 \frac{\ndiam(A)}{\prod_{i=1}^{m-2}p_i(A)}.
\end{equation}
The same ratio is preserved under any tower of visible suspensions.
\end{corollary}

\begin{proof}
The direct-sum filtration in \eqref{eq:suspension-filter} appends exactly one
relation direction born at degree $\ell$; it neither removes nor changes an
old birth.  Equation~\eqref{eq:suspension-diameter} multiplies the normalized
diameter by the same factor $\ell$.  Taking products proves
\eqref{eq:critical-core-conservation}, and iteration gives the final claim.
\end{proof}

\begin{corollary}[The three-point skeleton]
\label{cor:three-point-suspension-skeleton}
For every $h\geq2$,
\begin{equation}\label{eq:three-point-spectrum}
 \Rcal(h,3)
 =\left\{\binom{h+2}{2}\right\}
 \cup
 \left\{
  \binom{h+2}{2}-\binom{h-\ell+2}{2}:2\leq\ell\leq h
 \right\}.
\end{equation}
Every value in \eqref{eq:three-point-spectrum} is realized by the one-step
suspension
\[
 \Sigma_{\ell,1}\{0,1\}=\{0,1,\ell\};
\]
degrees $\ell>h$ all realize the maximum $\binom{h+2}{2}$.
\end{corollary}

\begin{proof}
The projective toric ideal of any three-point set is a height-one homogeneous
prime in a polynomial ring.  It is therefore principal, generated by a
binomial of some degree $\ell\geq2$, and its Hilbert series is
$(1-z^\ell)/(1-z)^3$.  The coefficient of $z^h$ is the right-hand side of
\eqref{eq:three-point-spectrum} when $\ell\leq h$, and is the maximum when
$\ell>h$.  Conversely, the displayed suspensions realize every such degree.
\end{proof}

\begin{remark}[Homological transition]
\label{rem:suspension-homological-transition}
Starting from $\{0,1\}$, iterated cyclic suspensions form
complete-intersection towers: their minimal resolutions are the Koszul
complexes on the successive suspension equations.  This accounts for the
entire three-point cardinality spectrum, but is not a classification of all
higher-cardinality sets.  At four points, the three-generator
codimension-two Hilbert--Burch seed used here is genuinely
non-complete-intersection
and has projective Cohen--Macaulay type two.  Further suspensions tensor its
resolution with the two-term Koszul factors
$1+\xi z^{\ell_i}$ and preserve type two.  Thus the endpoint propagation
retains, rather than erases, the first non-complete-intersection homological
mechanism.
\end{remark}

\section{Four-point relation and Hilbert preliminaries}\label{sec:relations}

\subsection{Homogeneous relations}

Let
\[
 A=\{0,a,b,d\},\qquad 0<a<b<d,\qquad \gcd(a,b,d)=1.
\]
Its homogeneous relation lattice is
\begin{equation}\label{eq:lattice}
 L_A=\left\{z=(z_0,z_1,z_2,z_3)\in\Z^4:
 \sum_{i=0}^3z_i=0,\quad az_1+bz_2+dz_3=0\right\}.
\end{equation}
It has rank two.  For $z\in L_A$, define its relation degree by
\begin{equation}\label{eq:rho}
 \rho(z)=\sum_i z_i^+=\sum_i z_i^-=\frac12\|z\|_1.
\end{equation}
We call $z$ $h$-active if $\rho(z)\leq h$.
We shall repeatedly use the following elementary consequence of the zero-sum
condition: if $0\leq c_i\leq C$, then
\begin{equation}\label{eq:weighteddegree}
 \left|\sum_i c_i z_i\right|\leq C\rho(z).
\end{equation}
Indeed, each of $\sum_i c_i z_i^+$ and $\sum_i c_i z_i^-$ lies between
$0$ and $C\rho(z)$.

A collision between two degree-$h$ monomials in four variables cancels to
an $h$-active relation.  Conversely, every active relation can be padded
to degree $h$ and gives a collision.  This elementary correspondence is
the lattice form of the toric ideal used below.

\begin{lemma}[Product bound]\label{lem:product}
If $z,w\in L_A$ are linearly independent, then
\begin{equation}\label{eq:product}
 d\leq \rho(z)\rho(w).
\end{equation}
Consequently, for every integer $h\geq1$, if $d>h^2$, the $h$-active
relations span a space of rank at most one.
\end{lemma}

\begin{proof}
Drop the zeroth coordinate and write
\[
 \bar z=(z_1,z_2,z_3),\qquad \bar w=(w_1,w_2,w_3).
\]
Both vectors are perpendicular to the primitive vector $(a,b,d)$, so
\[
 \bar z\times\bar w=m(a,b,d)
\]
for a nonzero integer $m$.  Looking at the third coordinate gives
\[
 d\leq |m|d=|z_1w_2-z_2w_1|
 \leq\rho(z)\rho(w).
\]
For completeness, write $z$ as a sum of $\rho(z)$ elementary transfers
$e_i-e_j$, and similarly write $w$ as a sum of $\rho(w)$ such transfers.
After the zeroth and third coordinates are deleted, every $2\times2$
determinant of two such transfers is $0$ or $\pm1$.  Bilinearity therefore
gives the last inequality.  This
proves \eqref{eq:product}; the final assertion is immediate.
\end{proof}

\begin{remark}
Lemma~\ref{lem:product} also has the standard projective B\'{e}zout
interpretation familiar from toric geometry; see, for example,
\cite{Sturmfels1996}.
The monomial curve $C_A\subset\mathbb P^3$ has degree $d$, while a
homogeneous relation of degree $r$ gives a binomial hypersurface of degree
$r$ containing it.  Two independent relations cut the projective torus in
codimension two, leading to the same product bound.  We retain the
elementary determinant proof because it records the sharp constant and the
active-degree filtration directly.
\end{remark}

\subsection{The Hilbert function in active rank at most one}

Let $K$ be a field and $S=K[x_0,x_1,x_2,x_3]$ with its standard grading.
Define
\begin{equation}\label{eq:toricmap}
 \varphi_A:S\longrightarrow K[u,v],\qquad
 x_i\longmapsto u^{d-a_i}v^{a_i},
\end{equation}
where $(a_0,a_1,a_2,a_3)=(0,a,b,d)$, and put
$I_A=\ker\varphi_A$.  The degree-$h$ monomials map to monomials indexed
by the elements of $hA$, and hence
\begin{equation}\label{eq:Hilbert}
 |hA|=\dim_K(S/I_A)_h.
\end{equation}

\begin{proposition}[Rank-zero and rank-one formula]\label{prop:rankone}
Let $h\geq1$ be an integer.  Suppose the $h$-active relations of $A$ have
rank at most one.
If their rank is zero, then $|hA|=M_h$.  If their rank is one and the
primitive active direction has degree $r$, then $2\leq r\leq h$ and
\begin{equation}\label{eq:rankone}
 |hA|=M_h-M_{h-r}.
\end{equation}
\end{proposition}

\begin{proof}
If there is no active relation, then $(I_A)_h=0$, and \eqref{eq:Hilbert}
gives $|hA|=\dim_K S_h=M_h$.

Suppose the active span has rank one.  Let $V$ be its rational span, and
let $z$ be the primitive generator of $V\cap\Z^4$, with sign chosen
arbitrarily.  Since $V$ lies in the two kernels defining \eqref{eq:lattice},
we have $z\in L_A$.  Some nonzero active relation is an integral multiple
of $z$, so $r=\rho(z)\leq h$.  Also $r\geq2$, since a relation of degree
one would identify two elements of $A$.

Set
\[
 f_z=x^{z^+}-x^{z^-}\in I_A.
\]
The toric ideal component $(I_A)_h$ is spanned by degree-$h$ fibre
binomials.  After cancelling the monomial gcd from such a binomial, its
exponent difference is $nz$ for some integer $n$.  Reversing $z$ if
necessary, the remaining binomial is a monomial multiple of
\[
 (x^{z^+})^{|n|}-(x^{z^-})^{|n|},
\]
which is divisible by $f_z$.  Conversely, every degree-$h$ multiple of
$f_z$ belongs to $I_A$.  Therefore
\begin{equation}\label{eq:principalpiece}
 (I_A)_h=(f_z)_h.
\end{equation}
Multiplication by the nonzero degree-$r$ polynomial $f_z$ injects
$S_{h-r}$ into $S_h$.  Thus
\[
 \dim_K(I_A)_h=M_{h-r},
\]
and \eqref{eq:rankone} follows.
\end{proof}

Formula~\eqref{eq:rankone} is O'Bryant's first two successive-minima formula
\cite[Theorem~2]{OBryant2025}.  That theorem, together with the explicit
family in his Lemma~1, already realizes every tetrahedral-difference value in
the list below by witnesses of arbitrarily large diameter; Nathanson gives
another explicit $h$-adic realization
\cite{NathansonTriangular,NathansonTetrahedral}.  The converse proved below is
that once normalized diameter exceeds $h^2$, no other
cardinalities occur.

\section{Four-point Cohen--Macaulay charts and cubic spectrum growth}
\label{sec:family}

We now construct the Hilbert functions that will fill the quadratic core.
The four-parameter formulation makes the determinantal structure
transparent and supplies the four independent variables needed for the
positive-density argument.
The determinantal and arithmetically Cohen--Macaulay mechanisms are classical
in the theory of projective monomial curves; see, for example,
\cite{BresinskyRenschuch1980,LiPatilRoberts2012,Morales1987,
MoralesSimis1993,PatilRoberts2003}.  The contribution here is the
four-parameter additive chart and the collision geometry of its Hilbert
values.

Let $p,q,r,s$ be positive integers such that
\begin{equation}\label{eq:familyhyp}
 1\leq q<p,\qquad 1\leq r<p,
 \qquad (p,q)=1,
 \qquad q+s\geq p-r+1.
\end{equation}
Set
\begin{equation}\label{eq:familyA}
 d=ps+qr,
 \qquad
 A_{p,q,r,s}=\{0,q,p,d\}.
\end{equation}
Then $0<q<p<d$ and $A_{p,q,r,s}$ is primitive.

As before, let $S=K[x_0,x_1,x_2,x_3]$ and let $I_A$ be the kernel of the
map \eqref{eq:toricmap} for $A=A_{p,q,r,s}$.

\begin{theorem}[Hilbert series of the family]\label{thm:family}
Under \eqref{eq:familyhyp}, put $\alpha=q+s-p+r-1$.  The toric ideal $I_A$
is generated by
\begin{align}
 f_1&=x_1^p-x_0^{p-q}x_2^q,\label{eq:f1}\\
 f_2&=x_2^{q+s}-x_0^\alpha x_1^{p-r}x_3,\label{eq:f2}\\
 f_3&=x_1^r x_2^s-x_0^{r+s-1}x_3.\label{eq:f3}
\end{align}
Its Hilbert series is
\begin{equation}\label{eq:HSfamily}
 \HS_{S/I_A}(t)=
 \frac{1-t^p-t^{q+s}-t^{r+s}+t^{p+s}+t^{q+r+s}}
 {(1-t)^4}.
\end{equation}
Consequently, if
\[
 \widetilde M_j=\begin{cases}
 \binom{j+3}{3},&j\geq0,\\
 0,&j<0,
 \end{cases}
\]
then, for every integer $h\geq0$,
\begin{equation}\label{eq:Hfamily}
 |hA|=\widetilde M_h-\widetilde M_{h-p}
 -\widetilde M_{h-q-s}-\widetilde M_{h-r-s}
 +\widetilde M_{h-p-s}+\widetilde M_{h-q-r-s}.
\end{equation}
\end{theorem}

\begin{proof}
The exponent $\alpha$ in \eqref{eq:f2} is nonnegative by
\eqref{eq:familyhyp}.  Directly comparing total degrees and $A$-weights
shows that the three binomials belong to $I_A$.

Choose a positive real number $\lambda$ such that
\begin{equation}\label{eq:lambda}
 \frac qp<\lambda<\frac{q+s}{p-r}.
\end{equation}
The interval is nonempty because $q(p-r)<p(q+s)$.  Order monomials first by
total degree, then by the weight vector $(0,\lambda,1,0)$, breaking any
remaining ties with an arbitrary monomial order.  The leading monomials of
$f_1,f_2,f_3$
are
\begin{equation}\label{eq:leading}
 x_1^p,\qquad x_2^{q+s},\qquad x_1^r x_2^s.
\end{equation}
The two nontrivial $S$-polynomial identities are
\begin{align}
 x_2^sf_1-x_1^{p-r}f_3&=-x_0^{p-q}f_2,\label{eq:S1}\\
 x_1^rf_2-x_2^qf_3&=-x_0^\alpha x_3f_1.\label{eq:S2}
\end{align}
The leading monomials of $f_1$ and $f_2$ are relatively prime.  Buchberger's
criterion therefore shows that $f_1,f_2,f_3$ form a Gr\"obner basis of the
ideal $J$ they generate, with
\begin{equation}\label{eq:initJ}
 \inideal(J)=(x_1^p,x_2^{q+s},x_1^rx_2^s).
\end{equation}

The standard monomials in $x_1,x_2$ form the L-shaped set
\begin{equation}\label{eq:Lshape}
 \begin{split}
 \mathcal B={}&\{x_1^ix_2^j:0\leq i<p,\ 0\leq j<q+s\}\\
 &\setminus
 \{x_1^ix_2^j:r\leq i<p,\ s\leq j<q+s\}.
 \end{split}
\end{equation}
Its cardinality is
\begin{equation}\label{eq:Lcount}
 |\mathcal B|=p(q+s)-(p-r)q=ps+rq=d.
\end{equation}
It follows from the Gr\"obner normal form that $S/J$ is a free
$K[x_0,x_3]$-module with basis $\mathcal B$.  In particular, $S/J$ is
Cohen--Macaulay of dimension two and degree $d$.

We have $J\subseteq I_A$.  The prime ideal $I_A$ has height two, and the
projective monomial curve $\operatorname{Proj}(S/I_A)$ has degree $d$.
Indeed, on the dense torus its parametrization is
\[
 t\longmapsto[1:t^q:t^p:t^d],
\]
which is birational because $\gcd(p,q,d)=1$.  After extending scalars to an
algebraic closure, degree and Hilbert series are unchanged, while a general
hyperplane is available.  The parametrization is given by degree-$d$ forms
without a common zero, so such a hyperplane pulls back to a divisor of degree
$d$ on $\mathbb P^1$; birationality therefore gives $\deg C_A=d$.

For clarity, we spell out the standard unmixed-degree argument that
$J=I_A$; see also \cite{BrunsHerzog1993} for the associativity formula and
unmixedness facts used here.  Since the ideals have the same height, $I_A$
is a minimal prime
of $J$.  The associativity formula for multiplicities says that its
contribution to $\deg(S/J)$ is at least $\deg(S/I_A)=d$.  Equality
\eqref{eq:Lcount} shows that $I_A$ is the only minimal prime and has
generic multiplicity one, so $J_{I_A}=(I_A)_{I_A}$.  If $I_A/J$ were
nonzero, then, as a submodule of the Cohen--Macaulay and hence unmixed
module $S/J$, it would have an associated prime among the minimal primes
of $J$, necessarily $I_A$.  This contradicts $(I_A/J)_{I_A}=0$.
Therefore $J=I_A$.

The standard monomials give
\begin{align*}
 \HS_{S/I_A}(t)
 &=\frac1{(1-t)^2}
 \left(
 \frac{(1-t^p)(1-t^{q+s})}{(1-t)^2}
 -t^{r+s}\frac{(1-t^{p-r})(1-t^q)}{(1-t)^2}
 \right)\\
 &=\frac{1-t^p-t^{q+s}-t^{r+s}+t^{p+s}+t^{q+r+s}}
 {(1-t)^4},
\end{align*}
which is \eqref{eq:HSfamily}.  Taking the degree-$h$ coefficient and using
\eqref{eq:Hilbert} proves \eqref{eq:Hfamily}.
\end{proof}

The two syzygies in \eqref{eq:S1} and \eqref{eq:S2} are the columns of
\[
 \Psi=
 \begin{pmatrix}
  x_2^s & x_0^\alpha x_3\\
  x_0^{p-q} & x_1^r\\
  -x_1^{p-r} & -x_2^q
 \end{pmatrix}.
\]
Thus Theorem~\ref{thm:family} also gives the minimal Hilbert--Burch
resolution
{\small
\[
 0\longrightarrow S(-p-s)\oplus S(-q-r-s)
 \xrightarrow{\ \Psi\ }
 S(-p)\oplus S(-q-s)\oplus S(-r-s)
 \longrightarrow I_A\longrightarrow0.
\]
}
Every entry of $\Psi$ has positive degree under \eqref{eq:familyhyp}
(in particular, $x_0^\alpha x_3$ has degree $\alpha+1\geq1$), so no scalar
cancellation is possible and the displayed graded resolution is minimal.
This recovers the numerator in \eqref{eq:HSfamily}; the Gr\"obner proof has
the additional advantage of identifying the L-shaped free basis directly.

\begin{remark}[The Euclidean-division slice]
Using generating functions, Chen and Yang obtained the exact specialization
of Theorem~\ref{thm:family} with $q=1$ and positive remainder
\cite[Corollary~1]{ChenYang2026}.  Their set $\{0,1,a,b\}$ with
$b=Qa+R$ becomes
\[
 \{0,1,p,ps+r\}
 \qquad\text{under}\qquad (a,Q,R)=(p,s,r),
\]
and $Q+R\geq a$ becomes $s+r\geq p$, equivalently
$1+s\geq p-r+1$.  Our hypotheses have $r=R>0$; Chen and Yang also treat
$R=0$.  Their six-term formula agrees with \eqref{eq:Hfamily} term by term
after Pascal's identity.  Theorem~\ref{thm:family} extends this slice to
arbitrary coprime $p,q$ and identifies the toric Cohen--Macaulay and
Hilbert--Burch mechanism; the positive-density box used here has
$q\asymp h$ rather than $q=1$.
\end{remark}

Although not needed for the density argument, the following slice isolates
a simpler quadratic--cubic defect formula contained in the same Hilbert
numerator.

\begin{corollary}[Three-term defect]\label{cor:defect}
Let $s=1$ in Theorem~\ref{thm:family}, and suppose
$p,q,r\leq h<q+r$.  Put
\[
 x=h-p,\qquad y=h-q,\qquad z=h-r.
\]
Then
\begin{equation}\label{eq:three-term}
 M_h-|hA_{p,q,r,1}|
 =Q(x)+T(y)+T(z),
\end{equation}
where
\[
 Q(x)=\binom{x+2}{2},\qquad T(y)=\binom{y+2}{3}.
\]
\end{corollary}

\begin{proof}
In \eqref{eq:Hfamily}, the final term vanishes because $q+r>h$.
Moreover,
\[
 \widetilde M_{h-p}-\widetilde M_{h-p-1}
 =\binom{h-p+2}{2}=Q(x),
\]
while
\[
 \widetilde M_{h-q-1}=T(y),\qquad
 \widetilde M_{h-r-1}=T(z).
\]
\end{proof}

\subsection{The cubic-density chart}\label{sec:energy}

We now use all four parameters in Theorem~\ref{thm:family}.  The extra
parameter turns the defect into a signed sum of four independent values of
one cubic polynomial.  Vaughan's sharp unweighted eighth-moment estimate
then gives a critical, no-$\varepsilon$ collision bound.

\subsubsection{A proportional four-parameter box}

For all sufficiently large integers $h$, define
\begin{align}
 U&=\left[\left\lceil\frac h{10}\right\rceil,
          \left\lfloor\frac{9h}{80}\right\rfloor\right]\cap\Z,
 &V&=\left[\left\lceil\frac h{20}\right\rceil,
          \left\lfloor\frac h{16}\right\rfloor\right]\cap\Z,
                                                        \label{eq:UVbox}\\
 B&=C=\left[\left\lceil\frac h5\right\rceil,
          \left\lfloor\frac{17h}{80}\right\rfloor\right]\cap\Z.
                                                        \label{eq:BCbox}
\end{align}
Thus each interval has cardinality $h/80+O(1)$.  For
$(u,v,b,c)\in U\times V\times B\times C$, set
\begin{equation}\label{eq:fullsparams}
 s=u-v,\qquad p=h-u,\qquad
 q=h-b-s,\qquad r=h-c-s,
 \qquad d=ps+qr.
\end{equation}

\begin{lemma}[The full-$s$ defect chart]\label{lem:fullschart}
The parameters in \eqref{eq:fullsparams} satisfy all the inequalities in
\eqref{eq:familyhyp}.  In addition,
\[
 q+r+s>h,
 \qquad 0<q<p<d<\frac{13}{20}h^2.
\]
Whenever $(p,q)=1$, the set $A_{p,q,r,s}$ is primitive and
\begin{equation}\label{eq:fourtermdefect}
 M_h-|hA_{p,q,r,s}|=D(u,v,b,c):=M_u-M_v+M_b+M_c.
\end{equation}
\end{lemma}

\begin{proof}
The box gives
\begin{equation}\label{eq:parameterbounds}
 \frac{3h}{80}\leq s\leq\frac h{16},\qquad
 \frac{71h}{80}\leq p\leq\frac{9h}{10},\qquad
 \frac{29h}{40}\leq q,r\leq\frac{61h}{80}.
\end{equation}
In particular all parameters are positive.  Moreover,
\[
 p-q=b-v>0,\qquad p-r=c-v>0,
\]
so $q,r<p$.  The remaining inequality in \eqref{eq:familyhyp} follows
with a large margin from
\[
 q+s=h-b\geq\frac{63h}{80},\qquad
 p-r+1=c-v+1\leq\frac{13h}{80}+1.
\]
Also
\[
 q+r+s=2h-b-c-s
 \geq2h-\frac{39h}{80}=\frac{121h}{80}>h.
\]
Thus the final term in \eqref{eq:Hfamily} vanishes.  Since
\[
 h-p=u,\qquad h-p-s=v,\qquad h-q-s=b,\qquad h-r-s=c,
\]
formula \eqref{eq:Hfamily} is exactly \eqref{eq:fourtermdefect}.

Since $s\geq1$, one has $d=ps+qr>p$.  Finally,
\[
 d\leq\left(\frac9{10}\frac1{16}
       +\left(\frac{61}{80}\right)^2\right)h^2
 =\frac{4081}{6400}h^2<\frac{13}{20}h^2.
\]
When $(p,q)=1$, primitivity follows from
$\gcd(p,q,d)=\gcd(p,q,ps+qr)=1$.
\end{proof}

\begin{lemma}[Primitive parameter count]\label{lem:fullsprimitive}
The number of quadruples in $U\times V\times B\times C$ for which
$(p,q)=1$ is $\gg h^4$.
\end{lemma}

\begin{proof}
The identities in \eqref{eq:fullsparams} give
\begin{equation}\label{eq:gcdrewrite}
 (p,q)=(h-u,h-b-u+v)=(h-u,b-v).
\end{equation}
For a positive integer $\ell$, put
\[
 A_\ell=\#\{u\in U:\ell\mid h-u\},\qquad
 B_\ell=\#\{(v,b)\in V\times B:\ell\mid b-v\}.
\]
Uniformly for $\ell\leq h$,
\[
 A_\ell=\frac{|U|}{\ell}+O(1),\qquad
 B_\ell=\frac{|V||B|}{\ell}+O(|V|).
\]
Writing $\mu$ for the M\"obius function and $\zeta$ for the Riemann zeta
function, M\"obius inversion gives
\begin{align*}
 &\#\{(u,v,b)\in U\times V\times B:(h-u,b-v)=1\}\\
 &\quad=\sum_{\ell\leq h}\mu(\ell)A_\ell B_\ell\\
 &\quad=|U||V||B|\sum_{\ell\leq h}\frac{\mu(\ell)}{\ell^2}
       +O(h^2\log h)\\
 &\quad=\frac{|U||V||B|}{\zeta(2)}+O(h^2\log h)
 \gg h^3.
\end{align*}
There are $|C|\gg h$ independent choices for $c$.
\end{proof}

\subsubsection{The chart--energy principle}

Both abundance arguments use the same energy mechanism.  The following
lemma separates this analytic estimate from the algebraic construction.

\begin{lemma}[Chart--energy principle]\label{lem:chart-energy}
Let $m\geq2$, let $F:\mathbb Z\to\mathbb Z$, let $I_1,\ldots,I_m$ be
nonempty finite
integer intervals, and let $\sigma_j\in\{1,-1\}$.  For
$\mathcal P\subseteq I_1\times\cdots\times I_m$, put
\[
 \Phi(n_1,\ldots,n_m)=\sum_{j=1}^m\sigma_jF(n_j),\qquad
 S_j(\alpha)=\sum_{n\in I_j}\exp(2\pi i\alpha F(n)).
\]
Then
\begin{equation}\label{eq:chart-energy-principle}
 |\Phi(\mathcal P)|\geq
 \frac{|\mathcal P|^2}
 {\displaystyle\prod_{j=1}^m
  \left(\int_0^1|S_j(\alpha)|^{2m}\,d\alpha\right)^{1/m}}.
\end{equation}
\end{lemma}

\begin{proof}
The collision count of $\Phi$ on the full product box is, by
orthogonality,
\[
 \int_0^1\prod_{j=1}^m|S_j(\sigma_j\alpha)|^2\,d\alpha.
\]
H\"older bounds this by the denominator in
\eqref{eq:chart-energy-principle}; since
$|S_j(-\alpha)|=|S_j(\alpha)|$, the signs do not affect this bound.  The
collision count on $\mathcal P$ is no larger, and
Cauchy--Schwarz bounds it below by
$|\mathcal P|^2/|\Phi(\mathcal P)|$.
\end{proof}

\subsubsection{The critical eighth moment}

We use the following classical mean-value theorem.  For the unweighted
sequence, the fixed integral-cubic formulation below is recorded in the
discussion immediately following \cite[Theorem~4.1]{HughesWooley2022},
building on \cite[Theorem~2]{Vaughan1986}.

\begin{lemma}[Vaughan's cubic eighth moment]\label{lem:cubiceighth}
Let $F\in\Z[t]$ be a fixed polynomial of degree three.  Then, for every
integer $N\geq 1$,
\begin{equation}\label{eq:cubiceighth}
 \#\left\{(n_1,\ldots,n_8)\in([-N,N]\cap\Z)^8:
  \sum_{i=1}^4F(n_i)=\sum_{i=5}^8F(n_i)\right\}
 \ll_F N^5.
\end{equation}
\end{lemma}

For an integer interval $I\subseteq[0,h]$, write
\[
 S_I(\alpha)=\sum_{n\in I}\exp(2\pi i\alpha M_n).
\]
Since
\begin{equation}\label{eq:centeredtetrahedral}
 6M_n=(n+1)(n+2)(n+3)=(n+2)^3-(n+2)
\end{equation}
is an integral cubic, Lemma~\ref{lem:cubiceighth} and orthogonality imply
\begin{equation}\label{eq:intervaleighth}
 \int_0^1|S_I(\alpha)|^8\,d\alpha\ll h^5
\end{equation}
for each interval in \eqref{eq:UVbox}--\eqref{eq:BCbox}.  Indeed, the
 integral counts equal sums of four values of $M_n$ with all variables in
 $I$.  Multiplying the equation by $6$ does not change its solutions, and
 the restricted solutions form a subset of those counted in
 \eqref{eq:cubiceighth} with $N=h$.
The polynomial in \eqref{eq:centeredtetrahedral} is fixed once and for all,
so the implied constant is uniform in $h$ and in the four intervals.

\subsubsection{The spectrum image}

\begin{proof}[Proof of Theorem~\ref{thm:main}]
Let $\mathcal P_h$ be the primitive parameter subset counted in
Lemma~\ref{lem:fullsprimitive}.  For every parameter tuple in this set,
$p=h-u\leq h$ and $r+s=h-c\leq h$.  Thus the independent binomials
$f_1$ and $f_3$ in \eqref{eq:f1} and \eqref{eq:f3} give two active relation
directions.  Every constructed four-set therefore has maximal active rank
two.

Apply Lemma~\ref{lem:chart-energy} with $m=4$, $F(n)=M_n$, signs
$(1,-1,1,1)$, and intervals $(U,V,B,C)$.  Together with
\eqref{eq:intervaleighth}, it gives
\[
 |D(\mathcal P_h)|
 \gg\frac{h^8}{h^5}=h^3.
\]
Every one of these defects lies in a fixed macroscopic window.  Indeed,
monotonicity of $M_x=(x+1)(x+2)(x+3)/6$ on $x\geq0$ and the definitions of
$U,V,B,C$ give
\begin{align*}
 D&\geq M_{h/10}-M_{h/16}+2M_{h/5}
     =\left(\frac1{6000}-\frac1{24576}+\frac1{375}\right)h^3+O(h^2)
     >\frac{h^3}{400},\\
 D&\leq M_{9h/80}-M_{h/20}+2M_{17h/80}
     =0.0034151\ldots\,h^3+O(h^2)
     <\frac{h^3}{250}
\end{align*}
for all sufficiently large $h$.  Thus every value just counted lies in the
stated window, proving the
local assertion \eqref{eq:localmain}.
No injectivity of the parametrization is needed: repeated descriptions
of the same four-set have the same defect by
\eqref{eq:fourtermdefect}, and hence are already charged to the collision
energy.

Lemma~\ref{lem:fullschart} turns distinct defects into distinct members
of $\Rcal(h,4)$ represented by primitive sets of normalized diameter less
than $13h^2/20$.  Thus $|\Rcal(h,4)|\gg h^3$.  The reverse estimate follows
from the ambient interval \eqref{eq:ambient}, which contains $O(h^3)$
integers when $k=4$.  Dividing by
$|[3h+1,\binom{h+3}{3}]\cap\mathbb Z|\sim h^3/6$ proves the final
assertion of Theorem~\ref{thm:main}.
\end{proof}

\section{Prime-coded suspension amplification and the ambient exponent}
\label{sec:endpoint-spectrum}

We first prove the quantitative amplification theorem and then verify its
hypotheses for the stable four-point Hilbert--Burch chart.  The suspension
degrees are primes of order $h$.  Their product gives a recoverable divisor
label with bounded overlap between prime tuples, while exact averaging
reduces the remaining collision estimate to a cubic problem.  This section
proves Theorem~\ref{thm:endpoint-spectrum}.

\subsection{A prime-coded spectrum-amplification theorem}
\label{sec:prime-suspension-amplification}

The endpoint construction is governed by a general exponent-raising
principle.  We formulate it quantitatively, keeping the algebraic,
arithmetic, and energy inputs separate.

The Betti, ACM, and Cohen--Macaulay-type assertions in this subsection are
over the arbitrary fixed coefficient field $K$ declared in
Section~\ref{sec:suspension-calculus}; consequently they hold separately over
every field.

\begin{definition}[Shift-stable binomial chart]
\label{def:shift-stable-chart}
Fix an integer $d\geq1$ and $0<\gamma<1$.  A \emph{shift-stable $d$-chart
of depth $\gamma h$} is a family
$(B_\theta)_{\theta\in\Theta_h}$, indexed by all sufficiently large integers
$h$ with $\Theta_h\ne\varnothing$, of primitive
$(d+1)$-subsets of $\mathbb Z_{\geq0}$ containing zero for which there are
fixed integers $s\geq1$ and $c_1,\ldots,c_s$, and integer residuals
$x_\nu(\theta,h)$, such that, uniformly for $\theta\in\Theta_h$ and
$J\in\mathbb Z$ with $0\leq J\leq\gamma h$,
\begin{equation}\label{eq:stable-chart-general}
 |(h-J)B_\theta|
 =\sum_{\nu=1}^{s}c_\nu
  \binom{x_\nu(\theta,h)-J+d}{d},
 \qquad x_\nu(\theta,h)-J\geq0.
\end{equation}
Terms for which $x_\nu(\theta,h)-J<0$ throughout the shift window are
omitted.  This condition is satisfied, in particular, when the family
remains in one finite Hilbert-numerator chamber throughout the window.
Here $d=|B_\theta|-1$ is the dimension of the ambient simplex and the degree
of the binomial coefficient as a polynomial in its upper index; it is not
the dimension of the associated projective curve.
\end{definition}

The next identity is the arithmetic source of the prime labels.  It applies
to every simplex term, independently of the geometry which produced it.

\begin{lemma}[Factorial box divisibility]
\label{lem:factorial-box-divisibility}
Let $d,t\geq0$ be integers, let $\ell_1,\ldots,\ell_t$ be positive
integers, put $R=\ell_1\cdots\ell_t$, and, for an indeterminate $x$, set
\[
 \Psi_{\boldsymbol\ell,d}(x)
 =\sum_{0\leq j_i<\ell_i}
   \binom{x-j_1-\cdots-j_t+d}{d}.
\]
Here $\binom{X}{d}=X(X-1)\cdots(X-d+1)/d!$ is the generalized binomial
polynomial (with $\binom{X}{0}=1$).
There is a universal polynomial
\[
 F_{d,t}(X,L_1,\ldots,L_t)
 \in\mathbb Z[X,L_1,\ldots,L_t]
\]
of degree $d$ in $X$ such that, after the specialization
$X=x$ and $L_i=\ell_i$,
\begin{equation}\label{eq:factorial-box-divisibility}
 (d+t)!\,\Psi_{\boldsymbol\ell,d}(x)
 =R F_{d,t}(x,\ell_1,\ldots,\ell_t).
\end{equation}
Its leading coefficient in $X$ is $(d+t)!/d!$.
\end{lemma}

\begin{proof}
Repeated hockey-stick summation, interpreted as a polynomial identity,
gives, with $[t]=\{1,\ldots,t\}$,
\begin{equation}\label{eq:factorial-box-inclusion-exclusion}
 \Psi_{\boldsymbol\ell,d}(x)
 =\sum_{J\subseteq[t]}(-1)^{|J|}
  \binom{x-\sum_{j\in J}\ell_j+d+t}{d+t}.
\end{equation}
After multiplication by $(d+t)!$, replace $x,\boldsymbol\ell$ by independent
indeterminates $X,\mathbf L$ and call the resulting integer polynomial
$P(X,\mathbf L)$.  On setting any $L_i=0$, subsets
containing $i$ cancel in pairs with those not containing $i$.  Hence every
$L_i$ divides $P$.  These indeterminates are pairwise coprime in the
polynomial UFD, so $\prod_iL_i$ divides $P$.  Finally, the $t$ finite
differences in \eqref{eq:factorial-box-inclusion-exclusion} lower the degree in $X$ from
$d+t$ to $d$ and give the asserted leading coefficient.
\end{proof}

Fix an integer $t\geq0$, fix pairwise disjoint intervals, and write
\begin{equation}\label{eq:amplifier-prime-windows}
 I_i=[\alpha_i,\beta_i],\qquad
 0<\alpha_i<\beta_i<1,
 \qquad \sum_{i=1}^{t}\beta_i<\gamma,
\end{equation}
and put $\mathbf I=(I_1,\ldots,I_t)$ and
\[
 \mathfrak P_i(h)
 =\{\ell\text{ prime}:\alpha_i h\leq\ell\leq\beta_i h\}.
\]
When $t=0$, products over $i$ and the prime-tuple product are interpreted as
the singleton empty tuple; in particular $R=1$.  For every
$\boldsymbol\ell\in\prod_i\mathfrak P_i(h)$, including the empty tuple when
$t=0$, let $\Theta_h(\boldsymbol\ell)\subseteq\Theta_h$ be a finite nonempty
retained subfamily for which the required suspension choices have been fixed.
We call
$(\theta,\boldsymbol\ell)$ \emph{admissible} when
$\boldsymbol\ell\in\prod_i\mathfrak P_i(h)$ and
$\theta\in\Theta_h(\boldsymbol\ell)$.  Define
\begin{align}
 V_{\theta,\boldsymbol\ell}(h)
 &=\sum_{0\leq j_i<\ell_i}
   |(h-j_1-\cdots-j_t)B_\theta|,
 \label{eq:amplifier-value}\\
 E_{\boldsymbol\ell}(h)
 &=\#\{(\theta,\theta')\in\Theta_h(\boldsymbol\ell)^2:
 V_{\theta,\boldsymbol\ell}(h)
 =V_{\theta',\boldsymbol\ell}(h)\}.
 \label{eq:amplifier-energy}
\end{align}
Under the empty-tuple convention this gives
$V_{\theta,\varnothing}(h)=|hB_\theta|$.  Thus the ambient chart itself may be
infinite, but every family to which the finite energy argument is applied is
explicitly finite.
For sufficiently large $h$, when every nonvacuous prime window is nonempty,
put
\[
 N_h=\min_{\boldsymbol\ell\in\prod_i\mathfrak P_i(h)}
       |\Theta_h(\boldsymbol\ell)|,
 \qquad
 E_h=\max_{\boldsymbol\ell\in\prod_i\mathfrak P_i(h)}
       E_{\boldsymbol\ell}(h).
\]

\begin{theorem}[Prime-coded suspension amplification]
\label{thm:prime-filter-amplification}
Fix integers $d\geq1$ and $t\geq0$, and a shift-stable $d$-chart.  For all
sufficiently large integers $h$ and every admissible pair
$(\theta,\boldsymbol\ell)$, the value
$V_{\theta,\boldsymbol\ell}(h)$ in \eqref{eq:amplifier-value} is realized as
$|hC|$ by a primitive $(d+1+t)$-set obtained from $B_\theta$ by $t$ cyclic suspensions,
and
\begin{equation}\label{eq:prime-suspension-quantitative}
 |\Rcal(h,d+1+t)|
 \gg_{d,t,I_1,\ldots,I_t}
 \left(\frac h{\log h}\right)^t\frac{N_h^2}{E_h}.
\end{equation}
If, throughout the shift window,
\begin{equation}\label{eq:amplifier-macroscopic-seed}
 c_0 h^d\leq |(h-J)B_\theta|\leq C_0 h^d
\end{equation}
for fixed $0<c_0<C_0$, then every value used to prove
\eqref{eq:prime-suspension-quantitative} lies in a fixed interval
\begin{equation}\label{eq:amplifier-macroscopic-output}
[c_{d,t,\mathbf I}h^{d+t},C_{d,t,\mathbf I}h^{d+t}],
\end{equation}
where these constants also depend on the fixed bounds $c_0,C_0$ in
\eqref{eq:amplifier-macroscopic-seed}.

Every counted value admits a witness $C$ satisfying
\begin{align}
 \mathcal H_C(z)
 &=\mathcal H_{B_\theta}(z)\prod_{i=1}^t[\ell_i]_z,
 \label{eq:amplifier-Hilbert-transport}\\
 \mathcal K_C(z)
 &=\mathcal K_{B_\theta}(z)\prod_{i=1}^t(1-z^{\ell_i}),
 \label{eq:amplifier-K-transport}\\
 \mathcal B_C(\xi,z)
 &=\mathcal B_{B_\theta}(\xi,z)\prod_{i=1}^t(1+\xi z^{\ell_i}),
 \label{eq:amplifier-Betti-transport}\\
 r_q(C)&=r_q(B_\theta)+\#\{i:\ell_i\leq q\}
 \qquad(q\in\mathbb Z_{\geq0}),
 \label{eq:amplifier-rank-transport}\\
 \ndiam(C)&=\left(\prod_{i=1}^t\ell_i\right)\ndiam(B_\theta).
 \label{eq:amplifier-diameter-transport}
\end{align}
If the seed projective coordinate ring is Cohen--Macaulay, then projective
Cohen--Macaulayness and Cohen--Macaulay type are preserved.  Since every
$\ell_i<h$, maximal degree-$h$ active rank propagates from the seed to $C$.
\end{theorem}

\begin{proof}
Primitivity is preserved under cyclic suspension.  At each step, if a
residue representative has not already been prescribed, primitivity of the
current set and primality of $\ell_i$ guarantee a nonzero old point not
divisible by $\ell_i$; choose such a point as the residue representative.
The exact profile
formula in Theorem~\ref{thm:cyclic-suspension} gives
\eqref{eq:amplifier-value}; the prime-window inequality keeps every term
inside the chart.

Apply Lemma~\ref{lem:factorial-box-divisibility} to every term of
\eqref{eq:stable-chart-general}.  With $R=\prod_i\ell_i$ this gives
\begin{equation}\label{eq:amplifier-divisor-identity}
 (d+t)!\,V_{\theta,\boldsymbol\ell}(h)
 =R G_{\theta,\boldsymbol\ell}(h),
 \qquad G_{\theta,\boldsymbol\ell}(h)\in\mathbb Z.
\end{equation}
For large $h$, the distinct selected primes exceed $(d+t)!$, whence
\begin{equation}\label{eq:amplifier-R-divides-V}
 R\mid V_{\theta,\boldsymbol\ell}(h).
\end{equation}

For one fixed prime tuple, Cauchy--Schwarz applied to the multiplicity
function of $\theta\mapsto V_{\theta,\boldsymbol\ell}(h)$ gives at least
$N_h^2/E_h$ distinct values.  The prime number theorem in fixed proportional
intervals supplies $\asymp_{t,I_1,\ldots,I_t}(h/\log h)^t$ tuples.  If an
integer $v$ is produced by one of them, all selected primes divide $v$ by
\eqref{eq:amplifier-R-divides-V}, whereas
\[
 v\leq\binom{h+d+t}{d+t}=O_{d,t}(h^{d+t}).
\]
Every admissible prime is bounded below by a positive constant times $h$,
so $v$ has only $O_{d,t,I_1,\ldots,I_t}(1)$ admissible prime divisors.  The
windows are disjoint; hence only $O_{d,t,I_1,\ldots,I_t}(1)$ ordered tuples
can produce $v$.  Summing the fixed-label images and dividing by this bounded
overlap proves \eqref{eq:prime-suspension-quantitative}.

There are exactly $R\asymp_{t,I_1,\ldots,I_t}h^t$ summands in
\eqref{eq:amplifier-value}; thus \eqref{eq:amplifier-macroscopic-seed}
implies \eqref{eq:amplifier-macroscopic-output}.  The transport formulas are
the iterated forms of Theorem~\ref{thm:cyclic-suspension} and
Corollary~\ref{cor:iterated-suspension}.
\end{proof}

\begin{definition}[Prime-energy stability]
\label{def:prime-energy-stability}
A shift-stable $d$-chart is \emph{prime-energy stable with exponents
$(a,\mu)$} if, for every fixed $t\in\mathbb Z_{\geq0}$ and every fixed system of windows
satisfying \eqref{eq:amplifier-prime-windows}, admissible suspension choices
can be made so that, uniformly over the resulting prime tuples,
\[
 N_h\geq h^{a-o(1)},\qquad E_h\leq h^{\mu+o(1)}.
\]
\end{definition}

\begin{corollary}[Energy-stable endpoint principle]
\label{cor:hereditary-endpoint}
Suppose the seed chart is prime-energy stable with exponents $(a,\mu)$.
Then, for every fixed $t\in\mathbb Z_{\geq0}$,
\begin{equation}\label{eq:amplifier-exponent}
 |\Rcal(h,d+1+t)|\geq h^{2a-\mu+t-o(1)}.
\end{equation}
Thus the deficit $d-(2a-\mu)$ from the ambient exponent is invariant throughout
this prime-energy-stable suspension tower.  In particular, the endpoint
identity $2a-\mu=d$ implies
\begin{equation}\label{eq:one-cardinality-propagation}
 |\Rcal(h,d+1+t)|=h^{d+t+o(1)}
\end{equation}
for every fixed $t\in\mathbb Z_{\geq0}$.
\end{corollary}

\begin{proof}
Insert the two hypotheses into
\eqref{eq:prime-suspension-quantitative} and absorb the fixed power of
$\log h$ into $h^{o(1)}$.  The upper bound in
\eqref{eq:one-cardinality-propagation} is the ambient bound.
\end{proof}

\begin{remark}[Scope]
\label{rem:amplifier-scope}
The theorem propagates witnesses and their exact Hilbert and homological
profiles.  It does not exclude unrelated representations of the same
cardinality, and its energy bound is not a pointwise frequency theorem.
Moreover, the prime windows contribute only $(h/\log h)^t$ labels, so this
mechanism alone does not prove positive density for $k\geq5$.
\end{remark}

\subsection{A moving-coefficient cubic eighth moment}

For $\nu\geq2$, let $\tau_\nu(n)$ be the number of ordered factorizations
of $n$ into $\nu$ positive factors.  For $X\geq2$, put
\begin{equation}\label{eq:maximal-divisor-function}
 \Delta_\nu(X)=\max_{1\leq n\leq X}\tau_\nu(n),
\end{equation}
and use the classical maximal-order estimate
\begin{equation}\label{eq:maximal-divisor-order}
 \Delta_\nu(X)\leq
 \exp\!\left(O_\nu\!\left(\frac{\log X}{\log\log X}\right)\right).
\end{equation}
See, for example, \cite[Chapter~I.5]{Tenenbaum2015}.

\begin{lemma}[Uniform moving-$L$ cubic eighth moment]
\label{lem:uniform-moving-cubic}
Fix $C\geq1$.  There is $C_1=C_1(C)$ such that, uniformly for $N\geq3$,
integers $L$ satisfying $|L|\leq CN^2$, and arbitrary sets
$E\subset[-CN,CN]\cap\mathbb Z$, one has
\begin{equation}\label{eq:uniform-moving-cubic}
 \int_0^1\left|
  \sum_{x\in E}\exp\bigl(2\pi i\alpha(x^3+Lx)\bigr)
 \right|^8d\alpha
 \ll_C N|E|^4\Delta_3(C_1N^3).
\end{equation}
In particular, when $|E|\ll_C N$, the right side is
\begin{equation}\label{eq:uniform-moving-cubic-maxorder}
 N^5\exp\!\left(O_C\!\left(
  \frac{\log N}{\log\log N}\right)\right).
\end{equation}
\end{lemma}

\begin{proof}
Put $\phi_L(x)=x^3+Lx$ and $A=|E|$.  The proof follows the foliation
underlying the univariate-polynomial restriction estimate of
Hughes--Wooley \cite[Theorem~4.1]{HughesWooley2022}; we give the details
that make uniformity in $L$ explicit.

For $m\in\mathbb Z$, let $c_2(m)$ count quadruples in $E^4$ satisfying
\begin{equation}\label{eq:restricted-two-pair}
 x_1+x_2=y_1+y_2,\qquad
 \phi_L(x_1)+\phi_L(x_2)-\phi_L(y_1)-\phi_L(y_2)=m,
\end{equation}
and let $c'_2(m)$ count quadruples satisfying only the second equation.
Writing
\[
 x_1=y+e_1,\qquad x_2=y+e_2,
 \qquad y_2=y+e_1+e_2,
\]
one obtains the coefficient-free factorization
\begin{equation}\label{eq:restricted-cubic-factorization}
 \phi_L(x_1)+\phi_L(x_2)-\phi_L(y)-\phi_L(y_2)
 =-3e_1e_2(2y+e_1+e_2).
\end{equation}
Thus
\begin{equation}\label{eq:c2-bounds}
 c_2(0)\leq3A^2,\qquad
 c_2(m)\ll\tau_3(|m|)\quad(m\ne0).
\end{equation}
Indeed, when the right side of
\eqref{eq:restricted-cubic-factorization} vanishes, one of its three
factors vanishes, and each case has at most $A^2$ solutions.  When it is
nonzero, an ordered signed factorization determines $e_1,e_2,y$ and hence
all four variables.

To control $c'_2(0)$, let $r(m)$ count pairs $(x,y)\in E^2$ with
$\phi_L(x)-\phi_L(y)=m$.  The identity
\begin{equation}\label{eq:first-cubic-factorization}
 \phi_L(x)-\phi_L(y)
 =(x-y)(x^2+xy+y^2+L)
\end{equation}
shows that, for each fixed $y$, the equation
$\phi_L(x)=\phi_L(y)$ has at most three solutions in $x$.  Thus
$r(0)\leq3A$, and
$r(m)\ll\tau_2(|m|)$ for $m\ne0$.  For the latter assertion, after fixing
the signed divisor $d=x-y$, the remaining equation is quadratic in $y$.
The involution $(x,y)\mapsto(y,x)$ gives $r(-m)=r(m)$, and
$\sum_m r(m)=A^2$.  Hence
\[
 c'_2(0)=\sum_m r(m)r(-m)=\sum_m r(m)^2.
\]
Since $|\phi_L(x)-\phi_L(y)|\ll_C N^3$, the zero fibre contributes
$r(0)^2\leq9A^2$, while
\[
 \sum_{m\ne0}r(m)^2
 \leq\left(\max_{m\ne0}r(m)\right)\sum_mr(m)
 \ll_C A^2\Delta_2(C_1N^3)
 \leq A^2\Delta_3(C_1N^3).
\]
Consequently,
\begin{equation}\label{eq:c2prime-zero}
 c'_2(0)\ll_C A^2\Delta_3(C_1N^3).
\end{equation}
Also, trivially,
\begin{equation}\label{eq:c2prime-mass}
 \sum_m c'_2(m)=A^4.
\end{equation}

It remains to recall the short foliation step.  If
\[
 S(\alpha)=\sum_{x\in E}\exp(2\pi i\alpha\phi_L(x)),\qquad
 T(\alpha,\beta)=
 \sum_{x\in E}\exp(2\pi i(\alpha\phi_L(x)+\beta x)),
\]
put
\[
 F(\beta)=\int_0^1|T(\alpha,\beta)|^4|S(\alpha)|^4\,d\alpha.
\]
Then $F$ is nonnegative and has a finite Fourier expansion.  Since
$T(\alpha,0)=S(\alpha)$, expansion at $\beta=0$ gives
\[
 \int_0^1|S(\alpha)|^8\,d\alpha=F(0)=\sum_n\widehat F(n).
\]
The index $n=(x_1-y_1)+(x_2-y_2)$ ranges over only $O_C(N)$ integers, and
nonnegativity gives $|\widehat F(n)|\leq\widehat F(0)$.  Orthogonality in
both variables identifies
\[
 \widehat F(0)=\sum_m c_2(m)c'_2(-m).
\]
Consequently
\begin{equation}\label{eq:cubic-foliation}
 \int_0^1|S(\alpha)|^8d\alpha
 \ll_C N\sum_m c_2(m)c'_2(-m).
\end{equation}
By \eqref{eq:c2-bounds}--\eqref{eq:c2prime-mass},
\[
 \sum_m c_2(m)c'_2(-m)
 \ll_C A^4\Delta_3(C_1N^3).
\]
This proves \eqref{eq:uniform-moving-cubic}.  Notice that $L$ cancels
completely from \eqref{eq:restricted-cubic-factorization}; its size is used
only to keep the integers in the divisor estimates of order $O_C(N^3)$.
\end{proof}

\subsection{The four-point endpoint seed}

Fix once and for all $\gamma_0=1/20$.  For each fixed integer $t\geq1$,
choose constants
\begin{equation}\label{eq:prime-window-constants}
 0<a_1<b_1<a_2<b_2<\cdots<a_t<b_t,
 \qquad \sum_{i=1}^t b_i<\gamma_0.
\end{equation}
For fixed $t$, regard these window endpoints as fixed data; constants
subscripted by $t$ may depend on them.  Put
\begin{equation}\label{eq:prime-windows}
 \mathfrak P_i(h)=\{\ell\text{ prime}:a_ih\leq\ell\leq b_ih\}.
\end{equation}
The prime number theorem in fixed proportional intervals gives
\begin{equation}\label{eq:prime-tuple-count}
 \prod_{i=1}^t|\mathfrak P_i(h)|\asymp_t(h/\log h)^t.
\end{equation}

For $\boldsymbol\ell=(\ell_1,\ldots,\ell_t)$ in the product of these
sets, retain the four-point parameters from
\eqref{eq:UVbox}--\eqref{eq:fullsparams}, but restrict to
\begin{equation}\label{eq:endpoint-parameter-set}
 \mathcal P_h(\boldsymbol\ell)=
 \left\{(u,v,b,c)\in U\times V\times B\times C:
  (p,q)=1,\quad \ell_i\nmid p\ (1\leq i\leq t)\right\}.
\end{equation}
Here and below $p,q,r,s$ are the functions of $(u,v,b,c)$ in
\eqref{eq:fullsparams}.  To identify this concrete family with the abstract
amplifier notation, take
\[
 \begin{aligned}
  \Theta_h
   &=\{(u,v,b,c)\in U\times V\times B\times C:(p,q)=1\},\\
  B_{(u,v,b,c)}&=A_{p,q,r,s},\qquad
  \Theta_h(\boldsymbol\ell)&=\mathcal P_h(\boldsymbol\ell).
 \end{aligned}
\]
Starting with $A_0=A_{p,q,r,s}$, define recursively
\begin{equation}\label{eq:iterated-prime-suspension}
 A_i=\Sigma_{\ell_i,p}A_{i-1}
 =\ell_i\mathbin{\cdot}A_{i-1}\cup\{p\},\qquad(1\leq i\leq t).
\end{equation}
The point $p$ belongs to every $A_i$, and
$p=(\ell_i-1)0+p$ is a representation of length $\ell_i$ by old points.
Thus Theorem~\ref{thm:cyclic-suspension} and
Corollary~\ref{cor:iterated-suspension} apply at every step.

\begin{lemma}[Stable suspended chart]
\label{lem:stable-suspended-chart}
Uniformly in the prime tuple,
\begin{equation}\label{eq:endpoint-parameter-count}
 |\mathcal P_h(\boldsymbol\ell)|\gg_t h^4.
\end{equation}
Every $A_t$ in \eqref{eq:iterated-prime-suspension} is a primitive
$(4+t)$-set with
\begin{equation}\label{eq:suspended-geometry}
 r_h(A_t)=t+2,\qquad
 \ndiam(A_t)=\left(\prod_{i=1}^t\ell_i\right)(ps+qr)
 \asymp_t h^{t+2}.
\end{equation}
If
$p_1(A_t)\leq\cdots\leq p_{t+2}(A_t)$ are its successive relation-birth
degrees, then
\begin{equation}\label{eq:endpoint-critical-birth-scales}
 p_j(A_t)\asymp_t h\quad(1\leq j\leq t+2),
 \qquad
 \ndiam(A_t)\asymp_t\prod_{j=1}^{t+2}p_j(A_t).
\end{equation}
Its projective coordinate ring is arithmetically Cohen--Macaulay of type two.
Moreover, if
\begin{equation}\label{eq:JRS-definitions}
 R=\prod_{i=1}^t\ell_i,\qquad
 J=j_1+\cdots+j_t,\qquad 0\leq j_i<\ell_i,
\end{equation}
then
\begin{equation}\label{eq:suspended-chart-convolution}
 |hA_t|=
 \sum_{0\leq j_i<\ell_i}
 \bigl(M_{h-J}-M_{u-J}+M_{v-J}-M_{b-J}-M_{c-J}\bigr),
\end{equation}
where $M_x=\binom{x+3}{3}$ for $x\geq0$ and $M_x=0$ for $x<0$.
\end{lemma}

\begin{proof}
Lemma~\ref{lem:fullsprimitive} supplies $\gg h^4$ tuples with $(p,q)=1$.
For a fixed prime $\ell_i\gg_t h$, the interval of possible values of
$p=h-u$ contains only $O_t(1)$ multiples of $\ell_i$.  Removing these for
all $i$ discards only $O_t(h^3)$ quadruples, which proves
\eqref{eq:endpoint-parameter-count}.

The two active relations of the seed have degrees $p$ and $r+s$, by
\eqref{eq:f1} and \eqref{eq:f3}, and both degrees are at most $h$.
Every $\ell_i<h$, so the rank assertion follows from
\eqref{eq:suspension-rank}.  The box inequalities give
$ps+qr\asymp h^2$, and
\eqref{eq:iterated-suspension-diameter} proves the diameter assertion.
The seed toric ideal is minimally generated in degrees
$p$, $q+s=h-b$, and $r+s=h-c$.  The box bounds give, uniformly over all
retained parameter tuples,
\[
 \frac{63h}{80}\leq\min\{p,q+s,r+s\}
 \leq\max\{p,q+s,r+s\}\leq h.
\]
Hence the seed has no nonzero relation of degree below $63h/80$, while
$f_1$ and $f_3$ give two independent directions by degree $h$; its two
successive births lie in $[63h/80,h]$ and are therefore $\asymp h$.
The exact filtration in
\eqref{eq:suspension-filter} appends the births
$\ell_1,\ldots,\ell_t\asymp_t h$.  This proves the first assertion in
\eqref{eq:endpoint-critical-birth-scales}, and the second follows either
from \eqref{eq:suspended-geometry} or from
Corollary~\ref{cor:suspension-critical-core}.
The Hilbert--Burch resolution following Theorem~\ref{thm:family} has final
free rank two.  Projective ACM and type two are preserved by
Corollary~\ref{cor:iterated-suspension}.

Equation~\eqref{eq:iterated-suspension-profile} gives
\[
 |hA_t|=\sum_{0\leq j_i<\ell_i}|(h-J)A_0|,
\]
because every displayed $J$ is at most $h$.  More precisely,
\[
 0\leq J\leq\sum_i(\ell_i-1)<h/20\leq\min\{u,v,b,c\}.
\]
Thus all five residuals in
\eqref{eq:suspended-chart-convolution} are nonnegative.  Formula
\eqref{eq:Hfamily} at degree $h-J$ gives exactly its summand; the sixth
Hilbert-numerator term still vanishes because $q+r+s>h$.
\end{proof}

The next identity is the reason the whole iteration remains a cubic
problem.  The factorial normalization $2^3\,3!=48$ turns a tetrahedral
number into a centered odd cubic.

\begin{lemma}[Factorial centered-cubic label]
\label{lem:centered-cubic-label}
Put
\begin{equation}\label{eq:S1S2YL-definitions}
 S_1=\sum_{i=1}^t(\ell_i-1),\qquad
 S_2=\sum_{i=1}^t(\ell_i^2-1),\qquad
 Y_x=2x+4-S_1,\qquad L=S_2-4,
\end{equation}
and define
\begin{equation}\label{eq:moving-cubic-polynomial}
 P_{\boldsymbol\ell}(x)=Y_x^3+LY_x.
\end{equation}
Then
\begin{equation}\label{eq:averaged-centered-cubic}
 48\sum_{0\leq j_i<\ell_i}M_{x-J}
 =R P_{\boldsymbol\ell}(x)
\end{equation}
whenever all $x-J\geq0$.  Consequently every set in
Lemma~\ref{lem:stable-suspended-chart} satisfies the exact integer identity
\begin{equation}\label{eq:factorial-prime-label}
 48|hA_t|=R\,G_{\boldsymbol\ell}(u,v,b,c),
\end{equation}
where
\begin{equation}\label{eq:G-prime-label}
 G_{\boldsymbol\ell}(u,v,b,c)=
 P_{\boldsymbol\ell}(h)-P_{\boldsymbol\ell}(u)
 +P_{\boldsymbol\ell}(v)-P_{\boldsymbol\ell}(b)
 -P_{\boldsymbol\ell}(c)\in\mathbb Z_{>0}.
\end{equation}
Uniformly on the construction,
\begin{equation}\label{eq:endpoint-macroscopic-window}
 |hA_t|\asymp_t Rh^3\asymp_t h^{t+3}.
\end{equation}
\end{lemma}

\begin{proof}
The elementary identity
\begin{equation}\label{eq:centered-tetrahedral-48}
 48M_n=(2n+4)^3-4(2n+4)
\end{equation}
is the centered form of $M_n=(n+1)(n+2)(n+3)/3!$.  Let the $j_i$ be
independent and uniform on $\{0,\ldots,\ell_i-1\}$, and put
$Z=2J-S_1$.  The distribution of $Z$ is symmetric, and
\[
 \mathbb EZ=\mathbb EZ^3=0,\qquad
 \mathbb EZ^2=\frac13\sum_i(\ell_i^2-1)=\frac{S_2}{3}.
\]
Since $2(x-J)+4=Y_x-Z$, averaging
\eqref{eq:centered-tetrahedral-48} gives
\[
 \mathbb E\bigl((Y_x-Z)^3-4(Y_x-Z)\bigr)
 =Y_x^3+(S_2-4)Y_x.
\]
Multiplication by the number $R$ of tuples $(j_1,\ldots,j_t)$ proves
\eqref{eq:averaged-centered-cubic}; inserting it into
\eqref{eq:suspended-chart-convolution} proves
\eqref{eq:factorial-prime-label}.

For the final assertion, each summand $|(h-J)A_0|$ in the residue
convolution is $\asymp h^3$, uniformly on the fixed box and on
$0\leq J<h/20$.  Indeed, in
\eqref{eq:suspended-chart-convolution} the leading term has
$h-J\geq19h/20$, whereas the three negative residuals satisfy
$u\leq9h/80$ and $b,c\leq17h/80$; the positive $M_{v-J}$ may be discarded
for a lower bound.  There are exactly $R$ summands, and
$R\asymp_t h^t$ by the prime windows.
\end{proof}

\begin{lemma}[Prime-energy stability of the four-point chart]
\label{lem:four-point-prime-energy-stability}
The four-point seed is a shift-stable $3$-chart of depth $\gamma_0h$.
For every fixed $t\in\mathbb Z_{\geq0}$ and every fixed system of prime
windows satisfying \eqref{eq:prime-window-constants},
the suspended four-point chart satisfies, uniformly in the prime tuple,
\begin{equation}\label{eq:four-point-prime-energy-stability}
 N_h\gg_t h^4,
 \qquad
 E_h\leq h^5\exp\!\left(O_t\!\left(
  \frac{\log h}{\log\log h}\right)\right).
\end{equation}
In particular it is prime-energy stable with exponents $(a,\mu)=(4,5)$.
\end{lemma}

\begin{proof}
The Hilbert formula \eqref{eq:Hfamily} and the box bounds
\eqref{eq:UVbox}--\eqref{eq:BCbox} show that the five retained residual
indices stay nonnegative for $0\leq J\leq\gamma_0h$, while the sixth lower
index in \eqref{eq:Hfamily} remains negative.  Thus the seed is
shift-stable at the asserted fixed depth.
For $t=0$, this is the original four-cubic energy estimate in the proof of
Theorem~\ref{thm:main}.  Assume $t\geq1$.  The parameter estimate is
\eqref{eq:endpoint-parameter-count}.  For fixed
$\boldsymbol\ell$, the constant term
$P_{\boldsymbol\ell}(h)$ does not affect collisions.  The maps
$x\mapsto Y_x$ send each of $U,V,B,C$ injectively to a length-$O(h)$
subset of one parity class in $[-C_th,C_th]$, and
$|L|=|S_2-4|\ll_th^2$.  Orthogonality, followed by H\"older, bounds the
collision count of $G_{\boldsymbol\ell}$ on the full product box by
\begin{equation}\label{eq:fixed-label-energy}
 \prod_{X\in\{U,V,B,C\}}
 \left(
  \int_0^1\left|
   \sum_{x\in X}
   \exp\bigl(2\pi i\alpha P_{\boldsymbol\ell}(x)\bigr)
  \right|^8d\alpha
 \right)^{1/4}.
\end{equation}
Lemma~\ref{lem:uniform-moving-cubic} shows that
\eqref{eq:fixed-label-energy} is at most
\[
 h^5\exp\!\left(O_t\!\left(
  \frac{\log h}{\log\log h}\right)\right).
\]
The collision count on the restricted set
$\mathcal P_h(\boldsymbol\ell)$ is no larger, proving
\eqref{eq:four-point-prime-energy-stability}.  The bound is uniform for every
fixed suspension depth and every fixed window system satisfying
\eqref{eq:prime-window-constants}.  Since $\gamma_0$ was fixed before the
window system, the preceding uniformity is exactly the
prime-energy-stability requirement in
Definition~\ref{def:prime-energy-stability}.
\end{proof}

\begin{proof}[Proof of Theorem~\ref{thm:endpoint-spectrum}]
The case $k=4$ is the four-point construction proving
Theorem~\ref{thm:main}.  Its fixed window is macroscopic on the $h^3$
scale, and the primitive chart points have active rank two.  In the notation
of Lemma~\ref{lem:fullschart}, their diameter is
$\ndiam(A_0)=ps+qr\asymp h^2$.  The Hilbert--Burch resolution shows that their
projective ACM coordinate rings have type two.  As in the proof of
Lemma~\ref{lem:stable-suspended-chart}, the three minimal generator degrees
are $\asymp h$, so both successive relation births are $\asymp h$ and
$\ndiam(A_0)\asymp p_1(A_0)p_2(A_0)$.  Suppose now that $k\geq5$,
and set $t=k-4$.

The fixed choice of $\gamma_0$ and \eqref{eq:prime-window-constants} make
Definition~\ref{def:shift-stable-chart} apply with $d=3$ and depth
$\gamma_0h$.  Then
Lemma~\ref{lem:four-point-prime-energy-stability} makes the four-point family
a prime-energy-stable endpoint seed with
\begin{equation}\label{eq:four-point-endpoint-triple}
 (d,a,\mu)=(3,4,5),\qquad 2a-\mu=d.
\end{equation}
The quantitative conclusion of
Theorem~\ref{thm:prime-filter-amplification}, with $t=k-4$, gives
\begin{align*}
 |\Rcal(h,k)|
 &\geq
 (h/\log h)^t h^3
 \exp\!\left(-O_k\!\left(
  \frac{\log h}{\log\log h}\right)\right)\\
 &=h^{k-1}
 \exp\!\left(-O_k\!\left(
   \frac{\log h}{\log\log h}\right)\right).
\end{align*}
The localization follows from \eqref{eq:endpoint-macroscopic-window};
primitivity, maximal active rank, diameter, projective
Cohen--Macaulayness, and type two follow from
Lemma~\ref{lem:stable-suspended-chart}.
Finally, the elementary ambient bound
$|\Rcal(h,k)|\leq\binom{h+k-1}{k-1}$, together with the lower bound just
proved, gives \eqref{eq:ambient-exponent}.
\end{proof}

\begin{proposition}[Fixed-scale divisor barrier for the suspended endpoint chart]
\label{prop:stable-factor-barrier}
Fix $k\geq5$, put $t=k-4$, and let $\mathcal L^{\mathrm{st}}_{h,k}$ be the
following explicitly defined set of labels.  Choose
$(u,v,b,c)\in U\times V\times B\times C$, define $p,q,r,s$ by
\eqref{eq:fullsparams}, require $(p,q)=1$, and choose integer degrees
$a_i h\leq\ell_i\leq b_i h$ with $\gcd(p,\ell_i)=1$ for
$1\leq i\leq t$, where the fixed windows satisfy
\eqref{eq:prime-window-constants}.  Starting from
$A_0=A_{p,q,r,s}$, define
\begin{equation}\label{eq:stable-integer-suspension}
 A_i=\Sigma_{\ell_i,p}A_{i-1}
 =\ell_i\mathbin{\cdot}A_{i-1}\cup\{p\}
 \qquad(1\leq i\leq t),
\end{equation}
and let $\mathcal L^{\mathrm{st}}_{h,k}$ be the set of all resulting labels
$|hA_t|$.  Put
\[
 \delta=1-\frac{1+\log\log2}{\log2}=0.086071\ldots.
\]
Then
\begin{equation}\label{eq:stable-factor-barrier}
 |\mathcal L^{\mathrm{st}}_{h,k}|
 \ll_k
 \frac{h^{k-1}}
 {(\log h)^\delta(\log\log h)^{3/2}}
 =o_k(h^{k-1}).
\end{equation}
Thus no refinement of collision estimates alone can give positive ambient
density from this stable full-residue chart.  This is a limitation of the
chart, not an upper bound for the full spectrum $\Rcal(h,k)$.
\end{proposition}

\begin{proof}
For $N=|hA_t|\in\mathcal L^{\mathrm{st}}_{h,k}$, put
\[
 R=\prod_{i=1}^t\ell_i,
 \qquad G=G_{\boldsymbol\ell}(u,v,b,c)
\]
with $G_{\boldsymbol\ell}$ as in \eqref{eq:G-prime-label}.  The window bound
gives
\[
 0\leq\sum_i(\ell_i-1)<h/20\leq\min\{u,v,b,c\},
\]
so the complete residue box used in
\eqref{eq:suspended-chart-convolution} is valid.  The residue convolution and
the calculation in the proof of Lemma~\ref{lem:centered-cubic-label}---which
use only these window and coprimality conditions, not primality---therefore give
\[
 48N=R G,\qquad R\asymp_k h^t,\qquad G\asymp_k h^3,\qquad
 N\asymp_k h^{t+3}.
\]
Set $\sigma=\min(t,3)$ and $d_*=\min(R,G)$.  For suitable constants depending
only on $k$, one has $c_kh^\sigma\leq d_*\leq C_kh^\sigma$, while
$d_*\mid48N$ and $d_*\leq\sqrt{48N}$.  Cover this fixed-ratio interval by
$O_k(1)$ intervals $(y,2y]$.  Ford's divisor-in-an-interval estimate
\cite[Corollary~2]{Ford2008DivisorInterval}, with $c=2$ in Ford's notation, requires
$1\leq y\leq X/2$.  For $X\geq4$, the range $3\leq y\leq\sqrt X$ implies
these hypotheses and has $Y=\min\{y,X/y\}+3=y+3$; hence its logarithmic
factors are comparable with those displayed below.  The corollary gives,
uniformly for
$3\leq y\leq\sqrt X$,
\[
 \#\{m\leq X:m\text{ has a divisor in }(y,2y]\}
 \ll \frac{X}{(\log y)^\delta(\log\log y)^{3/2}}.
\]
Apply this with $m=48N$ and $X=C'_kh^{t+3}$, where $C'_k$ is chosen so
that $48N\leq X$ uniformly.  Since
$\log y\asymp_k\log h$ and $N\mapsto48N$ is injective, summing over the
finitely many intervals proves \eqref{eq:stable-factor-barrier}.
The proof uses only fixed proportional degree scales and the exact
full-residue factorization.
\end{proof}

\begin{corollary}[Prime toric ACM almost-complete-intersection spectrum]
\label{cor:endpoint-acm-type-two}
Fix $k\geq4$.  The lower bound in \eqref{eq:endpoint-local} already counts
distinct degree-$h$ Hilbert values of arithmetically Cohen--Macaulay
projective monomial curves in $\mathbb P^{k-1}$ of Cohen--Macaulay type two.
More precisely, for $t=k-4$ and with $\xi$ recording homological degree, the
witnesses have graded Betti polynomial
\begin{equation}\label{eq:endpoint-betti-tower}
 \begin{split}
 \mathcal B_{A_t}(\xi,z)
 ={}&\bigl(
  1+\xi(z^p+z^{q+s}+z^{r+s})
  +\xi^2(z^{p+s}+z^{q+r+s})
 \bigr)\\
 &\hspace{18mm}\cdot
 \prod_{i=1}^t(1+\xi z^{\ell_i}).
 \end{split}
\end{equation}
Writing $\beta_i(A_t)=\sum_j\beta^{A_t}_{i,j}$, their total Betti polynomial
depends only on $k$:
\begin{equation}\label{eq:endpoint-total-Betti}
 \begin{aligned}
  \sum_{i=0}^{k-2}\beta_i(A_t)\xi^i
  &=\mathcal B_{A_t}(\xi,1)=(1+2\xi)(1+\xi)^{k-3},\\
  \beta_i(A_t)&=\binom{k-3}{i}+2\binom{k-3}{i-1}.
 \end{aligned}
\end{equation}
Here out-of-range binomial coefficients are zero.  In particular, the
homogeneous prime toric ideal $I_{A_t}$ has
\[
 \operatorname{ht}I_{A_t}=k-2,
 \qquad
 \mu(I_{A_t})=k-1=\operatorname{ht}I_{A_t}+1,
\]
so it is an almost complete intersection.  Thus the witnesses have ACM
coordinate rings of type two.  Their defining ideals are prime toric almost
complete intersections, all with the same total Betti format.  As
$h\to\infty$, they realize $h^{k-1+o_k(1)}$ distinct degree-$h$ Hilbert
values; at $k=4$, they realize a positive proportion of the ambient interval.
The internal degree shifts in
\eqref{eq:endpoint-betti-tower} vary with the parameters.
Their Hilbert numerator is obtained by setting $\xi=-1$ in
\eqref{eq:endpoint-betti-tower}; equivalently,
\begin{equation}\label{eq:endpoint-Hilbert-numerator-tower}
 \HS_{R_{A_t}}(z)=
 \frac{
  (1-z^p-z^{q+s}-z^{r+s}+z^{p+s}+z^{q+r+s})
  \prod_{i=1}^t(1-z^{\ell_i})}
 {(1-z)^k}.
\end{equation}
\end{corollary}

\begin{proof}
The seed Betti polynomial is the Hilbert--Burch resolution following
Theorem~\ref{thm:family}.  Formula
\eqref{eq:iterated-suspension-Betti} gives
\eqref{eq:endpoint-betti-tower}; setting $\xi=-1$ gives
\eqref{eq:endpoint-Hilbert-numerator-tower}.  The top free module has rank
two, so every displayed Cohen--Macaulay coordinate ring has type two.
Setting $z=1$ in \eqref{eq:endpoint-betti-tower} gives
\eqref{eq:endpoint-total-Betti}.  The projective semigroup map has domain a
polynomial ring in $k$ variables and image a two-dimensional domain, so its
homogeneous kernel $I_{A_t}$ is prime of height $k-2$.  Minimality of the
displayed resolution gives
$\mu(I_{A_t})=\beta_1(A_t)=k-1$, proving the almost-complete-intersection
assertion.
Finally, \eqref{eq:eliasbridge} identifies its degree-$h$ Hilbert value with
$|hA_t|$.  The endpoint lower bound and the ambient upper bound give the
stated exponent, while Theorem~\ref{thm:main} gives positive density when
$k=4$.
\end{proof}

\subsection{A complete-intersection comparison tower}

The same amplifier also yields a quantitative complete-intersection tower
from the three-point construction.  This comparison distinguishes the
type-one hypersurface seed from the type-two Hilbert--Burch seed above.

\begin{corollary}[Complete-intersection spectrum tower]
\label{cor:complete-intersection-spectrum-tower}
Fix $k\geq3$, and let $\Rcal_{\mathrm{CI}}(h,k)$ be the set of labels in
$\Rcal(h,k)$ having a primitive witness whose projective monomial curve is a
complete intersection.  There are constants $0<c_k<C_k$ such that, for all
sufficiently large integers $h$,
\begin{equation}\label{eq:complete-intersection-spectrum-tower}
 \left|\Rcal_{\mathrm{CI}}(h,k)
 \cap[c_kh^{k-1},C_kh^{k-1}]\right|
 \gg_k \frac{h^{k-2}}{(\log h)^{k-3}}.
\end{equation}
These witnesses are Gorenstein, have maximal active rank $k-2$, and satisfy
\begin{equation}\label{eq:complete-intersection-critical-core}
 \ndiam(A)=\prod_{i=1}^{k-2}p_i(A)\asymp_k h^{k-2}.
\end{equation}
With $\xi$ again recording homological degree, their Betti and Hilbert
numerators are respectively
\begin{equation}\label{eq:complete-intersection-tower-Betti}
 \prod_{e\in\{p,\ell_1,\ldots,\ell_{k-3}\}}(1+\xi z^e),
 \qquad
 \prod_{e\in\{p,\ell_1,\ldots,\ell_{k-3}\}}(1-z^e).
\end{equation}
\end{corollary}

\begin{proof}
Take
\[
 B_p=\{0,1,p\},\qquad \frac h4\leq p\leq\frac h2.
\]
Its projective toric ideal is generated by
$x_1^p-x_0^{p-1}x_2$, and for $n\geq p$,
\begin{equation}\label{eq:three-point-shift-stable-profile}
 |nB_p|=\binom{n+2}{2}-\binom{n-p+2}{2}
 =\frac{p(2n-p+3)}2.
\end{equation}
This is a shift-stable $d=2$ chart of depth $h/4$, with
$N_h\asymp h$.  Uniformly in this chart,
\begin{equation}\label{eq:three-point-macroscopic-chart}
 \frac18h^2\leq |(h-J)B_p|\leq h^2
 \qquad(0\leq J\leq h/4),
\end{equation}
for all sufficiently large $h$.  Choose $t=k-3$ disjoint
prime windows whose upper endpoints sum to less than $1/4$, and use $c=1$
at every suspension.  If $R=\prod_i\ell_i$ and
$S=\sum_i(\ell_i-1)$, the exact profile transform gives
\begin{equation}\label{eq:three-point-amplified-label}
 V_{p,\boldsymbol\ell}(h)
 =\frac{pR(2h-p+3-S)}2.
\end{equation}
Moreover,
\[
 V_{p+1,\boldsymbol\ell}(h)-V_{p,\boldsymbol\ell}(h)
 =R\left(h+1-p-\frac S2\right)>0.
\]
Thus the fixed-prime-tuple map is injective and $E_h=N_h\asymp h$.
Theorem~\ref{thm:prime-filter-amplification} now gives
\eqref{eq:complete-intersection-spectrum-tower} and its localization.

The seed has one relation born at degree $p$ and diameter $p$.  Each visible
suspension appends the birth $\ell_i$ and multiplies the diameter by
$\ell_i$.  Hence the final rank is $1+t=k-2$ and
\eqref{eq:complete-intersection-critical-core} follows exactly from
Corollary~\ref{cor:suspension-critical-core}.  Finally, the seed is a
hypersurface, and each suspension adjoins a regular equation.  Thus the
resulting coordinate rings are complete intersections and hence Gorenstein;
the mapping-cone formula gives
\eqref{eq:complete-intersection-tower-Betti}.
\end{proof}

\begin{remark}[A homological comparison, not a classification]
\label{rem:homological-comparison-hierarchy}
The complete-intersection seed produces
$h^{k-2-o(1)}$ macroscopic labels with type-one witnesses, one power of $h$
below the ambient label exponent.  The four-point Hilbert--Burch seed attains
the full ambient label exponent.  It produces $h^{k-1-o(1)}$ labels with
type-two witnesses.  This comparison concerns two explicit realization
classes stable under suspension.  It neither asserts that their label sets
are disjoint nor classifies their full fibres or excludes witnesses of other
homological types for the same integer.
\end{remark}

\part{Finite-observation geometry, compression, and frequency}
\label{part:structure}

The endpoint construction has maximal active rank.  We now place the same
marked family in a general finite-observation framework that simultaneously
accounts for its compression scale, type diversity, shape codimension, and
realization frequency.

\section{Finite additive observation and the root-sweep polytope}
\label{sec:root-sweep-dictionary}

The entire bounded addition table is encoded by a canonical polytope.  Put
\[
 E_{h,k}=\mathcal U_{h,k}
 =\left\{u\in\mathbb Z_{\geq0}^k:\sum_{i=1}^ku_i=h\right\},
 \qquad
 \mathsf A_{k-1}=\left\{z\in\mathbb Z^k:\sum_{i=1}^kz_i=0\right\},
\]
and, for $z\in(\mathsf A_{k-1})_{\mathbb R}$, put
\[
 \rho(z)=\sum_i z_i^+=\sum_i z_i^-=\frac12\lVert z\rVert_1,
 \qquad
 \mathsf R_{k-1}=\{z\in(\mathsf A_{k-1})_{\mathbb R}:\rho(z)\leq1\}.
\]
For an ordered real tuple $a=(a_1<\cdots<a_k)$, set
\[
 L_h(a)=\operatorname{span}_{\mathbb R}
 \{z\in\mathsf A_{k-1}:z\cdot a=0,\ \rho(z)\leq h\},
 \qquad r_h(a)=\dim L_h(a).
\]
Choose one vector from each primitive pair $\{z,-z\}$ and define the
\emph{root-sweep zonotope}
\begin{equation}\label{eq:root-sweep-zonotope}
 \mathfrak Z_{h,k}=
 \sum_{\substack{z\in\mathsf A_{k-1}\ \mathrm{primitive}/\{\pm1\}\\
                  \rho(z)\leq h}}[-z,z].
\end{equation}

\begin{proposition}[Hilbert-labelled arithmetic sweeps]
\label{prop:root-sweep-dictionary}
Let $h\geq1$ and $k\geq2$.  One has the exact difference identity
\begin{equation}\label{eq:root-sweep-difference}
 E_{h,k}-E_{h,k}=h\mathsf R_{k-1}\cap\mathsf A_{k-1}.
\end{equation}
For an ordered real tuple $a=(a_1<\cdots<a_k)$, write
\({\bf1}=(1,\ldots,1)\in\mathbb R^k\) and let
$a^\circ=a-k^{-1}(\sum_i a_i){\bf1}$ be its centered representative.  Let
$F_h(a)$ be the face of $\mathfrak Z_{h,k}$ exposed by $a^\circ$.  The
following data determine one another:
\begin{enumerate}[label=\textup{(\roman*)}]
\item the total preorder on $E_{h,k}$ induced by $u\mapsto u\cdot a$
      (equivalently, its ordered weak partition into level sets);
\item the complete ordered $h$-addition table of $a$, including ties;
\item the face $F_h(a)$.
\end{enumerate}
The dimension of $F_h(a)$ is the active relation rank $r_h(a)$.  In the
ambient space ${\bf1}^{\perp}$, its normal cone has dimension
$k-1-r_h(a)$, and the relative interior of that cone consists precisely of
the centered weights inducing the same complete table.  Moreover, the same
face determines every ordered addition table in degrees $1\leq j\leq h$.

Fix any field $K$.  If $a\in\mathbb Z^k$, put
$A=\{a_1,\ldots,a_k\}$ and define
\[
 \varphi_a:K[x_1,\ldots,x_k]\longrightarrow K[u,t,t^{-1}],
 \qquad x_i\longmapsto u t^{a_i},
 \qquad I_a=\ker\varphi_a,
\]
where $\deg u=1$ and $\deg t=0$.  Then the number of blocks in the face
label is the Hilbert value
\begin{equation}\label{eq:root-sweep-Hilbert-label}
 |hA|=\dim_K\bigl(K[x_1,\ldots,x_k]/I_a\bigr)_h.
\end{equation}
Thus the face records the equality rank and Hilbert/block label, while its
normal cone encodes the normalized shape, realization frequency, and integer
compression developed below.
\end{proposition}

\begin{proof}
If $z=u-v$ with $u,v\in E_{h,k}$, then $\rho(z)\leq h$.  Conversely, if
$z\in\mathsf A_{k-1}$ and $\rho(z)\leq h$, choose
$w\in\mathbb Z_{\geq0}^k$ of total mass $h-\rho(z)$ and write
\[
 z=(z^++w)-(z^-+w).
\]
This proves \eqref{eq:root-sweep-difference}.

A sweep of a point configuration is the ordered weak partition induced by a
linear functional.  Padrol and Philippe identify the sweep poset with the
opposite face poset of the sweep polytope, whose normal fan is the arrangement
orthogonal to the point differences \cite[Sections~2.1 and~2.3, especially
Definition~2.1 and Proposition~2.2]{PadrolPhilippe2024}.  Repeated and
nonprimitive parallel
generators do not change that fan, so \eqref{eq:root-sweep-difference} gives
the normally equivalent reduced zonotope \eqref{eq:root-sweep-zonotope}.  On
a zonotope face exposed by $a$, precisely the segments with $z\cdot a=0$
remain free; their span has dimension $r_h(a)$.  The usual
face--normal-cone correspondence gives the asserted description and dimension
of the normal cone.  For $j\leq h$, the padding map
\[
 E_{j,k}\longrightarrow E_{h,k},\qquad
 u\longmapsto u+(h-j)e_1,
\]
adds the same constant $(h-j)a_1$ to every sum and recovers the degree-$j$
table.  Finally, two degree-$h$ monomials have the same image under
$\varphi_a$ exactly when their exponent vectors yield the same $h$-sum,
which proves \eqref{eq:root-sweep-Hilbert-label}.
\end{proof}

The abstract sweep--face correspondence is due to Padrol--Philippe.  In this
specific type-$A$ configuration, Vietri studied strict fixed-degree monomial
weight orders and their lower-degree restrictions \cite{Vietri2002}, while
Bogart--Cu\'ellar described the corresponding strict $B_h$ comparison
arrangement in gap coordinates and counted the resulting rulers by
inside-out Ehrhart theory \cite{BogartCuellar2025}.  Moreover, after comparing
the root gauge with the Euclidean norm, the asymptotic orders of the
unlabelled cell and face counts follow from
B\'ar\'any--Bureaux--Lund, Theorem~5.6
\cite{BaranyBureauxLund2018}.  None of those unlabelled or strict-region facts
is claimed as new here.  The arithmetic refinement below instead retains, on
the same face and normal cone, ties, equality rank and its birth filtration,
the Hilbert/block label, the fixed-face frequency law, and the filtered
integral normal.

\section{Active-relation compression and the sharp core}
\label{sec:compression-master}

At the equality level, finite torsion-free fragments admit classical Freiman
models in $\mathbb Z$
\cite[Part~II, Lemma~2.3.4]{GeroldingerRuzsa2009}.  The qualitative existence
of integer models for finite real addition tables is also classical; see
\cite{NathansonMSTD2018,OBryant2025}.  The use of rational cones, extreme
rays, and Cramer's rule---including the generic factorial loss in the
determinant bound---has precedent in the order-preserving Freiman models of
Amirkhanyan, Bush, and Croot \cite{AmirkhanyanBushCroot2018}.  Here we obtain
a whole-set comparison model that treats equality faces of every rank,
together with a rank- and degree-sensitive interval whose universal
$h$-exponent is optimal for fixed $k$.  The sharp constant in the
maximal-rank core comes from the totally unimodular type-$A$
elementary-transfer configuration.
Nathanson's published universal interval has order $O_k(h^{k-1})$
\cite{NathansonCompression}.  Our theorem determines the optimal
$h$-exponent for fixed $k$ while preserving the complete comparison pattern,
including equality faces.

Put
\[
 \mathcal D_{h,k}=\left\{z\in\mathbb Z^k\setminus\{0\}:
       \sum_i z_i=0,\quad \rho(z):=\sum_i z_i^+\leq h\right\}.
\]
This is precisely the set of nonzero differences, after cancellation and
padding, of exponent vectors of degree at most $h$.

\begin{lemma}[Transfer minors]\label{lem:transfer-minors-master}
Let $1\leq s\leq k$, and let
$z^{(1)},\ldots,z^{(s)}\in\mathbb Z^k$ have coordinate sum zero.  Delete
the same $k-s$ coordinates from every row.  If $Q$ is the resulting
$s\times s$ matrix, then
\begin{equation}\label{eq:transfer-minors-master}
 |\det Q|\leq\prod_{i=1}^s\rho(z^{(i)}).
\end{equation}
In particular, if every $z^{(i)}\in\mathcal D_{h,k}$, then
$|\det Q|\leq h^s$.
\end{lemma}

\begin{proof}
A zero-sum integral vector $z$ is a sum of $\rho(z)$ elementary transfers
$e_p-e_q$.  After the fixed coordinates are deleted, a transfer is one of
\[
 0,\qquad \pm e_i,\qquad e_i-e_j.
\]
These are rows of a reduced directed incidence matrix.  Hence every square
matrix formed from them has determinant $0$ or $\pm1$: a nonzero determinant
corresponds to a forest rooted at the deleted vertices, and leaf elimination
gives absolute determinant one.  Expanding the determinant multilinearly in
transfer decompositions of the rows gives at most
$\prod_i\rho(z^{(i)})$ nonzero terms, each of absolute value one.
\end{proof}

\begin{theorem}[Optimal-order complete models and the sharp active core]
\label{thm:compression-master}
Let $h\geq1$, $k\geq2$, and let
$A=(a_0<a_1<\cdots<a_{k-1})$ be an ordered real $k$-set.  If
$r=r_h(A)$, define its successive active-relation degrees by
\[
 p_i(A)=\min\{p\in\mathbb Z_{\geq1}:\dim L_p(A)\geq i\},
 \qquad 1\leq i\leq r.
\]
Put $d_A=k-1-r$.  Since $r\leq k-2$, one has $d_A\geq1$.
Then there is an integer set
\[
 \begin{aligned}
 A^*&=(0=a_0^*<a_1^*<\cdots<a_{k-1}^*),\\
 a_{k-1}^*
 &\leq d_Ah^{d_A-1}\prod_{i=1}^{r}p_i(A)
 \leq(k-1-r)h^{k-2},
 \end{aligned}
\]
where the empty product is one,
such that, simultaneously for every $1\leq j\leq h$ and every
$u,v\in\mathcal U_{j,k}$,
\begin{equation}\label{eq:all-order-comparison}
 \operatorname{sgn}\bigl((u-v)\cdot A^*\bigr)
 =\operatorname{sgn}\bigl((u-v)\cdot A\bigr).
\end{equation}
Thus the ordered addition tables are identical in every degree $1\leq j\leq h$;
in particular, $|jA^*|=|jA|$ for $1\leq j\leq h$.

There is a complementary intrinsic statement.  Suppose $k\geq3$ and
\[
 B=(0=b_0<b_1<\cdots<b_{k-1}=d),
 \qquad \gcd(b_1,\ldots,b_{k-1})=1.
\]
For $1\leq i\leq k-2$, define the successive relation degrees
\[
 p_i(B)=\min\{p\in\mathbb Z_{\geq1}:\dim L_p(B)\geq i\}.
\]
Then
\begin{equation}\label{eq:successive-product-master}
 \ndiam(B)=d\leq\prod_{i=1}^{k-2}p_i(B).
\end{equation}
Consequently, for $h\geq2$,
\begin{equation}\label{eq:sharp-maximal-rank-core}
 \max_{\substack{B\subset\mathbb Z,\ |B|=k\\r_h(B)=k-2}}
       \ndiam(B)=h^{k-2}.
\end{equation}
\end{theorem}

\begin{proof}
We first prove the intrinsic product law.  Choose independent homogeneous
relations $z^{(1)},\ldots,z^{(k-2)}$ adapted to the filtration by relation
degree, so that $\rho(z^{(i)})\leq p_i(B)$.  Project away coordinate zero.
The projected vectors remain independent and lie in the orthogonal complement
of the primitive vector $b=(b_1,\ldots,b_{k-1})$.  Their integral exterior
product therefore has Hodge dual $mb$ for some nonzero integer $m$.  The
coordinate complementary to $b_{k-1}=d$ is a minor obtained by deleting
coordinates $0$ and $k-1$.  Lemma~\ref{lem:transfer-minors-master} gives
\[
 d\leq |m|d
 \leq\prod_{i=1}^{k-2}\rho(z^{(i)})
 \leq\prod_{i=1}^{k-2}p_i(B),
\]
 which proves \eqref{eq:successive-product-master}.  Given an arbitrary
 integer set, translate it and divide all its elements by the gcd of its
 differences before applying this argument; neither its active relations nor
 its normalized diameter changes.  If $r_h(B)=k-2$, all the successive
 degrees are at most $h$, so $d\leq h^{k-2}$.

For equality in \eqref{eq:sharp-maximal-rank-core}, take
\[
 P_{h,k}=\{0,1,h,h^2,\ldots,h^{k-2}\}.
\]
With coordinates indexed from $0$ to $k-1$, the $k-2$ vectors
\[
 h e_i-e_{i+1}-(h-1)e_0,
 \qquad 1\leq i\leq k-2,
\]
are independent homogeneous relations of degree $h$.  Hence
$r_h(P_{h,k})=k-2$---the full homogeneous relation space has dimension
$k-2$---and $\ndiam(P_{h,k})=h^{k-2}$.

We turn to bounded models.  Translate $A$ so that $a_0=0$ and regard
$a=(0,a_1,\ldots,a_{k-1})$ as a point of
$E=\{x\in\mathbb R^k:x_0=0\}$.  Its sign cone is
\[
 C(a)=\left\{x\in E:
 \operatorname{sgn}(z\cdot x)=\operatorname{sgn}(z\cdot a)
 \text{ for every }z\in\mathcal D_{h,k}\right\},
\]
where a zero sign denotes an equality and a nonzero sign a strict signed
inequality.  Let $K(a)=\overline{C(a)}$, obtained by making the strict
inequalities weak.  This really is the closure: if $x$ satisfies the weak
system, then $x+\varepsilon a\in C(a)$ for every $\varepsilon>0$.

The cone $K(a)$ is rational and pointed.  Indeed, $e_j-e_i$ belongs to
$\mathcal D_{h,k}$ for $i<j$, after padding, so its ordered closure satisfies
\[
 0=x_0\leq x_1\leq\cdots\leq x_{k-1};
\]
if both $x$ and $-x$ lie in $K(a)$, then $x=0$.  Its equality normals are
exactly the active relations of $A$.  Put
\[
 L=E\cap\bigcap_{\substack{z\in\mathcal D_{h,k}\\z\cdot a=0}}\ker z.
\]
Then $K(a)\subseteq L$.  Every remaining defining functional is strict at
$a$, and there are only finitely many of them, so $a$ is a relative-interior
point of $K(a)$ in $L$.  Hence $\operatorname{span}K(a)=L$, and
\begin{equation}\label{eq:sign-cone-dimension}
 d_A:=\dim\operatorname{span}K(a)=k-1-r_h(A).
\end{equation}

Choose active equality normals $z^{(1)},\ldots,z^{(r)}$ adapted to their
degree filtration, so that $\rho(z^{(i)})\leq p_i(A)$.  Consider a primitive
integral generator $v$ of an extreme ray of $K(a)$.  In
$E\cong\mathbb R^{k-1}$, the equality normals together with $d_A-1$ tight
inequality normals from $\mathcal D_{h,k}$ contain $k-2$ independent rows:
an extreme ray has relative codimension $d_A-1$ in $K(a)$.
The cofactor vector of these rows is an integral ray generator.  Each
coordinate is a minor obtained by deleting coordinate zero and one further
coordinate, so Lemma~\ref{lem:transfer-minors-master} bounds it by
\[
 H_A:=h^{d_A-1}\prod_{i=1}^r p_i(A).
\]
Dividing by the gcd to obtain the primitive generator only decreases the
coordinates, and the ordered closure gives
\begin{equation}\label{eq:bounded-rays-master}
 0=v_0\leq v_1\leq\cdots\leq v_{k-1}\leq H_A.
\end{equation}
For $k=2$, the same statement uses the empty determinant, equal to one.

Choose $d_A$ linearly independent extreme rays with primitive generators
$v^{(1)},\ldots,v^{(d_A)}$ and put
$a^*=\sum_{i=1}^{d_A}v^{(i)}$.  This sum lies in the relative interior of
$K(a)$.  Indeed, if a signed defining functional vanished on the sum, its
nonnegative value would vanish on every chosen ray; because those rays span
$\operatorname{span}K(a)$, the functional would vanish there, contrary to
its strict value at $a$.  Thus $a^*\in C(a)$, and
\eqref{eq:bounded-rays-master} gives
\[
 0=a_0^*<a_1^*<\cdots<a_{k-1}^*
 \leq d_AH_A
 \leq(k-1-r_h(A))h^{k-2}.
\]

If $u,v\in\mathcal U_{j,k}$ with $j\leq h$, then either $u=v$, or after
cancelling their common part, $u-v\in\mathcal D_{h,k}$.  Membership of $a$
and $a^*$ in the same sign cone proves \eqref{eq:all-order-comparison}.
\end{proof}

For a maximal-rank weak type
$\mathsf T\in\mathfrak T^{\max}_{h,k}$, define its primitive realization
diameter by
\[
 \mu_h(\mathsf T)=
 \min\{\ndiam(B):B\subset\mathbb Z,\ |B|=k,\ B\text{ primitive and realizes }
 \mathsf T\}.
\]

\begin{corollary}[Sharp type-level compression]
\label{cor:sharp-type-compression}
For every $h\geq2$ and $k\geq3$,
\begin{equation}\label{eq:sharp-type-compression}
 \max_{\mathsf T\in\mathfrak T^{\max}_{h,k}}\mu_h(\mathsf T)=h^{k-2}.
\end{equation}
Thus the maximal-rank boundary is sharp not only for individual sets: one
weak addition-table type itself forces the critical normalized diameter.
\end{corollary}

\begin{proof}
The complete-model part of Theorem~\ref{thm:compression-master}, followed by
primitive normalization, gives $\mu_h(\mathsf T)\leq h^{k-2}$ for every
maximal-rank type.  For sharpness, let $\mathsf T_{h,k}$ be the weak
degree-$h$ type of
\[
 P_{h,k}=\{0,1,h,h^2,\ldots,h^{k-2}\}.
\]
Every realization $B=(b_0<\cdots<b_{k-1})$ of this type preserves the
equalities
\[
 h b_i=b_{i+1}+(h-1)b_0,
 \qquad 1\leq i\leq k-2.
\]
After translation and primitive normalization, $b_0=0$ and $b_1=1$;
inductively $b_i=h^{i-1}$.  Hence every primitive realization of
$\mathsf T_{h,k}$ has normalized diameter $h^{k-2}$, proving
\eqref{eq:sharp-type-compression}.
\end{proof}

For an additive subsemigroup $X\subseteq\mathbb R$, write
\[
 \Rcal_X(h,k)=\{|hA|:A\subset X,\ |A|=k\}.
\]
Following Nathanson, let $N(h,k)$ be the least $N$ such that every member of
$\Rcal(h,k)$ has a representative in $[0,N-1]\cap\mathbb Z$.

\begin{corollary}[Optimal-order finite modeling]
\label{cor:optimal-finite-modeling}
For all $h\geq1$ and $k\geq2$,
\begin{equation}\label{eq:real-integer-spectrum}
 \Rcal_{\mathbb R}(h,k)=\Rcal_{\mathbb Q}(h,k)=\Rcal(h,k),
\end{equation}
and
\begin{equation}\label{eq:compression-master-bounds}
 1+\left\lceil\frac{M_{h,k}-1}{h}\right\rceil
 \leq N(h,k)
 \leq 1+\max\left\{(k-2)h^{k-2},
                    \sum_{j=0}^{k-2}h^j\right\}.
\end{equation}
Consequently, for every fixed $k\geq2$,
\begin{equation}\label{eq:optimal-compression-order}
 N(h,k)=\Theta_k(h^{k-2}).
\end{equation}
More finely, if
\[
 \nu_{h,k}(n)=\min\{\ndiam(A):A\subset\mathbb Z,\ |A|=k,\ |hA|=n\},
\]
then
\begin{equation}\label{eq:valuewise-rank-compression}
 \nu_{h,k}(n)
 \leq\bigl(k-1-r_{h,k}(n)\bigr)h^{k-2}
 \qquad(n\in\Rcal(h,k)).
\end{equation}
\end{corollary}

\begin{proof}
Theorem~\ref{thm:compression-master} gives an integer model for every real
ordered type and therefore proves \eqref{eq:real-integer-spectrum} and
\eqref{eq:valuewise-rank-compression}.  Every nonmaximal cardinality has
minimum active rank at least one, so it has a model of height at most
$(k-2)h^{k-2}$.  The maximum $M_{h,k}$ is realized by the superincreasing set
defined by
\[
 a_0=0,\qquad a_1=1,\qquad a_{i+1}=ha_i+1,
 \quad 1\leq i\leq k-2.
\]
Its height is $a_{k-1}=\sum_{j=0}^{k-2}h^j$.  It is a $B_h$-set: at the
largest differing index $i$, the total opposing contribution from all lower
coordinates is at most $ha_{i-1}<a_i$.  This proves the
upper bound in \eqref{eq:compression-master-bounds}.  For the lower bound, if
a set in
$[0,N-1]$ realizes that cardinality, its sums lie among the $h(N-1)+1$
integers in $[0,h(N-1)]$.  Hence $M_{h,k}\leq h(N-1)+1$.  Since
$M_{h,k}/h\asymp_k h^{k-2}$, \eqref{eq:optimal-compression-order} follows.
\end{proof}

\begin{remark}[Rank, shape, frequency, and compression]
\label{rem:rank-shape-compression}
For each complete signed/ordered face type of active rank $r$---equivalently,
for each sign-cone component of the realization locus of a weak type---the
translation-normalized sign cone has dimension $k-1-r$, and its projectivized
shape cell has dimension $k-2-r$.  The realization locus of a weak type of
rank $r$ is the finite union of its ordered sign-cone components; all these
components have the displayed dimensions, and the weak type's frequency
quasipolynomial has degree $k-r$.  Theorem~
\ref{thm:compression-master} shows that every such component contains an
integral point of diameter at most
\[
 (k-1-r)h^{k-2}.
\]
The filtered estimate sharpens this by replacing the factor $h^r$ with the
product of the successive active-relation degrees.  Thus the same rational
arrangement simultaneously determines compression, shape, frequency, and
relation degree.
\end{remark}

\subsection{Ordered-group shadows and optimal monomial weights}
\label{subsec:finite-shadows}

The complete-model theorem is not confined to subsets of the real line.
Every finite fragment of an ordered abelian group can first be scalarized,
with all requested equalities and strict inequalities intact.

For an element $g$ of a linearly ordered abelian group $\Gamma$, define
$\operatorname{sgn}_{\Gamma}(g)$ to be $-1$, $0$, or $1$ according as
$g<0$, $g=0$, or $g>0$.

\begin{lemma}[Finite scalarization]
\label{lem:finite-scalarization}
Let $\Gamma$ be a linearly ordered abelian group, let $H\leq\Gamma$ be a
finitely generated subgroup, and let $F\subset H$ be finite.  There is a
homomorphism $\chi:H\to\mathbb Z$ such that
\[
 \operatorname{sgn}\chi(g)=\operatorname{sgn}_{\Gamma}(g)
 \qquad(g\in F).
\]
\end{lemma}

\begin{proof}
The group $H$ is torsion-free, so fix coordinates $H\cong\mathbb Z^m$.
Let $V\subset\mathbb Z^m$ represent the positive members of
$F\cup(-F)$.  The strict rational system
\[
 \lambda\cdot v>0\qquad(v\in V)
\]
is feasible.  Otherwise Gordan's alternative, followed by rationality and
clearing denominators, would give a nonempty sum of positive elements of
$H$ equal to zero.  Multiplying a rational solution by a common denominator
produces the required integral functional.  Zeros map to zero, and the
negative signs follow by applying positivity to $-g$.
\end{proof}

For a linearly ordered abelian group $\Gamma$, integers $k\geq1$ and
$p\geq0$, and an ordered tuple
$A=(a_0<\cdots<a_{k-1})\subset\Gamma$, define
\[
 L_p^\Gamma(A)=\operatorname{span}_{\mathbb R}
 \left\{z\in\mathbb Z^k:
   \sum_i z_i=0,\ \rho(z)\leq p,\
   \sum_i z_i a_i=0\text{ in }\Gamma\right\}.
\]
For an integer $h\geq1$, put $r_h^\Gamma(A)=\dim L_h^\Gamma(A)$ and
\[
 p_i^\Gamma(A)=
 \min\{p\in\mathbb Z_{\geq1}:\dim L_p^\Gamma(A)\geq i\},
 \qquad 1\leq i\leq r_h^\Gamma(A).
\]

\begin{theorem}[Filtered ordered shadows]
\label{thm:ordered-group-shadow}
Let $\Gamma$ be any linearly ordered abelian group, let
$A=(a_0<\cdots<a_{k-1})\subset\Gamma$, where $k\geq2$, and let
$h\in\mathbb Z_{\geq1}$.  Write
$r=r_h^\Gamma(A)$ and $d=k-1-r$.  There are integers
\[
 0=b_0<b_1<\cdots<b_{k-1}
\]
such that, for every $1\leq j\leq h$ and $u,v\in\mathcal U_{j,k}$,
\begin{equation}\label{eq:ordered-group-shadow-signs}
 \operatorname{sgn}_{\Gamma}\!\left(\sum_i(u_i-v_i)a_i\right)
 =\operatorname{sgn}\!\left(\sum_i(u_i-v_i)b_i\right),
\end{equation}
and
\begin{equation}\label{eq:ordered-group-shadow-height}
 b_{k-1}
 \leq d h^{d-1}\prod_{i=1}^{r}p_i^\Gamma(A)
 \leq(k-1-r)h^{k-2},
\end{equation}
with the empty product understood as one.  For fixed $k$, the universal
$h$-exponent $k-2$ is optimal, already for subsets of $\mathbb Z$.
\end{theorem}

\begin{proof}
After translating, take $a_0=0$.  Let $H$ be the subgroup generated by the
tuple, and apply Lemma~\ref{lem:finite-scalarization} to
\[
 F_h=\left\{\sum_i z_i a_i:
 z\in\mathbb Z^k,\ \sum_i z_i=0,\ \rho(z)\leq h\right\}.
\]
The integer tuple $B=(\chi(a_0),\ldots,\chi(a_{k-1}))$ has exactly the same
signs, active rank, and successive active-relation degrees through $h$.
Theorem~\ref{thm:compression-master} applied to $B$ proves
\eqref{eq:ordered-group-shadow-signs} and
\eqref{eq:ordered-group-shadow-height}.

For optimality, when $h\geq2$ take
$P_{h,k}=\{0,1,h,h^2,\ldots,h^{k-2}\}$, and write its ordered elements as
$a_0<\cdots<a_{k-1}$.  Its complete table contains
\[
 h a_i=a_{i+1}+(h-1)a_0,
 \qquad1\leq i\leq k-2.
\]
Consequently every normalized integer shadow
$B=(b_0,\ldots,b_{k-1})$ of that table satisfies
\[
 h b_i=b_{i+1}+(h-1)b_0,
 \qquad1\leq i\leq k-2.
\]
Since $b_0=0$, it follows that $b_{i+1}=hb_i$; as $b_1\geq1$, one obtains
$b_{k-1}\geq h^{k-2}$.
\end{proof}

\begin{corollary}[Exact finite-profile universality]
\label{cor:finite-profile-transfer}
Fix $H\geq1$ and $k\geq2$, and let $G$ be a nontrivial torsion-free abelian
group.  The following four collections of truncated profiles are identical:
\[
 \bigl\{(|jA|)_{1\leq j\leq H}:A\subset G,\ |A|=k\bigr\},
 \quad
 \bigl\{(|jB|)_{1\leq j\leq H}:B\subset\mathbb Z,\ |B|=k\bigr\},
\]
\[
 \bigl\{(|P^{\times j}|)_{1\leq j\leq H}:
       P\subset\mathbb Z_{>0},\ |P|=k\bigr\},
 \quad
 \bigl\{(H_{C_B}(j))_{1\leq j\leq H}:
       B\subset\mathbb Z,\ |B|=k\bigr\},
\]
where $P^{\times j}$ is the $j$-fold product set and $C_B$ is the projective
monomial curve attached to $B$ over an arbitrary field.  More strongly,
equip the subgroup generated by an additive realization with a
translation-invariant linear order.  Its complete equality-and-order table
through degree $H$ then has an integer exponent representative.  The
resulting sumset, product-set, and monomial-curve exponent models have
identical ordered tables, active ranks, and relation-birth degrees through
$H$.

Every such table has a primitive integer representative
$B=\{0=b_0<\cdots<b_{k-1}\}$ with
\begin{equation}\label{eq:finite-profile-universal-height}
 b_{k-1}\leq dH^{d-1}\prod_{i=1}^r p_i
 \leq(k-1-r)H^{k-2},
 \qquad d=k-1-r,
\end{equation}
where $r$ and $p_i$ are its active rank and birth degrees through $H$.
The universal exponent $k-2$ for preserving the complete truncated table is
best possible.
\end{corollary}

\begin{proof}
The finitely generated subgroup of $G$ containing a given tuple is
torsion-free, hence isomorphic to $\mathbb Z^m$; equip it with a
translation-invariant lexicographic order and apply
Theorem~\ref{thm:ordered-group-shadow}.  This gives the integer model,
the filtered height bound, and preservation of ranks and births.  Dividing
the integer model by the gcd of its nonzero entries makes it primitive without
changing any comparison, equality, rank, or birth degree.  Conversely,
if $0\ne g\in G$, torsion-freeness makes $n\mapsto ng$ injective, so it
preserves every finite integer equality table and all sumset cardinalities.
Elias's identity \eqref{eq:eliasbridge} gives
$H_{C_B}(j)=|jB|$ for every $j$.

For the multiplicative model take $P=\{2^{b_0},\ldots,2^{b_{k-1}}\}$;
the exponential law identifies equality of products with equality of exponent
sums, and strict monotonicity preserves their order.  In the other direction,
apply the ordered
shadow theorem to the real tuple $(\log p)_{p\in P}$, since comparison and
equality of products are comparison and equality of the corresponding log
sums.  Optimality is the last assertion of
Theorem~\ref{thm:ordered-group-shadow}.
\end{proof}

The same theorem gives an optimal-order coefficient bound for finite
monomial-order data.  Fix an arbitrary field $K$.  Let
$W_k^{\mathrm{hom}}(h)$ be the least $W$ such
that every monomial order $\prec$ on $K[x_1,\ldots,x_k]$ admits
$w\in\mathbb Z_{\geq0}^k$ with $\min_iw_i=0$, $\max_iw_i\leq W$, and
\[
 x^\alpha\prec x^\beta
 \quad\Longleftrightarrow\quad
 w\cdot\alpha<w\cdot\beta
 \qquad(|\alpha|=|\beta|\leq h).
\]

\begin{theorem}[Optimal-order homogeneous monomial weights]
\label{thm:homogeneous-monomial-weight-complexity}
For all $k\geq2$ and $h\geq1$,
\begin{equation}\label{eq:homogeneous-weight-two-sided}
 1+h+\cdots+h^{k-2}
 \leq W_k^{\mathrm{hom}}(h)
 \leq(k-1)h^{k-2}.
\end{equation}
Thus $W_k^{\mathrm{hom}}(h)=\Theta_k(h^{k-2})$ for fixed $k$.
\end{theorem}

\begin{proof}
A monomial order extends by cancellation to a translation-invariant group
order on $\mathbb Z^k$.  Apply Theorem~\ref{thm:ordered-group-shadow} to
the standard basis vectors, sorted by their induced variable order.  They
have active rank zero, giving the upper bound after the weights are permuted
back.

For the lower bound take lexicographic order
$x_1\succ\cdots\succ x_k$.  A representing vector satisfies
$w_1>\cdots>w_k=0$, and
\[
 x_i x_k^{h-1}\succ x_{i+1}^h
 \qquad(1\leq i\leq k-1).
\]
Hence $w_i>hw_{i+1}$; integrality gives recursively
$w_i\geq1+h+\cdots+h^{k-1-i}$ and proves
\eqref{eq:homogeneous-weight-two-sided}.
\end{proof}

There is also an equality-sensitive version for monomial weights.  Let
$k,h\in\mathbb Z_{\geq1}$, let $\Gamma$ be a linearly ordered abelian group,
and let $\gamma_1,\ldots,\gamma_k\in\Gamma$.  Give
$x^\alpha$ weight $\sum_i\alpha_i\gamma_i$.  There is a
nonnegative integer vector $w$ preserving the entire homogeneous weight
preorder through degree $h$, including ties.  If
$\delta_0<\cdots<\delta_{\ell-1}$ are the distinct $\gamma_i$, put
\[
 r=r_h^\Gamma(\delta_0,\ldots,\delta_{\ell-1}),\qquad
 p_i=p_i^\Gamma(\delta_0,\ldots,\delta_{\ell-1})
 \quad(1\leq i\leq r).
\]
Then for $\ell=1$ one may take
$w=0$, while for $\ell\geq2$, when necessarily $0\leq r\leq\ell-2$, one may take
\begin{equation}\label{eq:monomial-gamma-shadow}
 \max_iw_i\leq
 (\ell-1-r)h^{\ell-2-r}\prod_{i=1}^r p_i.
\end{equation}
We use the maximal-weight initial-form convention.  Thus, for a homogeneous
polynomial $f=\sum_\alpha c_\alpha x^\alpha$, put
\[
 \begin{aligned}
 \operatorname{in}_\gamma(f)
 &=\sum_{\sum_i\alpha_i\gamma_i\,=\,
          \max_{\beta:c_\beta\ne0}\sum_i\beta_i\gamma_i}
          c_\alpha x^\alpha,\\
 \operatorname{in}_w(f)
 &=\sum_{w\cdot\alpha\,=\,\max_{\beta:c_\beta\ne0}w\cdot\beta}
          c_\alpha x^\alpha.
 \end{aligned}
\]
For a homogeneous ideal $I$, let
$\operatorname{in}_\gamma(I)_j$ and $\operatorname{in}_w(I)_j$ denote the
degree-$j$ spans of these initial forms for $f\in I_j$.  The preserved weight
preorder gives
\[
 \operatorname{in}_\gamma(I)_j=\operatorname{in}_w(I)_j
 \qquad(0\leq j\leq h).
\]
This assertion concerns the displayed monomial weight data; it is not a
global rank reduction for arbitrary valuations.

Indeed, apply Theorem~\ref{thm:ordered-group-shadow} to the ordered tuple of
distinct values and let $b_j$ be its integer shadow.  Set $w_i=b_j$ whenever
$\gamma_i=\delta_j$.  Aggregating the exponent counts within each equality
class converts every homogeneous comparison through degree $h$ into the
corresponding comparison on the $\delta$-tuple.  Hence the full preorder,
including all ties, is preserved, and so are the selected initial-form pieces.

Qualitative finite weight realization is classical: Robbiano gives the
matrix description of term orders \cite{Robbiano1985}, Vietri studies total
weight orders in fixed degree \cite{Vietri2002}, and Bayer--Mumford use
bounded-degree positional weights \cite{BayerMumford1993}.  Kaveh--Manon
prove qualitative scalarization of rational weighting matrices for a fixed
initial ideal \cite[Appendix Lemma~8.9 and Proposition~8.10]{KavehManon2019},
and order-preserving Freiman models were developed by
Amirkhanyan--Bush--Croot \cite{AmirkhanyanBushCroot2018}.  To our knowledge,
the contribution
of Theorems~\ref{thm:ordered-group-shadow} and
\ref{thm:homogeneous-monomial-weight-complexity} is the uniform optimal
$h$-exponent for the entire homogeneous truncated preorder, together with
the active-rank and relation-birth refinement in
\eqref{eq:ordered-group-shadow-height}.

\section{Active-relation rank: diversity, frequency, and shape}
\label{sec:rank-duality}

\subsection{Exact frequencies on rational flats}

The compression theorem has a complementary enumerative form: active rank
controls both how many collision patterns exist and how frequently each one
is realized.  We retain the notation $\mathcal U_{h,k}$, $\rho$, $L_h(a)$,
and $r_h(a)$ from Section~\ref{sec:root-sweep-dictionary}, and put
$M_{h,k}=|\mathcal U_{h,k}|=\binom{h+k-1}{k-1}$.  The rational-arrangement
counting used below is standard inside-out Ehrhart theory
\cite{BeckZaslavsky2006,BeckBogartPham2012,BogartCuellar2025}; the point is
to retain every collision stratum, group the strata by the exact value of
$|hA|$, and read off their degrees.  Since every vector in $L_h(a)$ is
orthogonal to both ${\bf 1}=(1,\ldots,1)$ and $a$, and these two vectors
are independent, one always has
\begin{equation}\label{eq:rank-frequency-range}
 0\leq r_h(a)\leq k-2.
\end{equation}

For $n\in\mathcal R(h,k)$, define the minimum active rank
\begin{equation}\label{eq:minimum-active-rank}
 r_{h,k}(n)=\min\left\{r_h(a):a\in\mathbb Z^k,
           a_1<\cdots<a_k,
          \left|\{u\cdot a:u\in\mathcal U_{h,k}\}\right|=n\right\}.
\end{equation}
The same minimum is obtained if integer tuples are replaced by real tuples.  This
also follows from O'Bryant's type-realization theorem
\cite[Theorem~7]{OBryant2025}; the proof below gives the rational-stratum
realization needed for the frequency count.

\begin{theorem}[Complete rank--frequency stratification]
\label{thm:rank-frequency}
Fix $h\geq2$ and $k\geq3$.  For every $n\in\mathcal R(h,k)$, the function
\[
 F^{(k)}_{h,n}(q)=
 \#\left\{A\in\binom{[q]}k:|hA|=n\right\}
\]
is a quasipolynomial in $q-1$ (and hence in $q$).  Its degree is exactly
\[
 \deg F^{(k)}_{h,n}=k-r_{h,k}(n).
\]
In particular, there is an explicitly computable positive rational number
$c_{h,k,n}$ such that
\begin{equation}\label{eq:rank-frequency-asymptotic}
 F^{(k)}_{h,n}(q)
 =c_{h,k,n}q^{k-r_{h,k}(n)}
  +O_{h,k}\!\left(q^{k-r_{h,k}(n)-1}\right).
\end{equation}
The possible frequency exponents therefore belong to
$\{2,3,\ldots,k\}$.
\end{theorem}

Here ``quasipolynomial'' means that there is an integer $P=P(h,k)\geq1$
such that, on each residue class modulo $P$, the function agrees with a
polynomial for every positive $q$.  The period can be taken to divide a
common denominator of the vertices of the finitely many rational polytopes
in the proof.  Thus no unspecified asymptotic threshold is needed, although
only the leading term will be used below.

\begin{proof}
Let $\mathcal Z_{h,k}$ be the finite set, modulo sign, of primitive nonzero
vectors $z\in\mathbb Z^k$ satisfying $\sum_i z_i=0$ and $\rho(z)\leq h$.
Consider the central rational arrangement
\[
 \mathcal A_{h,k}=\{H_z:z\in\mathcal Z_{h,k}\},
 \qquad H_z=\{x\in\mathbb R^k:z\cdot x=0\},
\]
inside the open ordered chamber
\[
 \mathcal O=\{x\in\mathbb R^k:x_1<\cdots<x_k\}.
\]
For an intersection flat $X$ of this arrangement, set
\[
 X^\circ=(X\cap\mathcal O)\setminus
 \bigcup_{\substack{H\in\mathcal A_{h,k}\\X\not\subseteq H}}(X\cap H).
\]
The nonempty sets $X^\circ$ are disjoint and partition $\mathcal O$.
If $x\in X^\circ$, then the active normals at $x$ are exactly the
arrangement normals whose hyperplanes contain $X$.  Consequently
\begin{equation}\label{eq:flat-rank}
 r_h(x)=\operatorname{codim}_{\mathbb R^k}X.
\end{equation}

We next check that the cardinality of the $h$-fold sumset is constant on
each $X^\circ$.  For $u,v\in\mathcal U_{h,k}$,
\[
 u\cdot x=v\cdot x\quad\Longleftrightarrow\quad (u-v)\cdot x=0.
\]
If $u\ne v$, divide $u-v$ by the gcd of its coordinates.  The resulting
primitive vector $z$ satisfies
$\rho(z)\leq\rho(u-v)\leq h$, so $H_z$ belongs to
$\mathcal A_{h,k}$.  Whether $x$ lies on $H_z$ is constant on $X^\circ$.
Thus the complete equality pattern on $\mathcal U_{h,k}$, and hence
\[
 N(X):=\left|\{u\cdot x:u\in\mathcal U_{h,k}\}\right|,
\]
is independent of the choice of $x\in X^\circ$.

All flats are rational.  Choose any nonempty connected component of
$X^\circ$, equivalently a chamber obtained by fixing the signs of the
remaining rational forms.  Inside the rational vector space $X$, that
component is cut out by finitely many strict rational linear inequalities
and therefore contains a rational point.  Multiplying by a common denominator
and then translating by a multiple of ${\bf1}$ gives a strictly increasing
integer point with the same active hyperplanes.  (Every arrangement normal
has coordinate sum zero, so translation changes none of its equations.)
It follows both that every nonempty stratum has integer realizations and
that the minimum in \eqref{eq:minimum-active-rank} is unchanged over the
reals.

It remains to count the integer points in a stratum.  Translate $[q]$ to
$\{0,1,\ldots,Q\}$, where $Q=q-1$; again this leaves every relation
equation invariant.  Let
\[
 \Lambda_X=X\cap\mathbb Z^k,
 \qquad
 P_X=X\cap\{0\leq x_1\leq\cdots\leq x_k\leq1\}.
\]
If $X^\circ$ is nonempty, choose $x\in X^\circ$.  Every arrangement normal
has coordinate sum zero, so ${\bf1}\in X$; hence a positive affine image
$\lambda x+t{\bf1}$ also lies in $X$.  Choosing $\lambda>0$ small and then
$t$ puts this image strictly inside
$0<x_1<\cdots<x_k<1$.  The strict inequalities persist on a relative
neighborhood in $X$, so $P_X$ has nonempty relative interior.  Thus $P_X$
is a rational polytope of full dimension $\dim X$ in $X$, with positive
relative volume.  Ehrhart's
theorem, applied in the lattice $\Lambda_X$, gives
\begin{equation}\label{eq:flat-ehrhart}
 |QP_X\cap\Lambda_X|
 =\operatorname{vol}_{\Lambda_X}(P_X)Q^{\dim X}
   +O_X(Q^{\dim X-1}),
\end{equation}
and the left side is a quasipolynomial in $Q$.  The relative normalized
volume in \eqref{eq:flat-ehrhart} is positive and rational.

To pass from $P_X$ to $X^\circ$, exclude the coordinate-equality faces and
the finitely many sections $X\cap H$ with $X\not\subseteq H$.  Finite
inclusion--exclusion expresses the resulting count as an integer linear
combination of Ehrhart quasipolynomials of rational polytopes; compare
\cite[Theorems~4.1--4.2 and Corollary~4.3]{BeckZaslavsky2006}.  Every
excluded section has
dimension at most $\dim X-1$.  Hence, if
\[
 E_X(q)=\#\{a\in X^\circ\cap\mathbb Z^k:
                    1\leq a_1<\cdots<a_k\leq q\},
\]
then $E_X(q)$ is a quasipolynomial in $Q=q-1$ and
\begin{equation}\label{eq:stratum-ehrhart}
 E_X(q)=\operatorname{vol}_{\Lambda_X}(P_X)q^{\dim X}
          +O_X(q^{\dim X-1}).
\end{equation}

Finally, the strata are disjoint, so exactly
\begin{equation}\label{eq:frequency-flat-sum}
 F^{(k)}_{h,n}(q)=
 \sum_{\substack{X:\ X^\circ\ne\varnothing\\N(X)=n}}E_X(q).
\end{equation}
The sum is finite.  By \eqref{eq:flat-rank}, its largest-dimensional terms
are precisely the flats of codimension $r_{h,k}(n)$.  For fixed $h,k$, the
arrangement has only finitely many flats, so their leading coefficients and
Ehrhart error constants are uniformly bounded in terms of $h,k$; after the
largest-dimensional leading terms are collected, all lower-dimensional
strata and all Ehrhart errors contribute
$O_{h,k}(q^{k-r_{h,k}(n)-1})$.  The positive leading coefficients cannot
cancel.  Therefore
\eqref{eq:frequency-flat-sum} has degree $k-r_{h,k}(n)$ and leading
coefficient
\begin{equation}\label{eq:rank-frequency-constant}
 c_{h,k,n}=
 \sum_{\substack{X:\ X^\circ\ne\varnothing, N(X)=n\\
          \operatorname{codim}X=r_{h,k}(n)}}
       \operatorname{vol}_{\Lambda_X}(P_X)>0.
\end{equation}
This coefficient is rational and can be computed effectively from the finite
arrangement.
The range of exponents follows from \eqref{eq:rank-frequency-range}.
\end{proof}

\begin{corollary}[Frequency of every weak addition-table type]
\label{cor:addition-type-frequency}
Fix $h\geq2$ and $k\geq3$.  For a strictly increasing
$a=(a_1,\ldots,a_k)$, let $\mathsf T_h(a)$ be the equivalence relation on
$\mathcal U_{h,k}$ defined by
\[
 u\sim_{\mathsf T_h(a)}v
 \quad\Longleftrightarrow\quad
 u\cdot a=v\cdot a.
\]
Equivalently, after quotienting the forced permutation symmetries, this is
O'Bryant's weak type of the $h$-fold addition table
\cite[Section~3]{OBryant2025}.
Let $\mathfrak T_{h,k}$ be the finite set of types realized by strictly
increasing integer $k$-tuples.  For $\mathsf T\in\mathfrak T_{h,k}$, put
\[
 r(\mathsf T)=\dim_{\mathbb R}
 \operatorname{span}_{\mathbb R}
 \{u-v:u\sim_{\mathsf T}v\}
\]
and
\[
 F^{(k)}_{h,\mathsf T}(q)=
 \#\left\{A\in\binom{[q]}k:\mathsf T_h(A)=\mathsf T\right\}.
\]
Then $F^{(k)}_{h,\mathsf T}(q)$ is a quasipolynomial in $q-1$, valid for
every positive integer $q$, of degree exactly $k-r(\mathsf T)$.  In
particular, there is an effectively computable
$c_{h,k,\mathsf T}\in\mathbb Q_{>0}$ such that
\begin{equation}\label{eq:addition-type-frequency}
 F^{(k)}_{h,\mathsf T}(q)
 =c_{h,k,\mathsf T}q^{k-r(\mathsf T)}
  +O_{h,k,\mathsf T}\!\left(q^{k-r(\mathsf T)-1}\right).
\end{equation}
Moreover,
\begin{equation}\label{eq:type-sum-frequency}
 F^{(k)}_{h,n}(q)=
 \sum_{\substack{\mathsf T\in\mathfrak T_{h,k}\\
 |\mathcal U_{h,k}/\mathsf T|=n}}
 F^{(k)}_{h,\mathsf T}(q).
\end{equation}
Hence summing over weak types with $n$ equivalence classes recovers
Theorem~\ref{thm:rank-frequency}.
\end{corollary}

\begin{proof}
For $\mathsf T\in\mathfrak T_{h,k}$, let
\[
 X_{\mathsf T}=
 \{x\in\mathbb R^k:(u-v)\cdot x=0
       \text{ whenever }u\sim_{\mathsf T}v\}.
\]
By construction,
\[
 X_{\mathsf T}=
 \left(\operatorname{span}_{\mathbb R}
 \{u-v:u\sim_{\mathsf T}v\}\right)^\perp.
\]
Thus an arrangement normal $z$ vanishes identically on $X_{\mathsf T}$ if
and only if it lies in the displayed span.  If $a$ realizes $\mathsf T$ and
$z$ lies in that span, then $z\cdot a=0$.  Since $\rho(z)\leq h$, choose
$w\in\mathbb Z_{\geq0}^k$ with $\sum_iw_i=h-\rho(z)$; then
\[
 z=(z^++w)-(z^-+w),
\]
and the two exponent vectors on the right are equivalent in $\mathsf T$.
The converse is immediate from the definition of $X_{\mathsf T}$.  Hence an
arrangement hyperplane contains $X_{\mathsf T}$ exactly when its equality
holds in $\mathsf T$.  Consequently the points having type exactly
$\mathsf T$ form the relatively open arrangement stratum
$X_{\mathsf T}^{\circ}$ from the proof of
Theorem~\ref{thm:rank-frequency}; this stratum is the finite union of its
sign-cone components inside the flat $X_{\mathsf T}$.  Moreover,
\[
 \operatorname{codim}X_{\mathsf T}=r(\mathsf T).
\]
Indeed, the orthogonal-complement identity gives the displayed codimension,
and the same padding argument identifies its defining span with $L_h(a)$.
Thus $r(\mathsf T_h(a))=r_h(a)$.

It follows that $F^{(k)}_{h,\mathsf T}(q)=E_{X_{\mathsf T}}(q)$ in the
notation of that proof.  The inclusion--exclusion count discussed above is
an exact quasipolynomial in $q-1$; by \eqref{eq:stratum-ehrhart}, it has degree
$\dim X_{\mathsf T}=k-r(\mathsf T)$ and positive rational leading coefficient
\[
 c_{h,k,\mathsf T}
 =\operatorname{vol}_{X_{\mathsf T}\cap\mathbb Z^k}
 \left(X_{\mathsf T}\cap
 \{0\leq x_1\leq\cdots\leq x_k\leq1\}\right).
\]
This proves \eqref{eq:addition-type-frequency}.  Finally,
$|hA|=|\mathcal U_{h,k}/\mathsf T_h(A)|$, so the disjoint sum
\eqref{eq:type-sum-frequency} is immediate.  Its largest degree is obtained
from the types of minimum rank, and their positive leading coefficients do
not cancel, giving precisely Theorem~\ref{thm:rank-frequency}.
\end{proof}

\subsection{Rank diversity and arrangement geometry}
\label{subsec:rank-diversity}

The preceding result fixes $h$ and lets the ambient interval grow.  We now
fix $k$ and let the table order $h$ grow.  The same active-relation rank
controls both the diversity of weak types as a function of $h$ and the
frequency of each fixed type as a function of $q$.

\begin{lemma}[A primitive-lattice shell count]
\label{lem:primitive-lattice-shell}
Fix integers $D\geq3$ and $1\leq s<D$.  A sublattice
$\Lambda\subset\mathbb Z^D$ is called primitive when
$\Lambda=(\operatorname{span}_{\mathbb R}\Lambda)\cap\mathbb Z^D$, and we
write
\[
 \Lambda^\perp=
 (\operatorname{span}_{\mathbb R}\Lambda)^\perp\cap\mathbb Z^D.
\]
An orientation of $\operatorname{span}_{\mathbb R}\Lambda$, together with
the standard orientation of $\mathbb R^D$, induces the orientation on
$\Lambda^\perp$ used below.
Write
\[
 \mathcal X_j=\operatorname{SO}_j(\mathbb R)\backslash
 \operatorname{SL}_j(\mathbb R)/\operatorname{SL}_j(\mathbb Z)
\]
for Horesh--Karasik's shape space of oriented rank-$j$ lattices, equipped
with the invariant quotient Radon measure induced from Haar measure, with any
fixed normalization.  Write $\operatorname{Gr}^+(s,D)$ for the oriented
Grassmannian of $s$-planes in $\mathbb R^D$, equipped with its invariant
measure.  Let $\Phi\subset\operatorname{Gr}^+(s,D)$ and
$E\times F\subset\mathcal X_s\times\mathcal X_{D-s}$ be positive-measure
boundary-controllable sets (BCSs) in the sense of
\cite[Definition~3.1]{HoreshKarasik2023}, with $E$ and $F$ relatively
compact.  Relatively compact domains with piecewise $C^1$ boundary suffice.
Let $N(T)$ be the number of oriented primitive rank-$s$ lattices
$\Lambda\subset\mathbb Z^D$ such that
\[
 \begin{gathered}
  \dfrac{T}{2}<\operatorname{covol}(\Lambda)\leq T,\qquad
  \operatorname{span}_{\mathbb R}\Lambda\in\Phi,\\
  \operatorname{shape}(\Lambda)\in E,\qquad
  \operatorname{shape}(\Lambda^\perp)\in F.
 \end{gathered}
\]
Then, as $T\to\infty$,
\begin{equation}\label{eq:primitive-lattice-shell}
 N(T)=\Theta_{D,s,\Phi,E,F}(T^D).
\end{equation}
\end{lemma}

\begin{proof}
In the notation of Horesh--Karasik, take $n=D$ and $d=s$, with direction
domain $\Phi$ and joint shape domain $E\times F$.  The controlled-boundary
hypotheses of \cite[Theorem~1.1]{HoreshKarasik2023} are satisfied; positive
measure makes the leading constant positive, and relative compactness of $E$
and $F$ selects the theorem's bounded case.  That theorem says that the
corresponding count of oriented lattices with covolume at most $T$ equals a
positive constant times $T^D$, with a power-saving error, jointly in the
direction and the two shapes.  Subtracting the formula at $T/2$
from the formula at $T$ multiplies the positive main term by $1-2^{-D}$
and leaves an $o(T^D)$ error.  Forgetting orientation has fibres of
cardinality at most two, so the same order also holds for the underlying
primitive lattices and rational spans used below.  This proves
\eqref{eq:primitive-lattice-shell}.
\end{proof}

For $0\leq r\leq k-2$, define
\[
 \tau_{k,r}(h)
 =\#\{\mathsf T\in\mathfrak T_{h,k}:r(\mathsf T)=r\}.
\]
Also put
\[
 G^{(k)}_{h,r}(q)
 =\#\left\{A\in\binom{[q]}k:r(\mathsf T_h(A))=r\right\}.
\]
For normalized gap directions, write
\[
 \Delta_{k-1}^{\circ}
 =\left\{x\in\mathbb R_{>0}^{k-1}:\sum_i x_i=1\right\}.
\]

\begin{theorem}[Rank duality for type diversity and frequency]
\label{thm:rank-diversity-frequency}
Fix $k\geq3$. For every $0\leq r\leq k-2$,
\begin{equation}\label{eq:rank-diversity}
 \tau_{k,r}(h)
 =\Theta_{k,r}\!\left(h^{(k-1)r}\right)
 \qquad(h\to\infty).
\end{equation}
For $r=0$ one has, more precisely, $\tau_{k,0}(h)=1$ for every $h$.
More locally, let $U$ be any nonempty relatively compact open subset of
$\Delta_{k-1}^{\circ}$. There is a constant $c_{k,U}>0$ such that, for all
sufficiently large $h$, at least
\[
 c_{k,U} h^{(k-1)(k-2)}
\]
maximal-rank types have a representative
$A=\{0=a_1<a_2<\cdots<a_k\}$ for which
\[
 \frac{(a_2-a_1,\ldots,a_k-a_{k-1})}{a_k-a_1}\in U,
 \qquad a_k\leq h^{k-2}.
\]

For each fixed $h$ and each $\mathsf T\in\mathfrak T_{h,k}$ of rank $r$,
the exact frequency law is
\begin{equation}\label{eq:master-type-frequency}
 F^{(k)}_{h,\mathsf T}(q)
 =c_{h,k,\mathsf T}q^{k-r}
  +O_{h,k,\mathsf T}(q^{k-r-1}),
 \qquad c_{h,k,\mathsf T}\in\mathbb Q_{>0},
\end{equation}
and this function is a quasipolynomial in $q-1$ for every positive integer
$q$. Consequently, whenever $\tau_{k,r}(h)>0$,
$G^{(k)}_{h,r}(q)$ is an exact quasipolynomial in $q-1$ of degree $k-r$,
with
\begin{equation}\label{eq:rank-aggregate-frequency}
 G^{(k)}_{h,r}(q)
 =C_{h,k,r}q^{k-r}+O_{h,k,r}(q^{k-r-1}),\qquad
 C_{h,k,r}=\sum_{\substack{\mathsf T\in\mathfrak T_{h,k}\\
                            r(\mathsf T)=r}}
                    c_{h,k,\mathsf T}\in\mathbb Q_{>0}.
\end{equation}
\end{theorem}

\begin{proof}
Put $d=k-1$. Translate a strictly increasing $k$-tuple so that its first
entry is zero, and write its positive gap vector as
\[
 g=(g_1,\ldots,g_d),\qquad g_i=a_{i+1}-a_i>0.
\]
If $z=(z_1,\ldots,z_{d+1})$ satisfies $\sum_jz_j=0$, set
\[
 c_i=\sum_{j=i+1}^{d+1}z_j\qquad(1\leq i\leq d).
\]
Then $z\cdot a=c\cdot g$, and the inverse transformation is
\begin{equation}\label{eq:gap-tail-map}
 z=(-c_1,c_1-c_2,\ldots,c_{d-1}-c_d,c_d).
\end{equation}
Thus this is a linear isomorphism from the homogeneous relation space to
$\mathbb R^d$. In these coordinates the relation degree is the norm
\begin{equation}\label{eq:gap-active-norm}
 \delta(c)
 =\rho(z)
 =\frac12\left(
 |c_1|+\sum_{i=1}^{d-1}|c_i-c_{i+1}|+|c_d|\right).
\end{equation}
Indeed, $\sum_jz_j=0$ gives
$\rho(z)=\sum_jz_j^+=\frac12\sum_j|z_j|$. Moreover,
\begin{equation}\label{eq:gap-norm-comparison-master}
 \|c\|_\infty\leq\delta(c)
 \leq\|c\|_1\leq\sqrt d\,\|c\|_2.
\end{equation}
For the first inequality, observe that
$c_i=-\sum_{j\leq i}z_j$ and that every partial sum of a zero-sum vector
has absolute value at most $\sum_jz_j^+$; the remaining inequalities
follow from \eqref{eq:gap-active-norm}.

Write
\[
 S_h=\{c\in\mathbb Z^d:\delta(c)\leq h\},\qquad
 S_h(g)=\{c\in S_h:c\cdot g=0\}.
\]
A homogeneous $z$ is the difference of two members of
$\mathcal U_{h,k}$ if and only if $\rho(z)\leq h$: necessity is immediate,
and for sufficiency one pads $z^+$ and $z^-$ by the same nonnegative vector
of total mass $h-\rho(z)$. It follows that $S_h(g)$ determines the weak
type completely, and that the rank of the type is
$\dim_{\mathbb R}\operatorname{span}_{\mathbb R}S_h(g)$.

We first prove the upper bound in \eqref{eq:rank-diversity}.
By \eqref{eq:gap-norm-comparison-master}, $|S_h|=O_d(h^d)$. If a realized
type has rank $r$, put
\[
 W=\operatorname{span}_{\mathbb R}S_h(g).
\]
Since $W\subseteq g^\perp$, one has the exact identity
\[
 S_h(g)=S_h\cap W.
\]
Thus $W$ determines the type, and $W$ is spanned by $r$ members of $S_h$.
Consequently
\[
 \tau_{k,r}(h)\leq |S_h|^r=O_{k,r}(h^{dr}).
\]

For the lower bound, first suppose that $d\geq3$ and $1\leq r<d$.
If $r<d-1$, consider the rational $r$-plane
\[
 W_0=\operatorname{span}_{\mathbb R}
       (e_1-e_d,\ldots,e_r-e_d).
\]
Choose an orientation of $W_0$.  Its orthogonal complement contains the
all-ones vector. Orthogonally
projecting that vector to $W^\perp$ shows that, throughout a sufficiently
small neighborhood of $W_0$ in the oriented Grassmannian, the space $W^\perp$
contains a strictly positive vector. Choose a positive-measure
boundary-controllable domain $\Phi$ with closure inside this neighborhood.

If $r=d-1$, let $U\Subset\Delta_{d}^{\circ}$ be the open set in the
localized assertion. Choose a nonempty open ball
$U_0\Subset U$ with smooth boundary, and take
\[
 \Phi=\{x^\perp:x\in U_0\}\subset\operatorname{Gr}^+(d-1,d),
\]
where $x^\perp$ has the orientation determined by the ambient orientation
and the positive normal $x$.
The map $x\mapsto x^\perp$ is a smooth coordinate chart on positive
directions, so $\Phi$ has positive measure and controlled boundary.
Every $W\in\Phi$ has a unique normalized positive normal lying in $U_0$.

Choose relatively compact positive-measure shape domains $E\subset\mathcal
X_r$ and $E^\perp\subset\mathcal X_{d-r}$ so that
$E\times E^\perp$ is boundary-controllable (small domains with piecewise
$C^1$ boundary suffice). Compactness of $E$ gives a constant
$B$ such that every rank-$r$ lattice $\Lambda$ of shape in $E$ has
$r$ independent vectors of Euclidean length at most
\[
 B\,\operatorname{covol}(\Lambda)^{1/r}.
\]
This is the usual successive-minima bound on a compact part of shape
space. Fix $\eta>0$ so small that
\[
 \sqrt d\,B\eta^{1/r}<1,\qquad \sqrt d\,\eta\leq1,
\]
and apply Lemma~\ref{lem:primitive-lattice-shell} with
$T=\eta h^r$.  After forgetting orientation, there remain
$\Theta_{d,r,\Phi,E,E^\perp}(h^{dr})$ primitive lattices $\Lambda$, since the
fibres have cardinality at most two.  Their real spans $W$ lie in the
projection of $\Phi$.  Each such $\Lambda$ contains $r$ independent vectors
$c_1,\ldots,c_r$ satisfying
\[
 \delta(c_i)\leq\sqrt d\,\|c_i\|_2
 \leq\sqrt d\,B T^{1/r}<h.
\]

It remains to ensure that $W$, rather than a larger space, is the exact
active span. The relatively open cone
\[
 C_W=W^\perp\cap\mathbb R_{>0}^d
\]
is nonempty. For each $c\in S_h\setminus W$, the section
$C_W\cap c^\perp$ is a proper hyperplane section of $C_W$; otherwise
$W^\perp\subseteq c^\perp$ and hence $c\in W$. Since $S_h$ is finite,
these sections do not cover $C_W$. The plane $W$ is rational, so rational
points are dense in $W^\perp$. We may therefore choose
\[
 g\in C_W\cap\mathbb Q^d
\]
outside all the forbidden sections and then clear denominators. For this
positive integral gap vector,
\[
 S_h(g)=S_h\cap W,
\]
and the vectors $c_1,\ldots,c_r$ show that this set spans $W$. Distinct
primitive lattices have distinct rational spans, while equal weak types
have equal active sets and hence equal spans.  We have thus obtained
$\Omega_{d,r}(h^{dr})$ distinct rank-$r$ types when $r<d-1$.  When
$r=d-1$, the same argument gives $\Omega_{d,r,U}(h^{dr})$ types for the
prescribed set $U$; choosing once and for all a fixed nonempty
$U\Subset\Delta_d^\circ$ also proves the global lower bound in this case.

When $r=d-1$, the same construction gives the asserted localization.
Indeed, $W^\perp\cap\mathbb Z^d$ is generated by a primitive vector
$g_W$, whose sign may be chosen so that $g_W>0$. By the definition of
$\Phi$,
\[
 \frac{g_W}{\|g_W\|_1}\in U_0\subset U.
\]
The standard
codimension-one determinant identity gives
\[
 \|g_W\|_2
 =\operatorname{covol}(g_W^\perp\cap\mathbb Z^d)
 =\operatorname{covol}(\Lambda)\leq T.
\]
Indeed, the Hodge dual of the wedge of a lattice basis is an integral
vector on the line $\mathbb Rg_W$, and primitivity of $\Lambda$ makes its
coefficient $\pm1$.
Here $g_W^\perp=W$, so every active relation lies in $W$.  The previously
constructed vectors $c_1,\ldots,c_{d-1}$ are active and span $W$; hence the
active span is exactly $W$, although further dependent active relations may
occur inside it.  The
associated cumulative-sum set
\[
 A_{g_W}=\{0,g_1,g_1+g_2,\ldots,g_1+\cdots+g_d\}
\]
satisfies
\[
 \max A_{g_W}=\|g_W\|_1
 \leq\sqrt d\,T
 =\sqrt d\,\eta h^{d-1}\leq h^{d-1}=h^{k-2}.
\]
Thus this same localized representative already lies in the exact core; no
change of its normalized gap direction is needed.  The same exact bound
holds in the $d=2$ case below.

The only case not covered by the lattice-counting lemma is $d=2$, where
necessarily $r=1$. For every coprime pair $p,q>0$ with $p+q\leq h$, take
\[
 c=(p,-q),\qquad g=(q,p).
\]
Then $c\cdot g=0$ and $\delta(c)=p+q\leq h$. Distinct primitive pairs give
distinct relation lines, and
\[
 \#\{(p,q)\in\mathbb Z_{>0}^2:(p,q)=1,\ p+q\leq h\}
 =\sum_{m=2}^{h}\varphi(m)
 =\sum_{m\leq h}\varphi(m)-1
 =\Theta(h^2),
\]
where $\varphi$ is Euler's totient function.
This proves the lower bound for $k=3$.
For the localized assertion, choose an open interval
$U_0\Subset U\subset\Delta_2^\circ$ and retain the pairs for which
\[
 \frac{(q,p)}{p+q}\in U_0,\qquad \frac h2<p+q\leq h.
\]
This is a fixed positive-area sector dilated by $h$. M\"obius inversion
counts $\Theta_U(h^2)$ primitive lattice points in it. The corresponding
cumulative-sum set has normalized gap vector in $U$ and maximum
$p+q\leq h$.

Finally, rank zero is the single type in which all distinct members of
$\mathcal U_{h,k}$ give distinct sums. It is realized, for example, by a
sufficiently rapidly increasing positional-base set; equivalently, choose
a positive integral gap vector avoiding the finitely many nonzero
hyperplanes from $S_h$. Hence $\tau_{k,0}(h)=1$, completing
\eqref{eq:rank-diversity}.

For $h=1$, the unique type has rank zero and frequency $\binom qk$, so the
assertion is immediate.  For $h\geq2$, the individual frequency assertion
\eqref{eq:master-type-frequency} is
Corollary~\ref{cor:addition-type-frequency}. Summing those exact
quasipolynomials over the finite, nonempty set of rank-$r$ types gives
\eqref{eq:rank-aggregate-frequency}. All leading coefficients are
positive, so neither the degree nor the leading term can cancel.
\end{proof}

\begin{remark}\label{rem:rank-duality-scope}
Theorem~\ref{thm:rank-diversity-frequency} concerns two separate asymptotic
regimes.  Formula \eqref{eq:rank-diversity} is a fixed-$k$ statement as
$h\to\infty$, whereas \eqref{eq:master-type-frequency} and
\eqref{eq:rank-aggregate-frequency} fix $h$ and let $q\to\infty$. No
uniform joint $(h,q)$ asymptotic is asserted.
\end{remark}

\begin{corollary}[Sharp enumeration of weak types]
\label{cor:weak-type-diversity}
For every fixed $k\geq3$, as $h\to\infty$,
\begin{equation}\label{eq:weak-type-total-growth}
 |\mathfrak T_{h,k}|
 =\Theta_k\!\left(h^{(k-1)(k-2)}\right).
\end{equation}
The maximal-rank types alone have this order.  More precisely, for every
nonempty relatively compact open subset
$U\Subset\Delta_{k-1}^{\circ}$, they already have this order among sets
whose normalized gaps lie in $U$ and whose largest element is at most
$h^{k-2}$.  In particular,
\[
 \tau_{4,0}(h)=1,\quad
 \tau_{4,1}(h)=\Theta(h^3),\quad
 \tau_{4,2}(h)=\Theta(h^6).
\]
Thus \eqref{eq:weak-type-total-growth} gives the sharp polynomial growth
order asked for by O'Bryant \cite[Question~6]{OBryant2025}.
\end{corollary}

\begin{proof}
Sum \eqref{eq:rank-diversity} over $0\leq r\leq k-2$. The largest
exponent occurs uniquely at $r=k-2$.  The localized lower bound is the
localized maximal-rank clause of
Theorem~\ref{thm:rank-diversity-frequency}, while the global upper bound in
\eqref{eq:rank-diversity} gives the matching order.
\end{proof}

There are two natural refinements of a weak type. An \emph{ordered type}
records the total preorder of all $h$-sums, including ties. A
\emph{generic type} records their strict total order for a $B_h$-set; this
is the combinatorial type used by Bogart and Cu\'ellar
\cite[Definition~2.4 and Theorem~2.5]{BogartCuellar2025}.
After norm comparison, the asymptotic order of the unlabelled
arrangement-cell count follows from B\'ar\'any--Bureaux--Lund
\cite[Theorem~5.6]{BaranyBureauxLund2018}.  The following proof identifies
the additive types explicitly and connects them to the filtered compression
window and the rank-refined weak count.

\begin{corollary}[Ordered and generic types]
\label{cor:ordered-generic-type-diversity}
For each fixed $k\geq3$, as $h\to\infty$, both the number of realized ordered
$h$-addition-table types and the number of realized generic $B_h$-types
are
\begin{equation}\label{eq:ordered-generic-growth}
 \Theta_k\!\left(h^{(k-1)(k-2)}\right).
\end{equation}
\end{corollary}

\begin{proof}
Theorem~\ref{thm:compression-master} realizes every ordered type by a set
\[
 0=a_0<a_1<\cdots<a_{k-1}\leq(k-1)h^{k-2}.
\]
There are only
$\binom{(k-1)h^{k-2}}{k-1}=O_k(h^{(k-1)(k-2)})$ such sets.  This proves
the upper bound for ordered types, and hence also for generic types.  The
ordered lower bound follows by forgetting the order: every realized weak
type has at least one realized ordered refinement, so
Corollary~\ref{cor:weak-type-diversity} applies.

In normalized gap coordinates the parameter space is the open simplex
\[
 P^\circ=\{g\in\mathbb R_{>0}^{d}:\textstyle\sum_i g_i=1\},
 \qquad d=k-1,
\]
of dimension $p=d-1=k-2$.  With $\delta$ denoting the gap-coordinate
relation degree from \eqref{eq:gap-active-norm}, put
\[
 S_h=\{c\in\mathbb Z^d:\delta(c)\leq h\}.
\]
The active hyperplanes are the sections $c^\perp\cap P^\circ$ with
$c\in S_h\setminus\{0\}$.

Generic types are the chambers. Every maximal-rank weak type is an
interior vertex of the normalized arrangement: its active normals have
rank $p$, and its positive realizing ray meets $P^\circ$ in one point.
Distinct such weak types give distinct vertices.

We use the following elementary deletion--restriction inequality. For any
finite real hyperplane arrangement $\mathcal A$ in an open convex set $P$,
let $R(\mathcal A,P)$ be its number of chambers and
$V(\mathcal A,P)$ its number of vertices in $P$. We use
$R=V=0$ for an empty $P$, and in dimension zero both numbers are one.
Then
\begin{equation}\label{eq:chambers-dominate-vertices}
 R(\mathcal A,P)\geq V(\mathcal A,P).
\end{equation}
Indeed, induct simultaneously on the ambient dimension and on
$|\mathcal A|$. Delete a hyperplane $H$ and restrict the remaining
arrangement to $P\cap H$.  Write
\[
 \mathcal A^H=\{H\cap H':H'\in\mathcal A\setminus\{H\}\},
\]
with empty intersections and repetitions removed.  The chamber recurrence is
\[
 R(\mathcal A,P)
 =R(\mathcal A\setminus\{H\},P)
  +R(\mathcal A^H,P\cap H).
\]
Every vertex newly present after $H$ is inserted is a vertex of the
restricted arrangement, although a vertex of the restriction may already
have been a vertex before insertion. Hence
\[
 V(\mathcal A,P)
 \leq V(\mathcal A\setminus\{H\},P)
   +V(\mathcal A^H,P\cap H).
\]
The induction hypotheses give \eqref{eq:chambers-dominate-vertices};
the stated conventions handle empty restrictions and dimension zero.
This is the deletion--restriction form of the usual arrangement region formula
\cite{BeckZaslavsky2006}.

Theorem~\ref{thm:rank-diversity-frequency} supplies
$\Theta_k(h^{d(d-1)})$ interior vertices, so
\eqref{eq:chambers-dominate-vertices} gives at least that many chambers.
Finally, each chamber is an open rational polyhedron and contains a
rational point; clearing denominators realizes its strict ordering by an
integer set. This proves the generic lower bound.
\end{proof}

\begin{corollary}[Optimal-order window saturation]
\label{cor:critical-window-saturation}
Fix $k\geq3$.  For each $h\geq2$, put
\[
 H_{h,k}=(k-1)h^{k-2},\qquad
 \mathcal M_{h,k}
 =\{A\subseteq\{0,\ldots,H_{h,k}\}:0\in A,\ |A|=k\}.
\]
The maps from $\mathcal M_{h,k}$ onto the realized weak types and onto the
realized ordered $h$-addition-table types are surjective.  Moreover, as
$h\to\infty$,
\[
 |\mathcal M_{h,k}|=\binom{H_{h,k}}{k-1}
 =\Theta_k\!\left(h^{(k-1)(k-2)}\right),
\]
and both target sets have the same order of magnitude.  Consequently, the
average fibre of either type map has order $\Theta_k(1)$, and choosing one
representative from each fibre gives a positive-proportion subfamily on which
the map is injective.  For either type map, such a full-order subfamily may
be chosen entirely among representatives of maximal-rank weak types, and
hence inside the exact core $[0,h^{k-2}]$.
\end{corollary}

\begin{proof}
Surjectivity follows from Theorem~\ref{thm:compression-master}.  The domain
count is immediate, and the two target counts are
Corollaries~\ref{cor:weak-type-diversity} and
\ref{cor:ordered-generic-type-diversity}.  Choosing one representative of
each type proves the first injective formulation.  Maximal-rank weak types
alone have the full order by Theorem~\ref{thm:rank-diversity-frequency};
choosing one ordered refinement of each and applying the exact core in
Theorem~\ref{thm:compression-master} proves the strengthened final sentence.
\end{proof}

\begin{corollary}[Many types at one cardinality]
\label{cor:many-types-one-cardinality}
For every fixed $k\geq3$ and all sufficiently large $h$, some
$n\in\Rcal(h,k)$ is realized by at least
\[
 c_k h^{(k-1)(k-3)}
\]
pairwise distinct maximal-active-rank weak types.  In particular, for
$k=5$ one cardinality is realized by $\gg h^8$ such types.
\end{corollary}

\begin{proof}
By Theorem~\ref{thm:rank-diversity-frequency}, there are
\[
 \Theta_k\!\left(h^{(k-1)(k-2)}\right)
\]
maximal-rank types.  There are, however, at most
$M_{h,k}=O_k(h^{k-1})$ possible cardinalities.  Pigeonholing proves the
claim.  This does not assert that the chosen cardinality has minimum active
rank $k-2$.
\end{proof}

\begin{corollary}[Bogart--Cu\'ellar's zero evaluation]
\label{cor:bogart-cuellar-type-growth}
Let $b_{m,h}(t)$ be Bogart and Cu\'ellar's quasipolynomial counting
$B_h$-rulers of length $t$ with $m+1$ markings. For every fixed $m\geq1$,
\begin{equation}\label{eq:bogart-cuellar-type-growth}
 (-1)^{m-1}b_{m,h}(0)
 =\Theta_m\!\left(h^{m(m-1)}\right)
 \qquad(h\to\infty).
\end{equation}
\end{corollary}

\begin{proof}
By \cite[Theorem~2.5]{BogartCuellar2025}, the left side of
\eqref{eq:bogart-cuellar-type-growth} is exactly the number of
combinatorial types of $B_h$-rulers with $m+1$ markings. These are the
generic types in Corollary~\ref{cor:ordered-generic-type-diversity} with
$k=m+1$, giving exponent
$(k-1)(k-2)=m(m-1)$ for $m\geq2$. For $m=1$ there is exactly one type,
which gives the same formula.
\end{proof}

The same conclusions hold multiplicatively. For a strictly increasing
set $P=\{p_1,\ldots,p_k\}\subset\mathbb Z_{>0}$, define its weak
$h$-multiplication type on $\mathcal U_{h,k}$ by
\[
 u\sim_P^\times v
 \quad\Longleftrightarrow\quad
 \prod_{i=1}^k p_i^{u_i}=\prod_{i=1}^k p_i^{v_i},
\]
and define its rank as
$\dim_{\mathbb R}\operatorname{span}\{u-v:u\sim_P^\times v\}$.
Its ordered $h$-multiplication type is the total preorder on
$\mathcal U_{h,k}$ defined by
\[
 u\preccurlyeq_P^\times v
 \quad\Longleftrightarrow\quad
 \prod_{i=1}^k p_i^{u_i}\leq\prod_{i=1}^k p_i^{v_i}.
\]
We call this type \emph{generic} when it is tie-free, equivalently when
$u\sim_P^\times v$ implies $u=v$.

\begin{corollary}[Multiplication-table types]
\label{cor:multiplicative-type-diversity}
For fixed $k\geq3$ and $0\leq r\leq k-2$, as $h\to\infty$, the number of weak
$h$-multiplication-table types of rank $r$ is
\[
 \Theta_{k,r}\!\left(h^{(k-1)r}\right).
\]
The total number of weak multiplication types, the total number of ordered
multiplication types, and the total number of generic multiplication types
are all
\begin{equation}\label{eq:multiplicative-type-growth}
 \Theta_k\!\left(h^{(k-1)(k-2)}\right).
\end{equation}
\end{corollary}

\begin{proof}
O'Bryant's fixed-degree weak-type transfer identifies the realized weak type
sets \cite[Theorem~9]{OBryant2025}.  The stronger ordered statement is the
additive--multiplicative part of Corollary~\ref{cor:finite-profile-transfer}:
the complete total preorders agree, with their equality relations, ranks, and
tie-free subfamilies.  Theorem~\ref{thm:rank-diversity-frequency} and
Corollary~\ref{cor:ordered-generic-type-diversity} therefore give all the
asserted counts.
\end{proof}

For four generators, the maximal-rank count also gives an additive derivation
of the growth orders for vertices and regions in the classical
two-dimensional Farey-line diagram, with the vertex lower bound localized in
the open triangle.

Here a \emph{Farey line of order $(N,N)$} means a line whose primitive
integer equation is
\[
 ux+vy+w=0,
 \qquad |u|,|v|\leq N,
\]
and whose intersection with the unit square is a nondegenerate segment, in
the sense of \cite[Definition~1]{KhoshnoudiradFareyLines2015}.

\begin{corollary}[Farey vertices and regions]
\label{cor:farey-sixth-order}
For each integer $N\geq1$, let $\mathrm{FV}(N)$ be the number of distinct
intersections of Farey lines of order $(N,N)$, and let $\mathrm{FR}(N)$ be
the number of connected regions into which those lines cut the open unit
square.  Then, as $N\to\infty$,
\begin{equation}\label{eq:farey-sixth-order}
 \mathrm{FV}(N)=\Theta(N^6),\qquad
 \mathrm{FR}(N)=\Theta(N^6).
\end{equation}
The lower bound for $\mathrm{FV}(N)$ already comes from intersections in
the open triangle $0<x<y<1$.
\end{corollary}

\begin{proof}
Normalize a four-set to $(0,x,y,1)$ with $0<x<y<1$. A homogeneous
relation gives a line
\[
 ux+vy+w=0
\]
with relation vector $(-u-v-w,u,v,w)$ and exact degree
\begin{equation}\label{eq:farey-relation-degree}
 \Delta(u,v,w)
 =\frac12\bigl(
 |u+v+w|+|u|+|v|+|w|\bigr).
\end{equation}
If $\Delta(u,v,w)\leq N$, then $|u|,|v|\leq N$: each coordinate of a
zero-sum vector has absolute value at most the total positive mass
$\Delta$. A line crossing the open triangle crosses the unit square in a
nondegenerate segment, so its primitive equation is a Farey line of order
$(N,N)$ in the sense of
\cite[Definition~1]{KhoshnoudiradFareyLines2015}.
Conversely, suppose that a Farey line of order $(N,N)$ crosses the open
triangle. Its values at the three vertices are
\[
 w,\qquad v+w,\qquad u+v+w,
\]
and zero lies strictly between their minimum and maximum. The total
variation along this two-step path is therefore at least the distance from
the first value to zero plus the distance from zero to the last value:
\[
 |w|+|u+v+w|\leq |v|+|u|.
\]
Using the primitive equation defining the Farey line,
\[
 \Delta(u,v,w)\leq |u|+|v|\leq2N.
\]
Thus, inside the triangle, the active arrangement of degree $N$ is
contained in the Farey arrangement of order $(N,N)$, which in turn is
contained in the active arrangement of degree $2N$.

For $k=4$, the rank-two types in
Theorem~\ref{thm:rank-diversity-frequency} are precisely distinct
intersections of two independent active lines in the open triangle.
Consequently
\[
 \mathrm{FV}(N)\gg N^6.
\]
Khoshnoudirad proved that the number $L(N)$ of Farey lines of order
$(N,N)$ satisfies
\[
 L(N)\sim\frac{2N^3}{\zeta(3)}
 \qquad\text{\cite[Theorem~3]{KhoshnoudiradFareyLines2015}}.
\]
The line-pair bound therefore gives
$\mathrm{FV}(N)\leq\binom{L(N)}2=O(N^6)$.

For the region count, insert the line segments crossing the square one at
a time.  Let $L_{\mathrm{int}}(N)\leq L(N)$ be the number of lines meeting
the open square.  If $m_p$ lines pass through an interior vertex $p$, the
resulting planar-arrangement formula is
\[
 \mathrm{FR}(N)
 =1+L_{\mathrm{int}}(N)+
   \sum_{p\in(0,1)^2}(m_p-1).
\]
Hence every distinct interior vertex contributes at least
one region, while the line-pair bound gives
$\mathrm{FR}(N)=O(L(N)^2)$.  The $\Theta(N^6)$ interior vertices already
constructed in $0<x<y<1$ give the matching lower bound.
\end{proof}

\begin{remark}\label{rem:farey-erratum}
The order-six conclusion does not use the order-five logarithmic estimate
stated in a 2015 follow-up on Farey vertices.  The erratum explains that
the argument treated only the positive-minimum case, omitting the zero case
\cite{KhoshnoudiradFareyVertices2015,KhoshnoudiradErratum2016}.  The proof
above instead combines the sharp line count with the independent
rank-diversity lower bound.  The same unlabelled exponent also follows
abstractly from the primitive-arrangement cell theorem of
B\'ar\'any--Bureaux--Lund \cite[Theorem~5.6]{BaranyBureauxLund2018}; the
contribution here is the explicit comparison in Farey coordinates between
the two arrangements, together with localization inside the open triangle.
\end{remark}

The arrangement has a second geometric interpretation after quotienting
the two affine symmetries of an exponent set: its cells are Hilbert-constant
and refine the degree-$h$ Hilbert-value loci of projective monomial curves.
On each normalized cell, active rank is its codimension in exponent-shape
space.

\subsection{Hilbert strata and the universal principal skeleton}

\begin{corollary}[Hilbert-value codimension and realization rarity]
\label{cor:hilbert-strata}
Fix $h\geq1$, $k\geq2$, and a field $K$.  Recall the open normalized
exponent-shape simplex
\[
 \mathcal S_k=
 \{x\in\mathbb R^k:0=x_1<x_2<\cdots<x_k=1\}
\]
and put
\[
 N_h(x)=\left|\{u\cdot x:u\in\mathcal U_{h,k}\}\right|.
\]
For every $n\in\mathcal R(h,k)$, the locus
\[
 \mathcal H_{h,k}(n)=\{x\in\mathcal S_k:N_h(x)=n\}
\]
is a finite disjoint union of relatively open rational polytopes, and
\begin{equation}\label{eq:hilbert-stratum-dimension}
 \max\{\dim P:P\text{ is a cell of }\mathcal H_{h,k}(n)\}
 =k-2-r_{h,k}(n).
\end{equation}
For rational $x\in\mathcal S_k$, choose $D\geq1$ with
$Dx\in\mathbb Z^k$ and let $C_x\subset\mathbb P^{k-1}$ be the projective
monomial curve defined by the kernel of
\[
 K[X_1,\ldots,X_k]\longrightarrow K[s,t],
 \qquad X_i\longmapsto s^{D(1-x_i)}t^{Dx_i}.
\]
The kernel is independent of the choice of $D$, and
\[
 H_{C_x}(h)=N_h(x).
\]
Consequently, $r_{h,k}(n)$ is exactly the codimension in exponent-shape
space of the largest stratum with degree-$h$ Hilbert value $n$.
If $h\geq2$, $k\geq3$, and $A_q$ is uniform in $\binom{[q]}k$, then
\begin{equation}\label{eq:rank-codimension-rarity}
 \mathbb P(|hA_q|=n)
 =k!c_{h,k,n}q^{-r_{h,k}(n)}
  +O_{h,k}(q^{-r_{h,k}(n)-1}),
\end{equation}
and hence the same integer has the four equivalent descriptions
\begin{equation}\label{eq:four-way-rank-identity}
 \boxed{
 r_{h,k}(n)
 =k-\deg F^{(k)}_{h,n}
 =\operatorname{codim}_{\mathcal S_k}\mathcal H_{h,k}(n)
 =-\lim_{q\to\infty}
   \frac{\log\mathbb P(|hA_q|=n)}{\log q}.}
\end{equation}
Here the codimension of a finite polyhedral union means the codimension of
its largest-dimensional cell.
\end{corollary}

\begin{proof}
If $h=1$, then $\mathcal R(1,k)=\{k\}$, every point has active rank zero,
and $\mathcal H_{1,k}(k)=\mathcal S_k$ has dimension $k-2$.  If $k=2$ and
$h\geq2$, then $\mathcal S_2=\{(0,1)\}$,
$\mathcal R(h,2)=\{h+1\}$, and again the active rank is zero.  Thus the
assertion holds in both edge regimes.  We may henceforth assume
$h\geq2$ and $k\geq3$.

Intersect the arrangement strata $X^\circ$ from the proof of
Theorem~\ref{thm:rank-frequency} with $\mathcal S_k$ and decompose along the
remaining rational hyperplanes.  This gives a finite disjoint union of
relatively open rational polytopes on each of which $N_h(x)$ is constant.
An active-rank-$r$ stratum has dimension $k-r$ before normalization.
Translation by ${\bf1}$ and positive dilation preserve every active relation,
and the normalization
\[
 x\longmapsto\frac{x-x_1{\bf1}}{x_k-x_1}
\]
has two-dimensional affine-equivalence fibres.  Its intersection with
$\mathcal S_k$ therefore has dimension $k-r-2$.  Taking the largest
dimension over strata with value $n$ proves
\eqref{eq:hilbert-stratum-dimension}.

For rational $x$, the degree-$h$ monomial $X^u$ maps to
\[
 s^{hD-D(u\cdot x)}t^{D(u\cdot x)}.
\]
Two such monomials have the same image exactly when their dot products with
$x$ agree.  Hence the degree-$h$ Hilbert function is $N_h(x)$.  The same
criterion shows that replacing $D$ by another common denominator does not
change the kernel.  Finally, every integral exponent set becomes a rational
point of $\mathcal S_k$ after translation and positive dilation, and every
rational point clears to an integral exponent set.
For $h\geq2$ and $k\geq3$, divide
\eqref{eq:rank-frequency-asymptotic} by
$\binom qk=q^k/k!+O_k(q^{k-1})$ to obtain
\eqref{eq:rank-codimension-rarity}.  Theorem~\ref{thm:rank-frequency} and
\eqref{eq:hilbert-stratum-dimension} then give
\eqref{eq:four-way-rank-identity}.
\end{proof}

\begin{corollary}[Universal principal-rank skeleton]
\label{cor:universal-principal-skeleton}
Fix $h\geq2$ and $k\geq3$.  Then
\begin{equation}\label{eq:universal-principal-skeleton}
 \Rcal^{[0]}(h,k)=\{M_{h,k}\},
 \qquad
 \Rcal^{[1]}(h,k)
 =\{M_{h,k}-M_{h-r,k}:2\leq r\leq h\}.
\end{equation}
In particular, exactly $h$ attainable cardinalities have minimum active
rank at most one.  They are exactly the values whose realization-frequency
quasipolynomial has degree greater than $k-2$: the maximum has degree $k$,
the other $h-1$ values have degree $k-1$, and every remaining attainable
cardinality has degree at most $k-2$.
\end{corollary}

\begin{proof}
Rank zero means that no two members of $\mathcal U_{h,k}$ have the same
image, and hence gives the unique value $M_{h,k}$.  If the active relation
space of a fixed $k$-set $A$ has rank one, let $z$ be its primitive active
direction and put $r=\rho(z)$.  Then $r\leq h$; moreover, $r=1$ would force
$z=\pm(e_i-e_j)$ and hence two elements of $A$ to coincide.  Thus
$2\leq r\leq h$.  Over any field $K$,
put
$S=K[x_1,\ldots,x_k]$ and
\[
 f_z=x^{z^+}-x^{z^-}.
\]
Every difference of two degree-$h$ exponent vectors in the same fibre is an
integral multiple of the primitive vector $z$.  After cancelling the common
monomial, the corresponding binomial is divisible by $f_z$, since
$X^m-Y^m$ is divisible by $X-Y$.  Conversely, every degree-$h$ multiple of
$f_z$ is a toric relation.  Hence the degree-$h$ toric-ideal piece is
$f_zS_{h-r}$, of dimension $M_{h-r,k}$, and therefore
\[
 |hA|=M_{h,k}-M_{h-r,k}.
\]
Conversely, for each $2\leq r\leq h$, the hyperplane with primitive normal
\[
 (r-1,-r,1,0,\ldots,0)
\]
meets the ordered chamber, and a generic point of this section lies on no
other active hyperplane.  The rational-stratum realization in the proof of
Theorem~\ref{thm:rank-frequency} therefore supplies an integer realization
of active rank one.  The displayed values are distinct because
$M_{0,k},M_{1,k},\ldots,M_{h-2,k}$ are strictly increasing.  This proves
\eqref{eq:universal-principal-skeleton}; the frequency assertions follow
from Theorem~\ref{thm:rank-frequency}.
\end{proof}

For $k=3$, the bound $r_h(A)\leq1$ for every three-element set $A$ shows that
Corollary~\ref{cor:universal-principal-skeleton} exhausts the entire
spectrum.  For $k=4$, the maximum active rank is two, so
Theorem~\ref{thm:rank-frequency} gives exactly the
$q^4,q^3,q^2$ trichotomy.  Fix
\[
 n\in\Rcal(h,4)\setminus
 \bigl(\{M_h\}\cup\{M_h-M_{h-r}:2\leq r\leq h\}\bigr).
\]
By Corollary~\ref{cor:universal-principal-skeleton}, every realization of
$n$ has active rank two.  Let $\mathcal S_{h,n}$ be the finite set of all
primitive normalized shapes $B=\{0<b<c<d\}$ satisfying
$\gcd(b,c,d)=1$ and $|hB|=n$; reflected shapes $d-B$ are retained as
separate elements when distinct.  This set is finite by
Lemma~\ref{lem:product}.  For a fixed shape $B$, every realization in
$[q]$ is uniquely $t+gB$ with $g\in\mathbb Z_{\geq1}$, and hence its
number of copies is
\[
 \sum_{1\leq g\leq\lfloor(q-1)/d\rfloor}(q-gd)
 =\frac{q^2}{2d}+O_d(q).
\]
Thus each corresponding affine-shape plane, spanned by ${\bf1}$ and $B$,
has leading coefficient $1/(2d)$, and
$c_{h,4,n}=\lambda_{h,n}$ is given by
\[
 \lambda_{h,n}=\frac12\sum_{B\in\mathcal S_{h,n}}\frac1{\max B}.
\]

\part{Spectrum--geometry synthesis, rigidity, and consequences}
\label{part:rigidity}

Parts~\ref{part:spectrum} and~\ref{part:structure} provide complementary
information about the same family.  This part combines those descriptions;
the subsequent rigidity, transfer, and distribution results refine the
synthesis or transfer it to other settings.

\newpage
\section{Spectrum--geometry synthesis and the rank--defect frontier}
\label{sec:collision-synthesis}

\subsection{The sharp rank--defect frontier}

\begin{theorem}[Sharp rank--defect frontier]
\label{thm:rank-defect-frontier}
For $h\geq2$, $k\geq3$, and $0\leq r\leq k-2$,
\begin{equation}\label{eq:rank-defect-frontier}
 \min_{\substack{A\subset\mathbb Z,\ |A|=k\\r_h(A)=r}}
 \left(M_{h,k}-|hA|\right)=r.
\end{equation}
Equivalently,
\[
 \max_{\substack{A\subset\mathbb Z,\ |A|=k\\r_h(A)=r}}
 |hA|=M_{h,k}-r.
\]
Equality holds precisely when $\operatorname{cnul}_h(A)=0$.
\end{theorem}

\begin{proof}
The lower bound is \eqref{eq:defect-rank-nullity}.  For equality, begin with
$A^{(0)}=\{0,1\}$.  Apply $r$ suspensions of degree $h$, always
using $c=1$.  Then apply $k-r-2$ further suspensions of arbitrary degrees
$\ell_i>h$, again using $c=1$.  Every intermediate set contains $1$, so all
these suspensions are valid, primitive, and add a point distinct from the
scaled old set.  The iterated $K$-polynomial is
\[
 (1-z^h)^r\prod_{i=1}^{k-r-2}(1-z^{\ell_i}).
\]
Since $h\geq2$ and every $\ell_i>h$, its degree-$h$ Hilbert coefficient is
\[
 |hA|=[z^h]\frac{(1-z^h)^r
        \prod_i(1-z^{\ell_i})}{(1-z)^k}
      =M_{h,k}-r.
\]
The relation-filtration formula gives $r_h(A)=r$: precisely the $r$
degree-$h$ suspension relations are visible, while all later relations are
born above degree $h$.  Thus $\Delta_h(A)=r$, proving sharpness.
\end{proof}

\begin{remark}\label{rem:rank-defect-label-scope}
Theorem~\ref{thm:rank-defect-frontier} constructs a rank-$r$ witness for the
label $M_{h,k}-r$.  It does not assert that this label has minimum active rank
$r$ among all of its representations.
\end{remark}

\begin{corollary}[Largest nonprincipal value]
\label{cor:largest-nonprincipal-value}
For every $h\geq2$ and $k\geq4$,
\[
 r_{h,k}(M_{h,k}-2)=2.
\]
Thus $M_{h,k}-2$ is the largest attainable cardinality outside the universal
principal skeleton.  It has a primitive complete-intersection witness, its
realization quasipolynomial has degree $k-2$, and its largest normalized
Hilbert-value stratum has dimension $k-4$.
\end{corollary}

\begin{proof}
Theorem~\ref{thm:rank-defect-frontier} with $r=2$ supplies a rank-two witness,
obtained by iterated regular cyclic extensions and hence a complete
intersection.  Corollary~\ref{cor:universal-principal-skeleton} shows that a
minimum-rank-zero or rank-one deficit is either zero or $M_{j,k}$ for some
$0\leq j\leq h-2$.  Since $M_{0,k}=1$, $M_{1,k}=k>2$, and $M_{j,k}$ is
increasing in $j$, deficit two is not principal.  Moreover, $M_{h,k}-1$ is
the principal value corresponding to $j=0$; since cardinalities are integral,
the attained value $M_{h,k}-2$ is the largest value outside the principal
skeleton.  The assertions on frequency and Hilbert strata follow from
Theorem~\ref{thm:rank-frequency} and Corollary~\ref{cor:hilbert-strata}.
\end{proof}

\subsection{Proof of the synthesis theorem}

\begin{proof}[Proof of Theorem~\ref{thm:maximal-rank-saturation}]
For part~\textup{(i)}, the endpoint construction begins with a four-point
Hilbert--Burch chart of active rank two and applies $k-4$ prime suspensions
of degrees below $h$.  Equation~\eqref{eq:suspension-rank} gives final rank
$2+(k-4)=k-2$.  When $k=4$, the birth-degree and diameter assertions are
verified in the proof of Theorem~\ref{thm:endpoint-spectrum}; when $k\geq5$,
they follow from \eqref{eq:endpoint-critical-birth-scales}.  Thus in every
case
\[
 p_i(A)\asymp_k h,
 \qquad
 \ndiam(A)\asymp_k\prod_i p_i(A)\asymp_k h^{k-2};
\]
Corollary~\ref{cor:suspension-critical-core} shows that these suspensions
preserve the ratio
\[
 \frac{\ndiam(A)}{\prod_i p_i(A)}.
\]
Equation~
\eqref{eq:suspension-toric} and
Corollary~\ref{cor:endpoint-acm-type-two} give the homological assertions.
Let
$\mathcal L^{\mathrm{end}}_{h,k}$ be the resulting set of distinct labels.
Theorem~\ref{thm:endpoint-spectrum} supplies its number and localization;
choose $\alpha_k,\beta_k$ there to be the constants furnished by this same
endpoint family.  This proves \eqref{eq:maximal-rank-label-saturation}.

For part~\textup{(ii)}, retain the explicit endpoint witnesses chosen in
part~\textup{(i)}.  For some $C_k>0$,
\[
 |\mathcal L^{\mathrm{end}}_{h,k}|\geq
 h^{k-1}\exp\!\left(-C_k\frac{\log h}{\log\log h}\right).
\]
Corollary~\ref{cor:universal-principal-skeleton} says that exactly $h$ labels
in the whole spectrum have minimum active rank at most one.  Hence at most
$h$ elements of $\mathcal L^{\mathrm{end}}_{h,k}$ have minimum rank at most
one.  The displayed lower bound divided by $h$ is
$h^{k-2-C_k/\log\log h}\to\infty$.  Delete those labels and partition the
remainder among the $k-3$ ranks $2,\ldots,k-2$.  One fibre has the lower
bound \eqref{eq:dominant-rank-layer-count}, with fixed factors absorbed in
the $O_k$-term; call its rank $r^\ast_{h,k}$ and its label set
$\mathcal E_{h,k}$.  Deleting labels does not alter their chosen endpoint
witnesses.

Theorem~\ref{thm:rank-frequency} gives every summand of
$F^{(k)}_{h,\mathcal E}$ degree $k-r^\ast_{h,k}$ with positive rational
leading coefficient.  Their finite sum has the same degree and no
leading-term cancellation.  Division by $\binom qk$ proves
\eqref{eq:dominant-rank-layer-rarity}; Corollary~\ref{cor:hilbert-strata}
gives the shape dimension.  The lower bound and the ambient upper bound
$|\Rcal(h,k)|=O_k(h^{k-1})$ give
\eqref{eq:dominant-intrinsic-rank-exponent}.  If $k=4$, rank two is the only
remaining rank and
$|\Rcal^{[2]}(h,4)|=|\Rcal(h,4)|-h=\Theta(h^3)$.
\end{proof}

\begin{remark}[Witnesses, types, and labels]
\label{rem:maximal-rank-saturation-scope}
Theorem~\ref{thm:maximal-rank-saturation} says that every label in the family
$\mathcal L^{\mathrm{end}}_{h,k}$ of
\eqref{eq:maximal-rank-label-saturation} is realized by a maximal-rank type.
It does not assert that every such label has minimum active rank $k-2$.
Part~\textup{(ii)} instead selects one intrinsic minimum-rank layer carrying
the full exponent, but that layer may depend on $h$.  Consequently:
\begin{itemize}
\item the selected rank is not asserted to stabilize or to equal $k-2$;
\item the probability law in part~\textup{(ii)} fixes $h$ before
$q\to\infty$ and asserts no joint $(h,q)$ limit;
\item part~\textup{(ii)} asserts polynomial-exponent saturation only; it makes
no positive-density claim for $k\geq5$.
\end{itemize}
\end{remark}

We next determine the sharp large-diameter boundary and the optimal
four-point $B_h$ rulers.

\section{Four-point diameter rigidity and optimal rulers}
\label{sec:four-point-rigidity}

\begin{theorem}[Large-diameter tail]\label{thm:tail}
For every $h\geq2$,
\[
 \{|hA|:A\subset\Z,\ |A|=4,\ \ndiam A>h^2\}
 =\{M_h\}\cup\{M_h-M_{h-r}:2\leq r\leq h\}.
\]
Every listed value occurs at arbitrarily large normalized diameter.
\end{theorem}

\begin{proof}
For an arbitrary four-set $A$, pass to
\[
 B=\frac{A-\min A}{g(A)}.
\]
This is an integral primitive set with $\max B=\ndiam(A)$.  Since every
relation vector has coordinate sum zero, $z\cdot B=(z\cdot A)/g(A)$;
hence the relation lattice, relation degrees, and active rank are unchanged.
The affine map also bijects $hA$ with $hB$.  Lemma~\ref{lem:product} applied
to $B$ therefore shows that at normalized diameter greater than $h^2$ the
active rank is at most one.  Proposition~\ref{prop:rankone} gives the upper
inclusion.

For the maximum, take
\[
 A_N=\{0,1,N,N^2\},\qquad N>h+1.
\]
An active relation in the last three coordinates satisfies
\[
 \alpha+N\beta+N^2\gamma=0,
 \qquad |\alpha|,|\beta|,|\gamma|\leq h.
\]
Since $N^2>h(N+1)$, one has $\gamma=0$, and then $N>h$ forces
$\beta=\alpha=0$.  Hence $A_N$ is a $B_h$-set and $|hA_N|=M_h$.

For $2\leq r\leq h$, take
\[
 A_{N,r}=\{0,1,N,N+r-1\},\qquad N>h^2.
\]
An active relation satisfies
\begin{equation}\label{eq:largewitness}
 N(\beta+\gamma)+\alpha+(r-1)\gamma=0.
\end{equation}
The second summand in \eqref{eq:largewitness} has absolute value at most
$h+(r-1)h\leq h^2<N$.  Hence $\beta+\gamma=0$ and
$\alpha+(r-1)\gamma=0$.  Every active relation is a multiple of
\[
 (r-1,-(r-1),-1,1),
\]
whose degree is $r$.  Proposition~\ref{prop:rankone} gives
$|hA_{N,r}|=M_h-M_{h-r}$.  Letting $N$ grow proves the result.
\end{proof}

\subsection{The rank-zero endpoint: optimal four-mark rulers}

We first determine the least normalized diameter of a four-element
$B_h$-set.  Put
\[
 \eta_0(h)=\min\{\ndiam(A):A\subset\Z,\ |A|=4,\ |hA|=M_h\}.
\]
In the interval-size convention of \cite{OBryantThick2024}, let
$R_h^{-1}(4)$ denote the least size of a positive-label interval supporting
a four-mark $B_h$ ruler.  Then $\eta_0(h)+1=R_h^{-1}(4)$.  Lam and Duan
studied this fixed-order optimization problem \cite{LamDuan1989}.  In the
present normalization, the values $\eta_0(3)=11$ and $\eta_0(4)=15$ are
recorded in later tabulations
\cite{OBryantThick2024,MartosDelgadoTrujillo2021}.  The equivalent
positive-label formula was
recorded as a conjecture in Jens Vo{\ss}'s 2013 OEIS entry A227589;
a conjectured equivalent recurrence was contributed by Colin Barker, and Martin Fuller
later supplied the construction below and reported exhaustive optimality
through $h=53$ \cite{OEISA227589}.  We prove the conjectured formula
uniformly in $h$.

\begin{lemma}[Four-point relation-ball bound]
\label{lem:fourpoint-relation-ball}
Let
\[
 A=\{0,a,b,D\},\qquad 0<a<b<D,\qquad \gcd(a,b,D)=1,
\]
and suppose that every nonzero homogeneous relation of $A$ has relation
degree at least $H$, where $H\in\mathbb R$ and $H\geq2$.  Then
\[
 D\geq \binom{H+1}{2}.
\]
If equality holds, then, for the gaps
\[
 g_1=a,\qquad g_2=b-a,\qquad g_3=D-b,
\]
one has
\begin{equation}\label{eq:fourpointcriticalgaps}
 g_1+g_3=H,\qquad
 \gcd(g_1,g_3)=1,\qquad
 g_2=\frac{H(H-1)}2.
\end{equation}
\end{lemma}

\begin{proof}
Relation degrees are integers.  Put $H_0=\lceil H\rceil$.  The hypothesis
remains true with $H_0$ in place of $H$.  Once the assertion is proved for
integral thresholds, it gives
\[
 D\geq\binom{H_0+1}{2}\geq\binom{H+1}{2}.
\]
Equality in the original bound forces $H_0=H$, because
$x\mapsto\binom{x+1}{2}$ is strictly increasing on $[2,\infty)$; the
integral equality classification then applies.  It therefore suffices below
to assume $H\in\mathbb Z_{\geq2}$.

Project the homogeneous relation lattice $L_A$ to its middle two
coordinates.  Its image is
\[
 \Lambda=\{(x,y)\in\Z^2:ax+by\equiv0\pmod D\}.
\]
The map $(x,y)\mapsto ax+by\pmod D$ is onto, so
\begin{equation}\label{eq:fourpointdeterminant}
 \det\Lambda=[\Z^2:\Lambda]=D.
\end{equation}
If $q=(ax+by)/D$, the corresponding full relation is
\[
 (q-x-y,x,y,-q).
\]
Putting $\alpha=a/D$ and $\beta=b/D$, its degree is
\[
 F_{\alpha,\beta}(x,y)
 =\frac12\bigl(
 |x|+|y|+|\alpha x+\beta y|
 +|(1-\alpha)x+(1-\beta)y|\bigr).
\]
For every $t>0$, put
\[
 \Omega_t=\{(x,y)\in\mathbb R^2:F_{\alpha,\beta}(x,y)<t\}.
\]
Consequently $\Omega_H$ contains no nonzero point of $\Lambda$.
The functional $F_{\alpha,\beta}$ is a norm: it is a sum of absolute values
of linear forms and satisfies
\[
 F_{\alpha,\beta}(x,y)\geq\frac12(|x|+|y|).
\]
For $0<\varepsilon<H$, the closed centrally symmetric convex body
\[
 \overline\Omega_{H-\varepsilon}
 =\{(x,y)\in\mathbb R^2:
       F_{\alpha,\beta}(x,y)\leq H-\varepsilon\}
\]
is contained in $\Omega_H$ and hence contains no nonzero point of
$\Lambda$.  Minkowski's theorem therefore gives
\[
 (H-\varepsilon)^2\operatorname{area}(\Omega_1)
 \leq4\det\Lambda=4D.
\]
Letting $\varepsilon\downarrow0$ yields
\begin{equation}\label{eq:fourpointMinkowski}
 H^2\operatorname{area}(\Omega_1)\leq4D.
\end{equation}

Write
\[
 L_1=\alpha x+\beta y,
 \qquad
 L_2=(1-\alpha)x+(1-\beta)y.
\]
Since $0<\alpha<\beta<1$, the four consecutive sign chambers in the
closed lower half-plane, read clockwise from the positive to the
negative $x$-axis, give the following linear pieces:
\[
\begin{array}{c|c}
 (x,y,L_1,L_2) & F_{\alpha,\beta}(x,y)\\ \hline
 (+,-,+,+) & x\\
 (+,-,-,+) & (1-\alpha)x-\beta y\\
 (+,-,-,-) & -y\\
 (-,-,-,-) & -x-y.
\end{array}
\]
Thus the boundary of the closed ball
$\overline\Omega_1=\{F_{\alpha,\beta}\leq1\}$ in the lower half-plane has
the five consecutive vertices
\[
 (1,0),\quad
 \left(1,-\frac{\alpha}{\beta}\right),\quad
 \left(\frac{1-\beta}{1-\alpha},-1\right),\quad
 (0,-1),\quad
 (-1,0).
\]
The table also shows that there are no further boundary changes there;
central symmetry supplies the other three vertices.  Applying the shoelace
formula to this octagon gives
\[
 \operatorname{area}(\Omega_1)
 =\operatorname{area}(\overline\Omega_1)
 =2+\frac{\alpha}{\beta}
   +\frac{1-\beta}{1-\alpha}
   -\frac{\alpha(1-\beta)}{\beta(1-\alpha)}.
\]
Substituting $\alpha=g_1/D$ and $\beta=(g_1+g_2)/D$ yields
\begin{equation}\label{eq:fourpointrelationballarea}
 \operatorname{area}(\Omega_1)
 =2+\frac{g_1g_2+g_1g_3+g_2g_3}
          {(g_1+g_2)(g_2+g_3)}.
\end{equation}
Let $s=g_1+g_3$.  The primitive outer-gap relation
\[
 \frac1{\gcd(g_1,g_3)}(-g_3,g_3,g_1,-g_1)
\]
has degree $s/\gcd(g_1,g_3)$, and hence $s\geq H$.  Moreover,
\[
 \frac{g_1g_2+g_1g_3+g_2g_3}
      {(g_1+g_2)(g_2+g_3)}
 -\frac{s}{D}
 =
 \frac{g_1g_2g_3}
 {D(g_1+g_2)(g_2+g_3)}>0.
\]
Combining this with \eqref{eq:fourpointMinkowski} gives
\[
 4D^2-2H^2D-H^3>0.
\]
For $C_H=\binom{H+1}{2}$, the left side equals $H^2$ at $D=C_H$
and $-H^2-4H+4$ at $D=C_H-1$.  Its positive root lies strictly
between these two integers, proving $D\geq C_H$.

Suppose now that $D=C_H$.  If $s\geq H+1$, then
\[
 H^2\operatorname{area}(\Omega_1)
 >
 H^2\left(2+\frac{H+1}{C_H}\right)
 =4C_H,
\]
contrary to \eqref{eq:fourpointMinkowski}.  Thus $s=H$.  The
outer-gap relation then forces $\gcd(g_1,g_3)=1$, and
$g_2=D-s=H(H-1)/2$, proving \eqref{eq:fourpointcriticalgaps}.
\end{proof}

\begin{theorem}[Optimal four-mark $B_h$ ruler]
\label{thm:optimal-four-mark}
For every $h\geq2$, put
\[
 D_h^{\max}=\binom{h+2}{2}+\mathbf 1_{\{2\nmid h\}}.
\]
Then
\begin{equation}\label{eq:optimal-four-mark}
 \boxed{\qquad \eta_0(h)=D_h^{\max}.\qquad}
\end{equation}
Writing
\[
 Q_h=1+\binom{h+1}{2},
\]
the minimum is attained by
\begin{equation}\label{eq:optimal-four-mark-witness}
 A_h^\star=\{0,1,Q_h,D_h^{\max}\}.
\end{equation}
Equivalently,
\[
 R_h^{-1}(4)=\binom{h+2}{2}+1+\mathbf 1_{\{2\nmid h\}}
\]
in the positive-label interval convention.
\end{theorem}

\begin{proof}
Normalize a four-element $B_h$-set as
\[
 A=\{0,a,b,D\},\qquad 0<a<b<D,\qquad \gcd(a,b,D)=1.
\]
The $B_h$ property says precisely that every nonzero homogeneous relation
has degree at least $H=h+1$.  Lemma~
\ref{lem:fourpoint-relation-ball} gives
\[
 D\geq C_H:=\binom{H+1}{2}=\binom{h+2}{2}.
\]
This is the required lower bound when $h$ is even.

Suppose that $h$ is odd and equality $D=C_H$ holds.  Write
$H=2\ell$, where $\ell\geq2$.  By
\eqref{eq:fourpointcriticalgaps}, after reflecting $A$ if necessary,
\[
 g_1=t\leq\ell,\qquad
 g_2=\ell(2\ell-1),\qquad
 g_3=2\ell-t,\qquad
 \gcd(t,2\ell)=1.
\]
Thus $t=2j+1$ is odd.  If
$c=(c_1,c_2,c_3)$ satisfies
$c_1g_1+c_2g_2+c_3g_3=0$, then
\[
 z(c)=(-c_1,c_1-c_2,c_2-c_3,c_3)\in L_A
\]
and
\[
 \rho(z(c))
 =\frac{|c_1|+|c_1-c_2|+|c_2-c_3|+|c_3|}{2}.
\]
For
\[
 c=(-j,1,-\ell-j)
\]
one checks that $c\cdot(g_1,g_2,g_3)=0$ and
$\rho(z(c))=\ell+t$.  If $t<\ell$, this is below $H$, a
contradiction.  If $t=\ell$, then
$\gcd(t,2\ell)=\ell$, forcing $\ell=1$, again a contradiction.
Hence $D\geq C_H+1$ when $h$ is odd.

It remains to verify \eqref{eq:optimal-four-mark-witness}.  Put
\[
 Q=Q_h,\qquad m=D_h^{\max}-Q=2\left\lceil\frac h2\right\rceil,
 \qquad A_h^\star=\{0,1,Q,Q+m\}.
\]
Every relation $z=(z_0,z_1,z_2,z_3)\in L_{A_h^\star}$ can be
parametrized by
\[
 u=z_2+z_3,\qquad q=-z_3,\qquad s=Qu-mq
\]
as
\begin{equation}\label{eq:optimal-ruler-parametrization}
 z=(s-u,-s,u+q,-q).
\end{equation}
If $u=0$, every nonzero such relation has
\[
 \rho(z)=(m+1)|q|\geq h+1.
\]
After replacing $z$ by $-z$, assume $u>0$.  If $q<0$, then
$s=Qu+m|q|$ and
\[
 \rho(z)\geq s\geq Q>h.
\]
It remains to consider $q\geq0$.  From
\eqref{eq:optimal-ruler-parametrization},
\begin{equation}\label{eq:optimal-ruler-degree}
 \rho(z)=u+q+\operatorname{dist}(Qu-mq,[0,u]).
\end{equation}
If $\rho(z)\leq h$, then $u+q\leq h$ and
$Qu-mq\leq h-q$.  Hence
\[
 (Q+m-1)u\leq hm.
\]
But
\[
 \frac{hm}{Q+m-1}
 =
 \begin{cases}
  \dfrac{2h}{h+3},&h\ \textup{even},\\[4pt]
  \dfrac{2h}{h+2},&h\ \textup{odd},
 \end{cases}
 <2.
\]
For $h=2$ this already contradicts $u\geq1$; otherwise it forces
$u=1$.

Let $r=\lfloor h/2\rfloor$.  If $0\leq q\leq r$, then
$Q-mq>1$, and \eqref{eq:optimal-ruler-degree} gives
\[
 \rho(z)=Q-(m-1)q
 \geq Q-(m-1)r=h+1.
\]
If $q\geq r+1$, then $Q-mq\leq0$, and
\[
 \rho(z)=(m+1)q+1-Q
 \geq(m+1)(r+1)+1-Q=h+1.
\]
Thus no nonzero relation has degree at most $h$, so $A_h^\star$
is a $B_h$-set.  Its diameter is $Q+m=D_h^{\max}$, completing the proof.
\end{proof}

\subsection{Compression into the quadratic core}

The nonmaximal tail values all have small witnesses.  For $2\leq r\leq h$
put
\begin{equation}\label{eq:smallwitness}
 C_{h,r}=\{0,1,(r-1)h+1,(r-1)h+r\}.
\end{equation}
Its diameter is at most $h^2$.  Write $m=(r-1)h+1$.  An active relation
satisfies
\[
 m(\beta+\gamma)+\alpha+(r-1)\gamma=0.
\]
By \eqref{eq:weighteddegree}, the second summand has absolute value at most
$(r-1)h=m-1$.  Since the first summand is divisible by $m$, both summands
must vanish.  Hence $\beta+\gamma=0$ and
$\alpha+(r-1)\gamma=0$; homogeneity then shows that every relation of degree
at most $h$ is an integral multiple of the primitive vector
\[
 v_r=(r-1,-(r-1),-1,1),
\]
whose relation degree is $r\leq h$.  Thus the active relation span has rank
one.  Therefore
\begin{equation}\label{eq:smallwitnessvalue}
 |hC_{h,r}|=M_h-M_{h-r}.
\end{equation}

\begin{corollary}[Quadratic core]\label{cor:core}
For every $h\geq2$,
\[
 \Rcal(h,4)=
 \{|hA|:A\subset\Z,\ |A|=4,\ \ndiam A\leq h^2\}\cup\{M_h\}.
\]
In particular, every nonmaximal value has a representative of normalized
diameter at most $h^2$.
\end{corollary}

\begin{proof}
Normalize a representative.  If its diameter is at most $h^2$, there is
nothing to prove.  Otherwise Theorem~\ref{thm:tail} applies, and each
nonmaximal tail value has the witness \eqref{eq:smallwitness}.
\end{proof}

For completeness, define the diameter-compression parameter
\[
 \nu(h,4)=\min\{D\in\mathbb Z_{\geq0}:\Rcal(h,4)=
 \{|hA|:A\subset[0,D]\cap\Z,\ |A|=4\}\}.
\]
Thus $\nu(h,4)=N(h,4)-1$ in Nathanson's interval-size convention
\cite{NathansonCompression}.

\begin{corollary}[Endpoint-sharp quadratic compression]
\label{cor:compression}
For every $h\geq2$,
\begin{equation}\label{eq:compressionbounds}
 D_h^{\max}\leq\nu(h,4)\leq\max\{h^2,D_h^{\max}\}.
\end{equation}
Consequently
\[
 \nu(2,4)=6,\qquad \nu(3,4)=11,
\]
while for every $h\geq4$,
\[
 \binom{h+2}{2}+\mathbf 1_{\{2\nmid h\}}
 \leq\nu(h,4)\leq h^2.
\]
In particular, $\nu(h,4)=\Theta(h^2)$.
\end{corollary}

Corollary~\ref{cor:optimal-finite-modeling} already determines the
quadratic order of $N(h,4)$ as part of the all-$k$ theorem.  The present
four-point results identify the exact minimum height of the maximal label,
improve the leading constant in the compression lower bound from $1/6$ to
$1/2$, and retain the $h^2$ universal height for every nonmaximal label.
The greedy endpoint witness used previously has diameter $h^2+h+1$
\cite{NathansonGreedyThird}; Theorem~\ref{thm:optimal-four-mark} replaces it
by the exact value $D_h^{\max}$.

\begin{proof}
Every witness for $M_h$ is a four-element $B_h$-set, so
Theorem~\ref{thm:optimal-four-mark} gives
$\nu(h,4)\geq D_h^{\max}$.  The same theorem realizes $M_h$ inside
$[0,D_h^{\max}]$, while Corollary~\ref{cor:core} realizes every nonmaximal
value inside $[0,h^2]$.  This proves the upper bound.  Finally,
$D_h^{\max}>h^2$ for $h=2,3$, whereas $D_h^{\max}\leq h^2$ for
$h\geq4$.
\end{proof}

Under Nathanson's convention $N(h,4)=\nu(h,4)+1$, both sides of
\eqref{eq:compressionbounds} increase by one.

Whether $\nu(h,4)=D_h^{\max}$ for every $h\geq2$---equivalently, whether
every nonmaximal label can always be modeled at the rank-zero optimal
height---is the remaining endpoint-sharp compression problem.

\begin{proposition}[Sharpness of the cutoff]\label{prop:boundary}
For every $h\geq2$, the set $A_h^\partial=\{0,1,h,h^2\}$ satisfies
\[
 |hA_h^\partial|=M_h-2.
\]
This value does not occur in the tail list in Theorem~\ref{thm:tail}.
\end{proposition}

\begin{proof}
Deleting coordinate zero identifies the homogeneous relation lattice
$L_{A_h^\partial}$ with
\[
 \ker_{\mathbb Z}(1,h,h^2)
 =\{(z_1,z_2,z_3)\in\mathbb Z^3:
       z_1+hz_2+h^2z_3=0\}.
\]
The lattice contains the independent degree-$h$ relations
\[
 v=(h-1,-h,1,0),\qquad w=(h-1,0,-h,1).
\]
Their projections are $v'=(-h,1,0)$ and $w'=(0,-h,1)$, and
\[
 v'\times w'=(1,h,h^2).
\]
This is exactly the primitive normal to the displayed kernel, so
$\mathbb Zv'+\mathbb Zw'$ has lattice index one.  Hence $v,w$ form a
$\mathbb Z$-basis of $L_{A_h^\partial}$.  An arbitrary relation is
$sv+tw$.  If
$|s|\geq2$, then its second coordinate has absolute value greater than
$h$.  If $|s|\leq1$ and $|t|\geq2$, then the third coordinate has absolute
value at least $2h-1$.  Among $(s,t)\in\{-1,0,1\}^2$, the combinations
with both entries nonzero have degree $h+1$ or $2h-1$.  Hence the only
nonzero active relations are $\pm v,\pm w$.

There is no room for monomial padding, and the two degree-$h$ binomials
\[
 x_0^{h-1}x_2-x_1^h,
 \qquad
 x_0^{h-1}x_3-x_2^h
\]
have four distinct monomials.  Thus $\dim_K(I_{A_h^\partial})_h=2$ and
$|hA_h^\partial|=M_h-2$.

The nonzero deficits in the tail list are
$M_{h-r}\in\{1,4,10,\ldots\}$, so the deficit $2$ is absent.
\end{proof}

\subsection{Exact defect-two diameter and a lattice obstruction}
\label{subsec:defect-two-diameter}

Proposition~\ref{prop:boundary} shows that defect two already occurs at the
quadratic cutoff.  The least diameter at which it occurs is substantially
smaller, and can be determined exactly.  Senger proved that the values
$M_h-2$ and $M_h-3$ are rare in the long-interval frequency problem
\cite[Theorem~2]{Senger2025}; the result below gives a complementary exact
extremal statement for the first of them.  For $h\geq2$, define
\[
 \eta_2(h)=\min\{\ndiam(A):A\subset\Z,\ |A|=4,
                    \ |hA|=M_h-2\}.
\]

\begin{theorem}[Exact diameter at defect two]
\label{thm:defecttwodiameter}
For every $h\geq2$,
\begin{equation}\label{eq:defecttwodiameter}
 \boxed{\qquad \eta_2(h)=1+\binom{h+1}{2}.\qquad}
\end{equation}
More explicitly, if
\[
 D_h=1+\binom{h+1}{2},
\]
then equality is attained by
\[
 \{0,1,D_h-h,D_h\}\qquad(h\ \textup{even})
\]
and by
\[
 \left\{0,\frac{h-1}{2},D_h-\frac{h+1}{2},D_h\right\}
 \qquad(h\ \textup{odd}).
\]
\end{theorem}

\begin{proof}
Normalize a putative witness as
\[
 A=\{0,a,b,D\},\qquad 0<a<b<D,\qquad \gcd(a,b,D)=1.
\]
By \eqref{eq:Hilbert},
\[
 \dim_K(I_A)_h=M_h-|hA|=2.
\]
If $I_A$ contained a nonzero binomial of degree $r<h$, multiplication by
that binomial would inject $S_{h-r}$ into $(I_A)_h$.  This would give
\[
 \dim_K(I_A)_h\geq \dim_K S_{h-r}=M_{h-r}\geq M_1=4,
\]
a contradiction.  Consequently every nonzero homogeneous relation of $A$
has degree at least $h$.

Write the three gaps as
\[
 g_1=a,\qquad g_2=b-a,\qquad g_3=D-b,
\]
and put $C=\binom{h+1}{2}$.  The initial-degree conclusion and
Lemma~\ref{lem:fourpoint-relation-ball}, applied with $H=h$, give
\[
 D\geq C.
\]
If equality held, the gaps would satisfy
\eqref{eq:fourpointcriticalgaps}.  We now exclude that equality.
For $c=(c_1,c_2,c_3)\in\Z^3$ satisfying
$c_1g_1+c_2g_2+c_3g_3=0$, the associated homogeneous relation is
\[
 z(c)=(-c_1,c_1-c_2,c_2-c_3,c_3),
\]
of degree
\begin{equation}\label{eq:defecttwotaildegree}
 \delta(c)=\rho(z(c))=
 \frac{|c_1|+|c_1-c_2|+|c_2-c_3|+|c_3|}{2}.
\end{equation}

Suppose first that $h=2m$.  Reflecting $A$ if necessary, write
$g_1=r\leq m$.  By \eqref{eq:fourpointcriticalgaps},
$\gcd(r,2m)=1$, so $r=2j+1$ is odd.  The vector
\[
 c=(-j,1,-m-j)
\]
is orthogonal to
\[
 (g_1,g_2,g_3)=(r,m(2m-1),2m-r)
\]
and satisfies $\delta(c)=m+r$.  If $r<m$, this is below $h$, a
contradiction.  Hence $r=m$; but then $\gcd(r,h)=m$, so $m=1$.
For $h=2$, the only critical set is $\{0,1,2,3\}$, whose ten degree-two
monomials have only seven distinct images.  Its defect is three, not two.

Suppose now that $h=2m+1$.  After reflection take $g_1=r\leq m$.
The three tail vectors
\[
 \begin{aligned}
 c^{(0)}&=(h-r,0,-r),\\
 c^{(1)}&=(-m,1,-m),\\
 c^{(2)}&=(m+1-r,1,-m-r)=c^{(0)}+c^{(1)}
 \end{aligned}
\]
are orthogonal to $(r,mh,h-r)$, and direct substitution into
\eqref{eq:defecttwotaildegree} gives
\[
 \delta(c^{(0)})=\delta(c^{(1)})=\delta(c^{(2)})=h.
\]
Their degree-$h$ binomials are
\[
 \begin{aligned}
 B_0&=x_1^{h-r}x_2^r-x_0^{h-r}x_3^r,\\
 B_1&=x_0^m x_2^{m+1}-x_1^{m+1}x_3^m,\\
 B_2&=x_1^{m-r}x_2^{m+r+1}
      -x_0^{m+1-r}x_3^{m+r}.
 \end{aligned}
\]
For $1\leq r\leq m$, the six displayed monomials are pairwise distinct.
Thus $B_0,B_1,B_2$ are linearly independent in $(I_A)_h$, contradicting
$\dim_K(I_A)_h=2$.  We have proved the strict lower bound
\begin{equation}\label{eq:defecttwolower}
 D\geq C+1=D_h.
\end{equation}

It remains to verify the asserted witnesses.  If $h=2m$, consider
\[
 A_h=\{0,1,D_h-h,D_h\}
\]
and the two relations
\[
 z=(m,-m-1,m,1-m),\qquad
 w=(m,1-m,-m-1,m).
\]
Both have degree $h$, and the determinant of their middle-coordinate
projections is
\[
 \left|\det
 \begin{pmatrix}
  -m-1&m\\
  1-m&-m-1
 \end{pmatrix}\right|=2m^2+m+1=D_h.
\]
The projected relation lattice has determinant $D_h$ by the same index
calculation as in \eqref{eq:fourpointdeterminant}; hence they form a basis of
the full relation lattice.  For $R=uz+vw$,
\[
 R_0+R_2=hu-v,\qquad R_0+R_3=u+hv.
\]
For any zero-sum vector, the absolute value of every coordinate subsum is
at most its degree.  Hence $\rho(R)\leq h$ implies
\[
 |hu-v|\leq h,\qquad |u+hv|\leq h.
\]
Squaring and adding gives
\[
 (h^2+1)(u^2+v^2)\leq2h^2<2(h^2+1).
\]
Thus $(u,v)$ is zero or a signed coordinate unit vector.

If $h=2m+1$, take
\[
 A_h=\{0,m,D_h-(m+1),D_h\}
\]
and the two relations
\[
 z=(m-1,-m,m+2,-m-1),\qquad
 w=(m+1,-m-1,-m,m).
\]
Since $D_h-(m+1)\equiv1\pmod m$, the set $A_h$ is primitive.  Consequently
the middle-coordinate projection of its full relation lattice has determinant
$D_h$, by the index calculation in \eqref{eq:fourpointdeterminant}.
Again both have degree $h$, while
\[
 \left|\det
 \begin{pmatrix}
  -m&m+2\\
  -m-1&-m
 \end{pmatrix}\right|=2m^2+3m+2=D_h.
\]
Their determinant therefore equals the determinant of the projected lattice,
so they form a basis of the full relation lattice.  For $R=uz+vw$,
\[
 R_0+R_2=hu+v,\qquad R_0+R_3=-2u+hv.
\]
If $\rho(R)\leq h$, put $U=hu+v$ and $V=-2u+hv$.
Then $|U|,|V|\leq h$ and
\[
 u=\frac{hU-V}{h^2+2},\qquad
 v=\frac{2U+hV}{h^2+2}.
\]
Since $h\geq3$, these inequalities give $|u|,|v|<2$.  Each of the four
diagonal choices $(u,v)\in\{-1,1\}^2$ violates either $|U|\leq h$ or
$|V|\leq h$.  Thus once more only zero and the signed coordinate unit
vectors remain.

In either parity, the only nonzero relations of degree at most $h$ are
$\pm z$ and $\pm w$, both of degree exactly $h$.  A primitive relation of
degree $h$ has a unique degree-$h$ binomial, with no room for monomial
padding, and the two binomials here have four distinct monomials.  Therefore
$\dim_K(I_{A_h})_h=2$, so $|hA_h|=M_h-2$.  Together with
\eqref{eq:defecttwolower}, this proves the theorem.
\end{proof}

We now pass from one-dimensional normalized diameter to Nathanson's lattice
compression problem.  For a nonempty finite $A\subset\mathbb R^n$, write
\[
 \operatorname{diam}_\infty(A)
 =\max_{x,y\in A}\lVert x-y\rVert_\infty.
\]
To keep the two interval conventions separate, for integers $h,n\geq1$ put
\[
 \nu_n(h,4)=\min\left\{D\in\Z_{\geq0}:
 \substack{\text{for every }t\in\Rcal(h,4)\text{ there exists}\\
 A\subset[0,D]^n\cap\Z^n\text{ with }|A|=4\text{ and }|hA|=t}
 \right\}.
\]
Thus $\nu_1(h,4)=\nu(h,4)$.  Problem~3 of
\cite{NathansonLattice2026} denotes these parameters by $N_n(h,4)$ and
$N(h,4)=N_1(h,4)$, respectively.  This $[0,D]$ convention differs from the
interval-size convention in \cite{NathansonCompression}, for which the
parameter is $\nu(h,4)+1$.

\begin{corollary}[Exact lattice obstruction and failure of root compression]
\label{cor:defecttwolattice}
For every $h\geq2$ and $n\geq1$, the least $D$ for which some
$A\subset[0,D]^n\cap\Z^n$ with $|A|=4$ satisfies $|hA|=M_h-2$ is
\begin{equation}\label{eq:defecttwolatticediameter}
 D_h=1+\binom{h+1}{2}.
\end{equation}
Consequently
\begin{equation}\label{eq:latticecompressionbounds}
 1+\binom{h+1}{2}
 \leq\nu_n(h,4)\leq\nu_1(h,4)
 \leq\max\{h^2,D_h^{\max}\}.
\end{equation}
In particular, $\nu_n(h,4)=\Theta(h^2)$ with constants uniform in $n$.
Moreover, for every $n\geq2$,
\begin{equation}\label{eq:rootcompressionfails}
 \nu_n(h,4)>\nu_1(h,4)^{1/n}.
\end{equation}
Thus the inequality proposed in \cite[Problem~3]{NathansonLattice2026}
fails already for $k=4$, for every $h\geq2$ and every $n\geq2$.
\end{corollary}

\begin{proof}
Let $A=\{a_0,a_1,a_2,a_3\}\subset\Z^n$ satisfy $|hA|=M_h-2$.
We first show that $A$ is collinear.  If its affine dimension is three,
its homogeneous relation lattice is zero, so all $M_h$ degree-$h$
compositions have distinct images.  If its affine dimension is two, its
homogeneous relation lattice is a saturated rank-one lattice.  Let $z$ be
its primitive generator and put $r=\rho(z)$.  The four points are distinct,
so $r\geq2$.  Exactly as in the proof of
Proposition~\ref{prop:rankone}, the homogeneous toric ideal is generated by
\[
 x^{z^+}-x^{z^-}.
\]
Its degree-$h$ defect is $M_{h-r}$ if $r\leq h$, and zero if $r>h$.
The possible defects are therefore $0,1,4,10,\ldots$, never two.  Affine
dimensions three and two are impossible, while four distinct points cannot
have affine dimension zero.  Hence $A$ has affine dimension one.

Write
\[
 A=a_0+\{t_0,t_1,t_2,t_3\}v,
\]
where $v\in\Z^n$ is primitive and the $t_i$ are distinct integers.  Put
\[
 q=\gcd\{t_i-t_j:0\leq i,j\leq3\},\qquad
 D_0=\frac{\max_i t_i-\min_i t_i}{q}.
\]
Then $D_0$ is the normalized diameter of the corresponding
one-dimensional four-set, whereas
\begin{equation}\label{eq:collinearnormalization}
 \operatorname{diam}_\infty(A)
 =qD_0\lVert v\rVert_\infty\geq D_0.
\end{equation}
Theorem~\ref{thm:defecttwodiameter} gives $D_0\geq D_h$.
Conversely, its one-dimensional witnesses embed along a coordinate axis.
This proves \eqref{eq:defecttwolatticediameter} and the first inequality in
\eqref{eq:latticecompressionbounds}.  The remaining inequalities follow
from axis-embedding all one-dimensional witnesses and
Corollary~\ref{cor:compression}.

Finally, if $h\geq4$, then
\[
 \nu_1(h,4)^{1/n}\leq\sqrt{\nu_1(h,4)}
 \leq h<D_h\leq\nu_n(h,4).
\]
For $h=2,3$, Corollary~\ref{cor:compression} gives
$\nu_1(2,4)=6$ and $\nu_1(3,4)=11$, while $D_2=4$ and $D_3=7$;
the same strict inequality follows directly.
This proves \eqref{eq:rootcompressionfails}.
\end{proof}

Thus equality of the integer and lattice spectra does not entail a
root-scale reduction in the diameter needed to realize them.

\section{Popularity, collision multiplicities, and conditional laws}
\label{sec:popularity}

We now turn to the probabilistic counterpart of the same rank geometry,
beginning with four-point realization frequencies.  The labels in the
large-diameter classification are exactly the asymptotic frequency spikes.
This gives an exact form of the triangular-gap phenomenon observed by
Nathanson and quantified by Senger
\cite{NathansonTriangular,Senger2025}.
Enumeration of $B_h$-sets of prescribed cardinality has a parallel
probabilistic literature; see, for example, Dellamonica, Kohayakawa, Lee,
R\"odl, and Samotij \cite{DellamonicaEtAl2018}.  Here the statistic is finer:
we condition on the exact label $|hA|$ and identify its frequency exponent.

For $n\in\Rcal(h,4)$ and $q\geq1$, put
\[
 F_{h,n}(q)=\#\left\{A\in\binom{[q]}4:|hA|=n\right\},
 \qquad [q]=\{1,\ldots,q\}.
\]
For fixed $h$, we call $n$ \emph{popular} when
$F_{h,n}(q)/q^2\to\infty$ as $q\to\infty$.

\begin{theorem}[Four-point active-rank profile]
\label{thm:spectral-rank}
For every $h\geq2$,
\begin{align}
 \Rcal^{[0]}(h,4)&=\{M_h\},\label{eq:rankzero-spectrum}\\
 \Rcal^{[1]}(h,4)&=
 \{M_h-M_{h-j}:2\leq j\leq h\},\label{eq:rankone-spectrum}\\
 \Rcal^{[2]}(h,4)&=
 \Rcal(h,4)\setminus
 \bigl(\Rcal^{[0]}(h,4)\cup\Rcal^{[1]}(h,4)\bigr).
 \label{eq:ranktwo-spectrum}
\end{align}
The first two layers have cardinalities $1$ and $h-1$ exactly.  As
$h\to\infty$, the third has cardinality
$|\Rcal(h,4)|-h=\Theta(h^3)$; in particular, all but a
$\Theta(h^{-2})$ proportion of the distinct values have minimum rank two.
Moreover, as $h\to\infty$,
\[
 \left|\Rcal^{[2]}(h,4)\cap
 \left[M_h-\frac{h^3}{250},M_h-\frac{h^3}{400}\right]\right|
 \gg h^3.
\]
For every rank-two cardinality, the realization count is an exact
quasipolynomial in $q-1$ of degree two and the normalized shape locus is
zero-dimensional.  Such a cardinality has no realization of normalized
diameter exceeding $h^2$; conversely, the rank-zero and rank-one labels are
exactly those admitting arbitrarily large diameter.  The largest rank-two
label is $M_h-2$, whose least possible normalized diameter is
$1+\binom{h+1}{2}$.
\end{theorem}

\begin{theorem}[Exact popularity stratification]\label{thm:popularity}
Fix $h\geq2$.  There are explicit positive rational constants
$\kappa_2,\ldots,\kappa_h$ such that, as $q\to\infty$,
\begin{equation}\label{eq:tailfrequency}
 F_{h,M_h-M_{h-r}}(q)=\kappa_rq^3+O_h(q^2)
 \qquad(2\leq r\leq h),
\end{equation}
and
\begin{equation}\label{eq:maxfrequency}
 F_{h,M_h}(q)=\binom q4-
 \left(\sum_{r=2}^h\kappa_r\right)q^3+O_h(q^2).
\end{equation}
For every
\[
 n\in\Rcal(h,4)\setminus
 \left(\{M_h\}\cup\{M_h-M_{h-r}:2\leq r\leq h\}\right)
\]
there is an explicit positive rational constant $\lambda_{h,n}$ such that
\begin{equation}\label{eq:corefrequency}
 F_{h,n}(q)=\lambda_{h,n}q^2+O_h(q).
\end{equation}
Thus exactly $h$ cardinalities are popular: the maximum has frequency of
order $q^4$, the other $h-1$ displayed values have frequency of order
$q^3$, and every remaining attainable value has frequency of order $q^2$.
\end{theorem}

\begin{proof}[Proof of Theorem~\ref{thm:spectral-rank}]
For a four-element set, the full homogeneous relation lattice has rank two,
so the active rank is one of $0,1,2$.  Proposition~\ref{prop:rankone} shows
that an active-rank-zero realizer has cardinality $M_h$, and a $B_h$ witness
gives the converse.  The same proposition shows that every active-rank-one
realizer has cardinality $M_h-M_{h-j}$ for some $2\leq j\leq h$.
Conversely, the witnesses in \eqref{eq:smallwitness} have active rank one and
realize all these values.  They are distinct because
$M_0<M_1<\cdots<M_{h-2}$.  This proves
\eqref{eq:rankzero-spectrum} and \eqref{eq:rankone-spectrum}; since no other
active rank is possible, \eqref{eq:ranktwo-spectrum} follows.
The first two layers have exactly $h$ values in total, so
\[
 |\Rcal^{[2]}(h,4)|=|\Rcal(h,4)|-h.
\]
Together with Theorem~\ref{thm:main}, this also gives the asserted
$1-\Theta(h^{-2})$ proportion.

The local window in \eqref{eq:localmain} contains $\gg h^3$ spectrum values.
Deleting the $h$ rank-zero and rank-one values removes only $O(h)$ of them,
so $\gg h^3$ values in that same window have minimum active rank two.  The
proof of Theorem~\ref{thm:main} supplies primitive maximal-rank witnesses for
this window; their ACM property follows from Theorem~\ref{thm:family}.  The
ambient bound gives the
matching $O(h^3)$ estimate.  Finally,
Theorem~\ref{thm:rank-frequency} says that each corresponding realization
count is an exact quasipolynomial in $q-1$ of degree $4-2=2$, with positive
leading coefficient.  Corollary~\ref{cor:hilbert-strata} gives normalized
shape dimension $4-2-2=0$.  Equivalently, its leading asymptotic has the form
\eqref{eq:corefrequency}.

Theorem~\ref{thm:tail} supplies arbitrarily large normalized-diameter
realizations for precisely the rank-zero and rank-one list, and says that no
other value can occur beyond diameter $h^2$.  Hence every realization of a
minimum-rank-two value lies in the quadratic core.  Finally, the deficit-one
value $M_h-1$ is the last member of the rank-one list, so $M_h-2$ is the
largest rank-two value; Theorem~\ref{thm:defecttwodiameter} gives its exact
least diameter $1+\binom{h+1}{2}$.
\end{proof}

We first give finite formulas for the constants.  For
$x=(x_1,x_2,x_3)\in\mathbb R^3$, define
\[
 z(x)=(-x_1,x_1-x_2,x_2-x_3,x_3)
\]
and
\begin{equation}\label{eq:gapdegree}
 \delta(x)=\rho(z(x))=
 \frac{|x_1|+|x_1-x_2|+|x_2-x_3|+|x_3|}{2}.
\end{equation}
Let $\mathfrak R_r$ be the set, modulo $c\sim-c$, of primitive integer
triples satisfying
\[
 \delta(c)=r,\qquad \min_i c_i<0<\max_i c_i.
\]
For every nonzero mixed-sign $x\in\mathbb R^3$, choose the sign and coordinate
order so that $x=(A,B,-D)$, where $A,B\geq0$ and $D>0$, and put
\begin{equation}\label{eq:relationweight}
 w(x)=\frac{D}{6(D+A)(D+B)}.
\end{equation}
This value is independent of the choices: the expression is symmetric in
$A,B$, and when one coordinate is zero either sign gives
$1/(6(A+D))$.  Permuting the three gap coordinates is a lattice-preserving
symmetry of the gap simplex (and leaves the translation weight unchanged),
so the coordinate reordering preserves the hyperplane-section coefficient.
Thus $w$ is homogeneous of degree $-1$ and continuous on each of the finitely
many sign/permutation chambers; their walls have measure zero.

For $r\geq2$, define the constants in Theorem~\ref{thm:popularity} by
\begin{equation}\label{eq:kappaformula}
 \boxed{\ \kappa_r:=\sum_{c\in\mathfrak R_r}w(c).\ }
\end{equation}
Lemma~\ref{lem:onehyperplane} and the proof of
Theorem~\ref{thm:popularity} below identify these constants with the
coefficients in \eqref{eq:tailfrequency}.
The sum is finite because the total variation of
$0,c_1,c_2,c_3,0$ is $2r$, so $|c_i|\leq r$.  It is positive because the
class of $c=(-(r-1),0,1)$ belongs to $\mathfrak R_r$ and contributes
$1/(6r)$.  Direct enumeration in $\mathfrak R_r$ gives, for example,
\[
 \kappa_2=\frac13,\qquad
 \kappa_3=\frac{43}{72},\qquad
 \kappa_4=\frac{73}{108}.
\]

\begin{proposition}[Size of the popularity constants]\label{prop:kappasize}
There are absolute constants $c_0,C_0>0$ such that
\[
 c_0r\leq\kappa_r\leq C_0r\qquad(r\geq2).
\]
Consequently
\[
 K_h:=\sum_{r=2}^h\kappa_r\asymp h^2.
\]
More precisely, using the real-variable functions $\delta$ and $w$ just
defined, put
\[
 \Omega=\{x\in\mathbb R^3:\delta(x)\leq1,
 \ \min_i x_i<0<\max_i x_i\},
 \qquad I=\int_\Omega w(x)\,dx.
\]
Then $0<I<\infty$ and
\begin{equation}\label{eq:Khasymptotic}
 \boxed{\ K_h\sim\frac{I}{2\zeta(3)}h^2.\ }
\end{equation}
Writing $\mathfrak K:=I/(2\zeta(3))>0$, one has
$K_h\sim\mathfrak K h^2$.
\end{proposition}

\begin{proof}
For each of the finitely many support patterns of $z(c)$, its positive
entries form a composition of $r$ and the absolute values of its negative
entries form another composition of $r$.  Since the total support has at
most four coordinates, this gives $O(r^2)$ primitive classes.  Moreover,
$\max(A,B,D)\geq r/4$, because the four-step path
$0,c_1,c_2,c_3,0$ has total variation $2r$.  Formula
\eqref{eq:relationweight} then gives $w(c)=O(1/r)$, and hence
$\kappa_r=O(r)$.

For the reverse bound, consider
\[
 c=(-a,b,s),\qquad a+s=r,\qquad
 \frac r3\leq a\leq\frac r2,\qquad
 \frac s2\leq b\leq s.
\]
Then
\[
 2\delta(c)=a+(a+b)+(s-b)+s=2r
\]
and
\[
 w(c)=\frac{a}{6(a+b)(a+s)}\geq\frac1{18r}.
\]
Primitivity is equivalent to
$\gcd(a,b,s)=\gcd(a,b,r)=1$.  The planar region of admissible pairs
$(a,b)$ has area
\[
 \int_{r/3}^{r/2}\frac{r-a}{2}\,da=\frac7{144}r^2.
\]
M\"obius inversion over the divisors of $r$ therefore gives
\[
 \#\{(a,b):\gcd(a,b,r)=1\}
 =\frac7{144}r^2\prod_{p\mid r}(1-p^{-2})
 +O(r\log\log r+d(r))\gg r^2,
\]
where $d(r)$ denotes the number of positive divisors of $r$.
 The resulting classes are distinct modulo sign.  Summing their weights
 proves $\kappa_r\gg r$ for all sufficiently large $r$.  After decreasing
 $c_0$ if necessary, the finitely many remaining values are covered by the
 positivity of $\kappa_r$ established above.

For the summatory assertion, homogeneity gives
$w(tx)=t^{-1}w(x)$, while the preceding estimate gives
$w(x)\ll1/\delta(x)$.  Thus $w$ is locally integrable in three
dimensions.  Let
\[
 S(X)=\sum_{\substack{c\in\mathbb Z^3\setminus\{0\}\\
                       c\ {\rm mixed},\ \delta(c)\leq X}}w(c).
\]
Weighted Riemann sums give
\begin{equation}\label{eq:allrelationRiemann}
 S(X)=IX^2+o(X^2).
\end{equation}
Indeed, first restrict to $\varepsilon\leq\delta(x)\leq1$, where the
integrand is bounded and Riemann integrable.  The continuous and discrete
contributions from $\delta\leq\varepsilon$ are $O(\varepsilon^2)$ after
division by $X^2$: by \eqref{eq:gapdegree}, the $\delta$-unit ball is a
fixed rational three-dimensional polytope, so subtracting consecutive
rational Ehrhart counts shows that a degree shell contains $O(m^2)$ lattice
points; moreover, $w=O(1/m)$ on that shell.
For fixed $\varepsilon$, let $X\to\infty$ in the ordinary Riemann sum on
the annulus; the two uniform $O(\varepsilon^2)$ bounds then permit
$\varepsilon\downarrow0$ and prove \eqref{eq:allrelationRiemann}.

M\"obius inversion, together with
$\delta(dc)=d\delta(c)$ and $w(dc)=d^{-1}w(c)$, now yields
\[
 P(h):=\sum_{\substack{c\ {\rm primitive, mixed}\\\delta(c)\leq h}}w(c)
 =\sum_{d\leq h}\frac{\mu(d)}d S(h/d).
\]
The same shell estimate used above gives $S(X)\ll X^2$ for $X\geq1$.
Consequently
\[
 \frac{P(h)}{h^2}
 =\sum_{d\leq h}\frac{\mu(d)}{d^3}
   \frac{S(h/d)}{(h/d)^2}.
\]
For each fixed $d$, the final factor tends to $I$ by
\eqref{eq:allrelationRiemann}, and it is uniformly bounded by the shell
estimate.  The summands are therefore dominated in absolute value by
$C/d^3$.  Dominated convergence and
$\sum_{d\geq1}\mu(d)d^{-3}=1/\zeta(3)$ give
\[
 P(h)=\frac{I}{\zeta(3)}h^2+o(h^2).
\]
The classes in \eqref{eq:kappaformula} identify $c$ with $-c$, so division
by two proves \eqref{eq:Khasymptotic}.
\end{proof}

For an integer $\ell\geq1$, if $\mathcal B^*_{\ell,4}(q)$ denotes the
four-subsets of $[q]$ that are
$B_\ell$-sets but not $B_{\ell+1}$-sets, the same constants give
\begin{equation}\label{eq:firstcollisionfrequency}
 |\mathcal B^*_{\ell,4}(q)|
 =\kappa_{\ell+1}q^3+O_\ell(q^2),
 \qquad \kappa_{\ell+1}\asymp\ell.
\end{equation}
Here and below in this paragraph, the asymptotic is as $q\to\infty$ with
$\ell$ fixed.
Indeed, away from pairwise intersections of relation hyperplanes,
membership in $\mathcal B^*_{\ell,4}(q)$ is exactly the assertion that the
primitive relation degree is $\ell+1$.

\begin{lemma}[One relation hyperplane]\label{lem:onehyperplane}
Fix an integer $r\geq2$ and let $c\in\mathfrak R_r$.  The number of sets
$A=\{a_0<a_1<a_2<a_3\}\subset[q]$ satisfying
$z(c)\cdot(a_0,a_1,a_2,a_3)=0$ is
\[
 w(c)q^3+O_c(q^2).
\]
The asymptotic is as $q\to\infty$.
\end{lemma}

\begin{proof}
Put
\[
 g_1=a_1-a_0,\qquad g_2=a_2-a_1,\qquad g_3=a_3-a_2.
\]
For every homogeneous vector $z$, one has
\[
 z\cdot(a_0,a_1,a_2,a_3)
 =-z_0g_1-(z_0+z_1)g_2+z_3g_3.
\]
The correspondence
\[
 c=(-z_0,-z_0-z_1,z_3),\qquad
 z=(-c_1,c_1-c_2,c_2-c_3,c_3)
\]
is unimodular.  Thus the relation becomes
\begin{equation}\label{eq:gaprelation}
 c_1g_1+c_2g_2+c_3g_3=0.
\end{equation}

Choose signs and gap coordinates so that $c=(A,B,-D)$.
Equation~\eqref{eq:gaprelation} is
\[
 Au+Bv=Ds.
\]
Since $c$ is primitive, the lattice
\[
 \Lambda=\{(u,v)\in\mathbb Z^2:Au+Bv\equiv0\pmod D\}
\]
has index $D$.  Once $(u,v)$ is chosen, the third gap is
$s=(Au+Bv)/D$, and there are $q-u-v-s$ possible translations into
$[q]$.  The gaps $u,v,s$ are strictly positive.  Passing to the closed
nonnegative section in the integral below changes only boundary terms, whose
degree-one weighted lattice count is $O_c(q^2)$.  Weighted Ehrhart counting
with this degree-one weight on the fixed index-$D$ lattice gives the $q^3$
main term below, with $O_c(q^2)$ error.  Thus
\begin{align*}
 &\frac{q^3}{D}
 \int_{\substack{x,y\geq0\\
 ((D+A)/D)x+((D+B)/D)y\leq1}}
 \left(1-\frac{D+A}{D}x-\frac{D+B}{D}y\right)dx\,dy
 +O_c(q^2)\\
 &\hspace{35mm}=
 \frac{D}{6(D+A)(D+B)}q^3+O_c(q^2).
\end{align*}
\end{proof}

\begin{proof}[Proof of Theorem~\ref{thm:popularity}]
For fixed $h$, only finitely many primitive relation directions have degree
at most $h$.  Two distinct directions determine two independent rational
hyperplanes in $\mathbb R^4$; their intersection contains $O_h(q^2)$
integer points of $[q]^4$.  Lemma~\ref{lem:onehyperplane} and
inclusion--exclusion therefore show that the union of the primitive
degree-$r$ relation hyperplanes has cardinality
\begin{equation}\label{eq:hyperplaneunion}
 \kappa_rq^3+O_h(q^2).
\end{equation}

A four-set is a $B_h$-set exactly when it lies on none of the active
relation hyperplanes.  Outside their pairwise intersections, a non-$B_h$
set has active rank one.  If its primitive direction has degree $r$,
Proposition~\ref{prop:rankone} gives
\[
 |hA|=M_h-M_{h-r}.
\]
These values are distinct as $r$ varies.  Equation~
\eqref{eq:hyperplaneunion}, with all intersections absorbed in the error,
gives \eqref{eq:tailfrequency}; subtracting the union of all active
hyperplanes from $\binom q4$ gives \eqref{eq:maxfrequency}.

Now let $n$ be an attainable cardinality outside the displayed list.  Every
representative has at least two independent active relations, so
Lemma~\ref{lem:product} implies that its normalized diameter is at most
$h^2$.  Let $\mathcal S_{h,n}$ be the resulting finite set of primitive
normalized shapes
\[
 B=\{0<b<c<d\},\qquad \gcd(b,c,d)=1,\qquad |hB|=n.
\]
Every representative in $[q]$ occurs uniquely as $t+gB$, and the number of
copies of one shape is
\[
 \sum_{g=1}^{\lfloor(q-1)/d\rfloor}(q-gd)
 =\frac{q^2}{2d}+O_d(q).
\]
Consequently \eqref{eq:corefrequency} holds with
\begin{equation}\label{eq:lambdaformula}
 \boxed{\ \lambda_{h,n}=\frac12
 \sum_{B\in\mathcal S_{h,n}}\frac1{\max B}>0.\ }
\end{equation}
\end{proof}

\begin{corollary}[First-order law for a random four-set]
\label{cor:random-four-set-first-order}
Fix $h\geq2$.  As $q\to\infty$, a uniformly random
$A\in\binom{[q]}4$ satisfies
\[
 \mathbb P(|hA|=M_h-M_{h-r})
 =\frac{24\kappa_r}{q}+O_h(q^{-2})
 \qquad(2\leq r\leq h),
\]
and
\[
 \mathbb P(|hA|=M_h)
 =1-\frac{24K_h}{q}+O_h(q^{-2}).
\]
Every other attainable cardinality has probability
$\Theta_{h,n}(q^{-2})$.  In particular,
\[
 \mathbb P(A\text{ is not }B_h)
 =\frac{24K_h}{q}+O_h(q^{-2}),
 \qquad K_h\asymp h^2.
\]
\end{corollary}

\begin{proof}
Divide \eqref{eq:tailfrequency}, \eqref{eq:maxfrequency}, and
\eqref{eq:corefrequency} by
$\binom q4=q^4/24+O(q^3)$.  The final assertion follows because $A$ is a
$B_h$-set exactly when $|hA|=M_h$.
\end{proof}

The preceding law resolves \cite[Problem~9]{NathansonTriangular}
asymptotically.  Indeed, let
\[
 \mathcal N_h(q)=\left\{A\in\binom{[q]}4:A\text{ is not a }B_h\text{-set}\right\}
\]
and let $\mathcal P_h(q)$ consist of those $A\in\mathcal N_h(q)$ for which
\[
 |hA|\in\{M_h-M_{h-r}:2\leq r\leq h\}.
\]
Then Theorem~\ref{thm:popularity} gives
\[
 |\mathcal N_h(q)|=K_hq^3+O_h(q^2),\qquad
 |\mathcal P_h(q)|=K_hq^3+O_h(q^2),
\]
and consequently
\begin{equation}\label{eq:problem9}
 \boxed{\ \frac{|\mathcal P_h(q)|}{|\mathcal N_h(q)|}
       =1+O_h(q^{-1})\longrightarrow1.\ }
\end{equation}

Senger proves that the tetrahedral-difference values occur
$\Omega(h^{-5}q^3)$ times and that the intervening sizes occur
$O(h^{13}q^2)$ times, when $q$ is sufficiently large relative to $h$
\cite[Theorem~2]{Senger2025}.  Theorem~\ref{thm:popularity} is a sharp
asymptotic refinement of that result: it gives explicit leading constants,
the full $q^3$ correction to the maximum, and a positive quadratic
asymptotic for every remaining attainable cardinality, including those
below the last popular spike.

\subsection{The higher-cardinality popularity law}\label{subsec:popularity-higher-k}

The hyperplane argument is not special to four elements.  O'Bryant proposed a
closely related family of higher-cardinality candidate values
\cite{OBryant2025}, and Senger conditionally extrapolated its gap pattern
\cite[Section~4]{Senger2025}.  The theorem below identifies the feasible
popular list exactly, including the maximum and excluding the impossible
degree-one endpoint.  It also determines the frequency scales and
coefficients, giving a precise asymptotic answer to Nathanson's request to
discover the pattern for larger set sizes
\cite[Problem~10]{NathansonTriangular}.

For $k\geq3$ and $j\geq0$, put
\[
 M_{j,k}=\binom{j+k-1}{k-1},
 \qquad
 F^{(k)}_{h,n}(q)=
 \#\left\{A\in\binom{[q]}k:|hA|=n\right\}.
\]
For fixed $h,k$, we use the corresponding convention that $n$ is popular
when $F^{(k)}_{h,n}(q)/q^{k-2}\to\infty$.

\begin{theorem}[Higher-cardinality popularity law]
\label{thm:popularity-higher-k}
Fix $k\geq3$ and $h\geq2$.  There are explicit positive rational
constants $\kappa_{k,r}$, $2\leq r\leq h$, such that
as $q\to\infty$,
\begin{equation}\label{eq:pop-hk-tail}
 F^{(k)}_{h,M_{h,k}-M_{h-r,k}}(q)
 =\kappa_{k,r}q^{k-1}+O_{h,k}(q^{k-2}),
\end{equation}
and
\begin{equation}\label{eq:pop-hk-max}
 F^{(k)}_{h,M_{h,k}}(q)
 =\binom qk-
 \left(\sum_{r=2}^h\kappa_{k,r}\right)q^{k-1}
 +O_{h,k}(q^{k-2}).
\end{equation}
Every other attainable cardinality has frequency $O_{h,k}(q^{k-2})$.
Thus the $h$ values
\begin{equation}\label{eq:pop-hk-values}
 \{M_{h,k}\}\cup
 \{M_{h,k}-M_{h-r,k}:2\leq r\leq h\}
\end{equation}
are exactly the popular values.  In decreasing order, their consecutive
gaps are
\[
 \binom{j+k-2}{k-2}\qquad(0\leq j\leq h-2).
\]
\end{theorem}

To describe the constants, let $\mathcal Z_{k,r}$ be the finite set,
modulo sign, of primitive vectors $z\in\mathbb Z^k$ such that
\[
 \sum_{i=1}^kz_i=0,\qquad
 \rho(z):=\sum_{i=1}^kz_i^+=r,
\]
and whose hyperplane meets the open ordered chamber
$0<x_1<\cdots<x_k<1$.  Let $\omega_k(z)$ be the relative normalized
$(k-1)$-volume of that rational hyperplane section; equivalently, it is the
leading Ehrhart coefficient in
\begin{equation}\label{eq:pop-hk-ehrhart}
 \#\{1\leq a_1<\cdots<a_k\leq q:z\cdot a=0\}
 =\omega_k(z)q^{k-1}+O_z(q^{k-2}).
\end{equation}
Then
\begin{equation}\label{eq:pop-hk-kappa}
 \boxed{\ \kappa_{k,r}=
 \sum_{z\in\mathcal Z_{k,r}}\omega_k(z).\ }
\end{equation}
In particular, $\kappa_{k,r}$ is positive and rational.  Positivity already
follows from
\[
 z=(r-1,-r,1,0,\ldots,0),
\]
whose relation hyperplane meets the open ordered chamber.

\begin{proof}[Proof of Theorem~\ref{thm:popularity-higher-k}]
Equation~\eqref{eq:pop-hk-ehrhart} is the standard Ehrhart estimate for a
rational polytope of dimension $k-1$; replacing weak by strict order only
changes boundary terms of order $q^{k-2}$.  For fixed $h,k$, there are only
finitely many primitive relation directions of degree at most $h$; taking
the maximum of their Ehrhart error constants makes every such estimate
uniform in $O_{h,k}(q^{k-2})$.
Intersections of two distinct relation hyperplanes have codimension two in
$\mathbb R^k$, and hence contain $O_{h,k}(q^{k-2})$ ordered integer points.

Outside the union of these pairwise intersections, a non-$B_h$ set has a
unique primitive active direction $z$, of some degree $r\leq h$.
Strict ordering excludes $r=1$, since then
$z=\pm(e_i-e_j)$ would force $a_i=a_j$.  Thus $2\leq r\leq h$.
Write $z=z^+-z^-$.  Over any field $K$, put $S=K[x_1,\ldots,x_k]$ and
\[
 f_z=x^{z^+}-x^{z^-}\in
 S.
\]
Every degree-at-most-$h$ toric binomial has, after cancellation, exponent
difference $mz$ for some integer $m$; it is therefore divisible by $f_z$,
because $X^m-Y^m$ is divisible by $X-Y$.  Conversely, every multiple of
$f_z$ is a relation.  Thus the degree-$h$ relation space is
$f_zS_{h-r}$, of dimension $M_{h-r,k}$, and
\[
 |hA|=M_{h,k}-M_{h-r,k}.
\]
This is the $k$-variable version of the principal-ideal argument in
Proposition~\ref{prop:rankone}.  No active relation gives the maximum
$M_{h,k}$.

Summing \eqref{eq:pop-hk-ehrhart} over primitive directions and absorbing
all pairwise intersections proves \eqref{eq:pop-hk-tail} and
\eqref{eq:pop-hk-max}; any other cardinality is confined to those
intersections.  Finally, with $M_{-1,k}:=0$, Pascal's identity gives
\[
 M_{j,k}-M_{j-1,k}=\binom{j+k-2}{k-2},
\]
which proves the gap assertion.
\end{proof}

There is also a sharp aggregate law for the coefficients.  Put
$d=k-1$.  In gap coordinates define
\[
 c_j=-\sum_{i=1}^jz_i\qquad(1\leq j\leq d).
\]
The inverse map is unimodular and is given by
\[
 z=(-c_1,c_1-c_2,\ldots,c_{d-1}-c_d,c_d).
\]
Hence $z$ is primitive exactly when $c$ is primitive, and
\begin{equation}\label{eq:pop-hk-gapdegree}
 \delta_k(c)=\rho(z(c))=
 \frac12\left(
 |c_1|+\sum_{j=2}^{d}|c_j-c_{j-1}|+|c_d|
 \right).
\end{equation}
The relation hyperplane meets the ordered chamber exactly when $c$ has at
least one positive and one negative coordinate.

For a nonzero mixed real vector $c\in\mathbb R^d$, let $w_k(c)$ be the
Euclidean $(k-1)$-volume of
\[
 \left\{(t,g_1,\ldots,g_d)\in\mathbb R_{\geq0}^{k}:
 t+\sum_{j=1}^dg_j\leq1,\quad c\cdot g=0\right\}
\]
divided by $\|c\|_2$.  If $c$ is primitive integral, then the lattice in
the displayed hyperplane has covolume $\|c\|_2$, so
$w_k(c)=\omega_k(z(c))$.  For real $c$, this definition satisfies
$w_k(tc)=t^{-1}w_k(c)$ for $t>0$.

\begin{lemma}[Regularity of the section weight]
\label{lem:section-weight-regularity}
On the mixed cone, $w_k$ is Borel measurable and is continuous on every
strict sign chamber.  Its extension by zero to the nonmixed cone is therefore
continuous almost everywhere.  Moreover, for every $\varepsilon>0$, it is
bounded on $\{c:\varepsilon\leq\delta_k(c)\leq1\}$.
\end{lemma}

\begin{proof}
Fix a strict sign chamber and put
$P=\{i:c_i>0\}$ and $N=\{j:c_j<0\}$.  In the simplex with coordinates
$(t,g_1,\ldots,g_d)$, the section $c\cdot g=0$ is the convex hull of
$0$, the $t$-vertex $e_t$, and the points
\[
 v_{ij}(c)=
 \frac{|c_j|e_i+c_i e_j}{c_i+|c_j|}
 \qquad(i\in P,\ j\in N).
\]
Indeed, these are precisely the simplex vertices lying on the section and
the intersections with edges whose endpoints lie on opposite sides.  The
displayed vertices depend continuously on $c$, and the Euclidean
$(k-1)$-volume of their convex hull is continuous.  Since $\|c\|_2>0$ on
the chamber, division by $\|c\|_2$ preserves continuity.  The same
description, with the zero-coordinate vertices included, gives measurability
on the finitely many
lower-dimensional sign strata.  Their union, together with the boundary of
the mixed cone, lies in the coordinate hyperplanes and has measure zero.
Finally, every section is convex and lies in the fixed unit simplex, so the
isodiametric inequality bounds its $(k-1)$-volume uniformly.  Equivalence of
the norms $\delta_k$ and $\|\cdot\|_2$ then bounds $\|c\|_2^{-1}$ on the
stated annulus.
\end{proof}

\begin{proposition}[Aggregate higher-cardinality coefficient]
\label{prop:pop-hk-summatory}
For integers $m\geq0$ and $k\geq3$, write
\[
 K_{m,k}=\sum_{r=2}^{m}\kappa_{k,r},
\]
with the empty sum equal to zero when $m=0$ or $1$.
Let
\[
 I_k=\int_{\substack{c\in\mathbb R^{k-1}\ {\rm mixed}\\
                      \delta_k(c)\leq1}}w_k(c)\,dc.
\]
Then $0<I_k<\infty$ and, for fixed $k\geq3$,
\begin{equation}\label{eq:pop-hk-summatory}
 K_{h,k}=\sum_{r=2}^h\kappa_{k,r}
 \sim\frac{I_k}{2\zeta(k-1)}h^{k-2}.
\end{equation}
Here the asymptotic is as $h\to\infty$.  For $k=4$, these constants agree
with the preceding ones: $\kappa_{4,r}=\kappa_r$ and $I_4=I$.
In particular, $K_{h,k}\asymp_k h^{k-2}$.
\end{proposition}

\begin{proof}
The function $\delta_k$ is a norm on $\mathbb R^d$ and hence, by
finite-dimensional norm equivalence,
$\delta_k(c)\asymp_k\|c\|_2$.  A hyperplane section of the fixed unit
simplex has uniformly bounded volume, so
\[
 w_k(c)\ll_k\|c\|_2^{-1}\ll_k\delta_k(c)^{-1}.
\]
Since $d=k-1\geq2$, this proves integrability at the origin; positivity is
clear on any open mixed-sign cone.

Let
\[
 S_k(X)=\sum_{\substack{c\in\mathbb Z^d\setminus\{0\}\\
                         c\ {\rm mixed},\ \delta_k(c)\leq X}}w_k(c).
\]
Homogeneity and weighted Riemann sums give
\[
 S_k(X)=I_kX^{d-1}+o_k(X^{d-1}).
\]
For completeness, after removing $\delta_k(c)\leq\varepsilon X$, this is,
by Lemma~\ref{lem:section-weight-regularity}, an ordinary Riemann sum with a
bounded, almost-everywhere continuous integrand.  By
\eqref{eq:pop-hk-gapdegree}, the $\delta_k$-unit ball is a fixed rational
$d$-dimensional polytope; subtracting consecutive rational Ehrhart counts
shows that the shell $m-1<\delta_k(c)\leq m$ contains
$O_k(m^{d-1})$ lattice points.  On the removed region each such point has
weight $O_k(1/m)$, so the total contribution is
$O_k((\varepsilon X)^{d-1})$.  Homogeneity gives the same
$O_k(\varepsilon^{d-1})$ bound for the normalized continuous contribution.
Letting $X\to\infty$ first and then $\varepsilon\downarrow0$ proves the
displayed asymptotic.

Replacing $c$ by $-c$ leaves both the section hyperplane and its Euclidean
normal length unchanged, so $w_k(-c)=w_k(c)$.  Thus M\"obius inversion over
all primitive mixed vectors counts every class modulo sign twice.  Since
$w_k(mc)=m^{-1}w_k(c)$,
\[
 2K_{h,k}=\sum_{m\leq h}\frac{\mu(m)}mS_k(h/m).
\]
The bound $S_k(X)\ll_kX^{d-1}$ and dominated convergence, using
$\sum m^{-d}<\infty$, now give
\[
 2K_{h,k}\sim
 I_kh^{d-1}\sum_{m=1}^{\infty}\frac{\mu(m)}{m^d}
 =\frac{I_k}{\zeta(d)}h^{d-1}.
\]
Substituting $d=k-1$ proves \eqref{eq:pop-hk-summatory}.
\end{proof}

The endpoint $k=3$ is completely explicit.  Modulo sign, every mixed
vector has the form $c=(a,-b)$ with $a,b>0$; here
\[
 \delta_3(c)=a+b,
 \qquad w_3(c)=\frac1{2(a+b)}.
\]
Thus
\begin{equation}\label{eq:pop-k3-explicit}
 \kappa_{3,r}=\frac{\varphi(r)}{2r},
 \qquad
 K_{h,3}\sim\frac{h}{2\zeta(2)}.
\end{equation}
Moreover, three distinct integers have a one-dimensional homogeneous
relation lattice, so the ``other cardinality'' case in
Theorem~\ref{thm:popularity-higher-k} is empty when $k=3$.

\subsection{Exact collision multiplicities on the principal skeleton}
\label{subsec:rankone-multiplicity}

For an ordered integer set
$A=(a_1<\cdots<a_k)$ and $n\in\mathbb Z$, put
\[
 r_{A,h}(n)
 =
 \#\{u\in\mathcal U_{h,k}:u\cdot A=n\}.
\]
Thus $r_{A,h}$ is Nathanson's representation function, counting
nondecreasing $h$-term representations.  For $j\geq1$, write
\[
 H_j(A;h)=\#\{n\in hA:r_{A,h}(n)=j\}.
\]
Also put
\[
\begin{aligned}
  C_2(A;h)&=\#\{n\in hA:r_{A,h}(n)=2\},\\
  C_{\geq3}(A;h)&=\#\{n\in hA:r_{A,h}(n)\geq3\}.
 \end{aligned}
\]
\par\medskip
\begin{theorem}[Exact rank-one representation law]
\label{thm:rankone-multiplicity}
Let $h\geq2$, $k\geq3$, and let $A\subset\mathbb Z$ be an ordered
$k$-set with $r_h(A)=1$.  Let $z$, chosen up to sign, be the primitive
generator of the rank-one lattice
\[
 L_h(A)\cap\mathbb Z^k=\mathbb Zz,
\]
and put $r=\rho(z)$.  Then $2\leq r\leq h$, and, for every $j\geq1$,
\begin{equation}\label{eq:rankone-multiplicity}
 \boxed{\quad
 H_j(A;h)
 =
 M_{h-(j-1)r,k}
 -2M_{h-jr,k}
 +M_{h-(j+1)r,k}.
 \quad}
\end{equation}
Here and below $M_{s,k}=0$ when $s<0$.  In particular,
\begin{equation}\label{eq:rankone-max-multiplicity}
 \max_{n\in hA}r_{A,h}(n)
 =1+\left\lfloor\frac hr\right\rfloor
\end{equation}
and
\[
 |hA|=M_{h,k}-M_{h-r,k}.
\]
The collision counts satisfy
\begin{align}
 C_2(A;h)
 &=M_{h-r,k}-2M_{h-2r,k}+M_{h-3r,k},
 \label{eq:rankone-simple-collisions}\\
 C_{\geq3}(A;h)
 &=M_{h-2r,k}-M_{h-3r,k}.
 \label{eq:rankone-multiple-collisions}
\end{align}
\end{theorem}

\begin{proof}
Since $r_h(A)=1$, some nonzero active relation is an integral multiple
$mz$ with $\rho(mz)\leq h$.  Hence $r=\rho(z)\leq h$.  If $r=1$, then
$z=e_i-e_j$ up to sign for distinct $i,j$, forcing $a_i=a_j$, a
contradiction.  Thus $r\geq2$.

For $n\in hA$, let
\[
 F_n=\{u\in\mathcal U_{h,k}:u\cdot A=n\}.
\]
If $u,v\in F_n$, then $u-v$ is a homogeneous relation and
$\rho(u-v)\leq h$.  Hence $u-v\in L_h(A)\cap\mathbb Z^k=\mathbb Zz$.
Because $z$ is primitive and the nonnegative simplex is convex, each
$F_n$ is a consecutive lattice segment parallel to $z$.

For $m\geq0$, put
\[
 T_m
 =
 \#\{u\in\mathcal U_{h,k}:u+mz\in\mathcal U_{h,k}\}.
\]
Writing $z=z^+-z^-$, the map
\[
 w\longmapsto u=w+mz^-
\]
is a bijection from $\mathcal U_{h-mr,k}$ to the set counted by $T_m$
when $h-mr\geq0$; both sets are empty otherwise.  Consequently
\begin{equation}\label{eq:rankone-chain-pairs}
 T_m=M_{h-mr,k}.
\end{equation}
A fibre of cardinality $L$ contributes exactly $(L-m)^+$ pairs to
$T_m$.  Therefore
\[
 T_m=\sum_{n\in hA}\bigl(r_{A,h}(n)-m\bigr)^+.
\]
Taking the second finite difference at $m=j$ gives
\[
 H_j(A;h)=T_{j-1}-2T_j+T_{j+1},
\]
and \eqref{eq:rankone-multiplicity} follows from
\eqref{eq:rankone-chain-pairs}.  The largest $m$ for which $T_m>0$ is
$\lfloor h/r\rfloor$, proving \eqref{eq:rankone-max-multiplicity}.
Finally,
\[
 |hA|=T_0-T_1,
 \qquad
 C_{\geq3}(A;h)=T_2-T_3,
\]
while $C_2(A;h)=H_2(A;h)$, proving the remaining formulas.
\end{proof}

\subsection{A universal conditional limit law}
\label{subsec:conditional-limit}

The coefficient asymptotic from Proposition~\ref{prop:pop-hk-summatory}
has a distributional consequence.  For a non-$B_h$ set
$A=\{a_1<\cdots<a_k\}$, let
\[
 \tau(A)=\min\{\rho(z):0\ne z\in\mathbb Z^k,
                  \ \sum_i z_i=0,\ z\cdot(a_1,\ldots,a_k)=0\}
\]
be its first relation degree.

\begin{corollary}[Universal first-relation and sumset laws]
\label{cor:conditional-limit}
Fix $k\geq3$, and take the limits first as $q\to\infty$ and then as
$h\to\infty$.  If $A$ is uniform in $\binom{[q]}k$, conditional on $A$
not being a $B_h$-set, then
\begin{equation}\label{eq:first-relation-limit}
 \frac{\tau(A)}h\ \Longrightarrow\ X_k,
 \qquad \mathbb P(X_k\leq t)=t^{k-2}\quad(0\leq t\leq1).
\end{equation}
Equivalently, $X_k$ has the $\operatorname{Beta}(k-2,1)$ law.  Moreover,
\begin{equation}\label{eq:sumset-ratio-limit}
 \frac{|hA|}{M_{h,k}}\ \Longrightarrow\
 Y_k:=1-(1-X_k)^{k-1},
\end{equation}
whose distribution function is
\begin{equation}\label{eq:sumset-ratio-cdf}
 \mathbb P(Y_k\leq y)=
 \left[1-(1-y)^{1/(k-1)}\right]^{k-2}
 \qquad(0\leq y\leq1).
\end{equation}
Furthermore,
\begin{equation}\label{eq:sumset-ratio-mean}
 \mathbb E Y_k
 =1-\frac{(k-2)!(k-1)!}{(2k-3)!}.
\end{equation}
\end{corollary}

\begin{proof}
For fixed $h,k$, the one-hyperplane asymptotics and the fact that pairwise
intersections contribute only $O_{h,k}(q^{k-2})$ give
\[
 \lim_{q\to\infty}
 \mathbb P(\tau(A)=r\mid A\text{ is not }B_h)
 =\frac{\kappa_{k,r}}{K_{h,k}},
 \qquad K_{m,k}=\sum_{s=2}^m\kappa_{k,s}\quad(m\geq0),
\]
where the sum is empty, and therefore zero, for $m=0,1$.
The same estimate shows that, under this conditioning, the probability of
two independent active relations tends to zero.  Proposition~
\ref{prop:pop-hk-summatory} says
\[
 K_{h,k}\sim C_kh^{k-2},\qquad C_k=\frac{I_k}{2\zeta(k-1)}>0.
\]
Therefore, for $0\leq t\leq1$,
\[
 \lim_{h\to\infty}\lim_{q\to\infty}
 \mathbb P\left(\frac{\tau(A)}h\leq t
                  \mathrel{\Big|} A\text{ is not }B_h\right)
 =\lim_{h\to\infty}\frac{K_{\lfloor th\rfloor,k}}{K_{h,k}}
 =t^{k-2},
\]
with the endpoints following directly.  This proves
\eqref{eq:first-relation-limit}.

Outside the asymptotically negligible multiple-relation locus, the active
rank is one, and the principal-binomial formula gives
\[
 \frac{|hA|}{M_{h,k}}
 =1-\frac{M_{h-\tau(A),k}}{M_{h,k}}.
\]
Uniformly for integers $0\leq s\leq h$, the finite product formula for the binomial
coefficients gives
\[
 \frac{M_{h-s,k}}{M_{h,k}}
 =(1-s/h)^{k-1}+O_k(h^{-1}).
\]
Hence, if $\tau(A)/h\to x$, then $|hA|/M_{h,k}$ tends to
$1-(1-x)^{k-1}$.  The continuous mapping theorem proves
\eqref{eq:sumset-ratio-limit}.  Since the displayed map is increasing and has
inverse $y\mapsto1-(1-y)^{1/(k-1)}$,
\eqref{eq:sumset-ratio-cdf} follows.
Finally,
\[
 \mathbb E(1-X_k)^{k-1}
 =(k-2)\int_0^1t^{k-3}(1-t)^{k-1}\,dt
 =\frac{(k-2)!(k-1)!}{(2k-3)!},
\]
which proves \eqref{eq:sumset-ratio-mean}.
\end{proof}

\begin{corollary}[Conditional simple-collision phase law]
\label{cor:conditional-collision-phase}
Fix $k\geq3$.  There is a unique number
$\theta_k\in(0,1/3)$ satisfying
\begin{equation}\label{eq:collision-threshold}
 (1-\theta_k)^{k-1}
 -3(1-2\theta_k)^{k-1}
 +2(1-3\theta_k)^{k-1}=0.
\end{equation}
If $A_q$ is uniform in $\binom{[q]}k$, then, with the limits taken in
the displayed order,
\begin{equation}\label{eq:conditional-collision-phase}
 \boxed{\quad
 \lim_{h\to\infty}\lim_{q\to\infty}
 \mathbb P\left(
  C_2(A_q;h)>C_{\geq3}(A_q;h)
  \mathrel{\Big|} A_q\text{ is not }B_h
 \right)
 =1-\theta_k^{\,k-2}.
 \quad}
\end{equation}
For $k=4$,
\[
 \theta_4=\frac{21-\sqrt{69}}{62},
 \qquad
 1-\theta_4^2=0.9580848601\ldots.
\]
\end{corollary}

\begin{proof}
On the active-rank-one locus, with primitive relation degree $r$,
Theorem~\ref{thm:rankone-multiplicity} gives
\begin{equation}\label{eq:collision-difference}
 \begin{aligned}
 D_{h,k}(r)
 &:=C_2(A;h)-C_{\geq3}(A;h)\\
 &=M_{h-r,k}-3M_{h-2r,k}+2M_{h-3r,k}.
 \end{aligned}
\end{equation}
For fixed $h,k$, the multiple-relation locus has conditional probability
tending to zero as $q\to\infty$.  Hence the one-hyperplane asymptotics
give
\[
 \lim_{q\to\infty}
 \mathbb P\left(
  C_2(A_q;h)>C_{\geq3}(A_q;h)
  \mathrel{\Big|} A_q\text{ is not }B_h
 \right)
 =
 \frac{1}{K_{h,k}}
 \sum_{r=2}^h
 \kappa_{k,r}\,
 \mathbf 1_{\{D_{h,k}(r)>0\}}.
\]

Put $d=k-1$ and, for $0\leq x\leq1$, define
\[
 \Phi_k(x)
 =
 (1-x)^d
 -3\bigl((1-2x)^+\bigr)^d
 +2\bigl((1-3x)^+\bigr)^d.
\]
The elementary polynomial expansion of the binomial coefficient gives,
uniformly for the shifted integers $s$ occurring below,
\[
 \frac{d!}{h^d}M_{s,k}
 =\bigl((s/h)^+\bigr)^d+O_k(h^{-1}),
\]
where $M_{s,k}=0$ for $s<0$.  Consequently, uniformly for $2\leq r\leq h$,
\begin{equation}\label{eq:collision-difference-scaling}
 \frac{d!}{h^d}D_{h,k}(r)
 =\Phi_k(r/h)+O_k(h^{-1}).
\end{equation}

We verify the asserted sign change.  For $0<x<1/3$, set
\[
 t=\frac{1-3x}{1-2x}\in(0,1).
\]
Then
\[
 \Phi_k(x)
 =(1-2x)^d g_d(t),
 \qquad
 g_d(t)=(2-t)^d-3+2t^d.
\]
Now
\[
 g_d'(t)
 =d\bigl(2t^{d-1}-(2-t)^{d-1}\bigr)
\]
changes sign exactly once.  Moreover,
$g_d(0)=2^d-3>0$, $g_d(1)=0$, and $g_d(t)<0$ immediately to the
left of $1$.  Thus $g_d$, and hence $\Phi_k$, has exactly one
interior zero corresponding to $\theta_k$.  Also
\[
 \Phi_k(x)<0\quad(0<x<\theta_k),
 \qquad
 \Phi_k(x)>0\quad(\theta_k<x<1).
\]
For $1/3\leq x<1/2$, the latter positivity follows from
$(1-x)/(1-2x)\geq2$ and $2^d>3$; for $1/2\leq x<1$ it is immediate.

Let $X_h$ be the grid-valued random variable with
\[
 \mathbb P(X_h=r/h)=\frac{\kappa_{k,r}}{K_{h,k}}
 \qquad(2\leq r\leq h).
\]
The fixed-$h$ formula above is the probability that
$D_{h,k}(hX_h)>0$, and Corollary~\ref{cor:conditional-limit} says that
$X_h\Longrightarrow X_k$, where
$\mathbb P(X_k\leq x)=x^{k-2}$.  To justify the moving sign condition
explicitly, put
\[
 Z_k=\Phi_k^{-1}(0)=\{0,\theta_k,1\},
 \qquad U_\eta=\{x\in[0,1]:|\Phi_k(x)|\leq\eta\}.
\]
By the uniform estimate \eqref{eq:collision-difference-scaling}, for every
$\eta>0$ and all sufficiently large $h$, the signs of $D_{h,k}(r)$ and
$\Phi_k(r/h)$ agree whenever $r/h\notin U_\eta$.  Hence the difference
between the fixed-$h$ probability and $\mathbb P(X_h>\theta_k)$ is at most
$\mathbb P(X_h\in U_\eta)$.  Portmanteau gives
\[
 \limsup_{h\to\infty}\mathbb P(X_h\in U_\eta)
 \leq\mathbb P(X_k\in U_\eta).
\]
As $\eta\downarrow0$, the right side tends to
$\mathbb P(X_k\in Z_k)=0$.  Also
$\mathbb P(X_h>\theta_k)\to\mathbb P(X_k>\theta_k)$ because the limiting
law has no atom at $\theta_k$.  Taking first $q\to\infty$ and then
$h\to\infty$, the iterated limit in
\eqref{eq:conditional-collision-phase} is therefore
\[
 \mathbb P(X_k>\theta_k)
 =1-\theta_k^{k-2}.
\]
\enlargethispage{\baselineskip}
For $k=4$, expanding \eqref{eq:collision-threshold} gives
\[
 \theta(-3+21\theta-31\theta^2)=0,
\]
and the root in $(0,1/3)$ is the stated $\theta_4$.
\end{proof}

\begin{remark}[Nathanson's collision problem]
\label{rem:nathanson-problem11}
Problem~11 of \cite{NathansonTriangular} asks for
$C_{\geq3}(A;h)<C_2(A;h)$ for almost all $A$, without conditioning.
For fixed $h\geq2$ and $k\geq3$, however, almost every
$A_q\in\binom{[q]}k$ is a
$B_h$-set, and then both quantities are zero.  Indeed,
Theorem~\ref{thm:popularity-higher-k} gives
\[
 \mathbb P\bigl(C_2(A_q;h)>C_{\geq3}(A_q;h)\bigr)
 \leq
 \mathbb P(A_q\text{ is not }B_h)
 =O_{h,k}(q^{-1}).
\]
Thus the strict inequality in Problem~11 is false as stated.
Corollary~\ref{cor:conditional-collision-phase} gives the exact answer to its
natural version conditioned on the occurrence of a collision, in the
iterated-limit regime proved here.
\end{remark}

For $k=4$, the conditional limit has
\[
 \mathbb P(Y_4\leq y)=\left[1-(1-y)^{1/3}\right]^2,
 \qquad \mathbb E Y_4=\frac9{10}.
\]

\begin{remark}[The joint regime]
The order of limits is essential to the statement proved here.  The error
terms in the fixed-$(h,k)$ hyperplane counts are not presently uniform
enough to assert the same law in an arbitrary joint regime $h=h(q)$.  At
fixed $h,k$, the one-hyperplane terms make the expected number of primitive
active relation directions asymptotic to $k!K_{h,k}/q$, while
$K_{h,k}\sim C_kh^{k-2}$.  Thus the natural critical scale is
$q\asymp h^{k-2}$.  Determining the largest joint range in which the
conditional $\operatorname{Beta}(k-2,1)$ law persists, and the limiting point
process of primitive active directions modulo sign, marked by $\rho(z)/h$,
when $h^{k-2}/q\to\lambda\in(0,\infty)$, remains open.  A Poisson-type
description would require uniform control of the multiple-relation strata.
\end{remark}

\begin{corollary}[Random configurations versus random labels]
\label{cor:label-configuration-reversal}
For fixed $h\geq2$ and $k\geq3$, as $q\to\infty$ a uniformly random
$A\in\binom{[q]}k$ satisfies
\[
 \mathbb P(A\text{ is not }B_h)
 =\frac{k!K_{h,k}}q+O_{h,k}(q^{-2}),
 \qquad K_{h,k}\asymp_k h^{k-2}.
\]
For each $2\leq r\leq h$,
\[
 \mathbb P(|hA|=M_{h,k}-M_{h-r,k})
 =\frac{k!\kappa_{k,r}}q+O_{h,k}(q^{-2}).
\]
Moreover,
\begin{equation}\label{eq:random-configuration-principal}
 \begin{aligned}
  \mathbb P\bigl(r_{h,k}(|hA|)\leq1\bigr)
   &=1-O_{h,k}(q^{-2}),\\
  \mathbb P\bigl(r_h(A)=1\mid A\text{ is not }B_h\bigr)
   &=1-O_{h,k}(q^{-1}).
 \end{aligned}
\end{equation}
By contrast, if $k\geq4$ is fixed and $N_{h,k}$ is uniform in the finite set
$\Rcal(h,k)$, then, as $h\to\infty$,
\begin{equation}\label{eq:random-label-nonprincipal}
 \mathbb P\bigl(r_{h,k}(N_{h,k})\leq1\bigr)
 =\frac{h}{|\Rcal(h,k)|}
 =h^{-(k-2)+o_k(1)}\longrightarrow0.
\end{equation}
Thus principal labels are asymptotically negligible among attainable
cardinalities, although their realizations account for asymptotically almost
every fixed-order configuration in a growing interval.
Here $h$ is fixed before $q\to\infty$;
no joint $(h,q)$ assertion is made.
\end{corollary}

\begin{proof}
The first two displays are Theorem~\ref{thm:popularity-higher-k} divided by
$\binom qk=q^k/k!+O_k(q^{k-1})$.  That theorem also says that every
nonprincipal label has $O_{h,k}(q^{k-2})$ realizations.  There are only
finitely many labels for fixed $h,k$, so the first assertion in
\eqref{eq:random-configuration-principal} follows from
Corollary~\ref{cor:universal-principal-skeleton}.  Outside the union of two
distinct active-relation hyperplanes a non-$B_h$ tuple has actual active rank
one; those intersections contain $O_{h,k}(q^{k-2})$ points, whereas the
non-$B_h$ locus has order $q^{k-1}$.  This proves the conditional assertion.
Finally, Corollary~\ref{cor:universal-principal-skeleton} gives exactly $h$
labels of minimum rank at most one, and
Theorem~\ref{thm:endpoint-spectrum} gives
$|\Rcal(h,k)|=h^{k-1+o_k(1)}$ for fixed $k\geq4$.
\end{proof}

\section{Consequences and two spectrum questions}\label{sec:questions}

Theorem~\ref{thm:endpoint-spectrum} determines the ambient polynomial
exponent at every fixed cardinality $k\geq4$, while
Theorem~\ref{thm:main} gives the sharper positive-density conclusion at the
four-point transition.  We record consequences that transport this spectrum
abundance beyond its original one-dimensional formulation.

Recall that for an additive abelian group $G$ we write
$\Rcal_G(h,k)=\{|hA|:A\subset G,\ |A|=k\}$.

\begin{corollary}[Universality over torsion-free groups]
\label{cor:lattice-transfer}
Let $G$ be any nontrivial torsion-free abelian group.  For all $h,k\geq1$,
\begin{equation}\label{eq:torsion-free-spectrum-universality}
 \Rcal_G(h,k)=\Rcal(h,k).
\end{equation}
Consequently the spectrum-only conclusions
\eqref{eq:endpoint-local}, \eqref{eq:ambient-exponent}, and
\eqref{eq:localmain}, together with the four-point positive-density
conclusion, hold with $\Rcal$ replaced by $\Rcal_G$.  The diameter-filtered
integer-witness assertion \eqref{eq:quantmain} is not asserted for an
abstract group $G$.  In particular,
$|\Rcal_G(h,k)|=h^{k-1+o_k(1)}$ for every fixed $k\geq4$, and
$|\Rcal_G(h,4)|=\Theta(h^3)$ with positive lower density.
\end{corollary}

\begin{proof}
The case $k=1$ is immediate, so assume $k\geq2$.
Let $A\subset G$ be finite.  Its finitely generated subgroup is torsion-free
and hence admits a translation-invariant group order.  Theorem~
\ref{thm:ordered-group-shadow} gives an integer set with the same complete
$h$-addition table and therefore the same sumset cardinality.  This proves
$\Rcal_G(h,k)\subseteq\Rcal(h,k)$.  Conversely, if $g\ne0$ lies in $G$, the
map $n\mapsto ng$ embeds every finite integer set and preserves all sumset
cardinalities, proving the reverse inclusion.
\end{proof}

\begin{corollary}[Exponent dichotomy for abelian groups]
\label{cor:group-exponent-dichotomy}
Let $G$ be an abelian group and fix $k\geq4$.
\begin{enumerate}[label=\textup{(\roman*)}]
\item If $G$ has unbounded exponent, then
\begin{equation}\label{eq:unbounded-group-spectrum-exponent}
 |\Rcal_G(h,k)|=h^{k-1+o_k(1)}
 \qquad(h\to\infty).
\end{equation}
\item If $mG=0$ for some positive integer $m$, then
\begin{equation}\label{eq:bounded-group-spectrum}
 |\Rcal_G(h,k)|\leq m^{k-1}
 \qquad(h\geq1).
\end{equation}
\end{enumerate}
 Consequently, when $|G|\geq k$, the spectrum sizes are bounded in $h$ if
 and only if $G$ has bounded exponent.  The bounded-exponent case has
 logarithmic growth exponent $0$, while the unbounded-exponent case has
 exponent $k-1$.
\end{corollary}

\begin{proof}
If $G$ has unbounded exponent, Nathanson's embedding argument gives
\[
 \Rcal(h,k)\subseteq\Rcal_G(h,k)
\]
for every $h,k$ \cite[Theorem~3]{NathansonProblems}: for a finite integer
witness $B\subset[0,D]$, choose $g\in G$ of order greater than $hD$ and map
$b\mapsto bg$.  Theorem~\ref{thm:endpoint-spectrum} supplies the lower
exponent in \eqref{eq:unbounded-group-spectrum-exponent}, while the universal
bound $|hA|\leq M_{h,k}$ gives $|\Rcal_G(h,k)|\leq M_{h,k}=O_k(h^{k-1})$.

Now suppose $mG=0$.  Translate a $k$-set $A$ so that it contains zero.
The subgroup generated by its other $k-1$ elements is an image of
$(\mathbb Z/m\mathbb Z)^{k-1}$, so
\[
 |hA|\leq|\langle A\rangle|\leq m^{k-1}.
\]
Thus every element of $\Rcal_G(h,k)$ belongs to
$\{1,\ldots,m^{k-1}\}$, proving \eqref{eq:bounded-group-spectrum}.
If $|G|\geq k$, the spectrum is nonempty for every $h$, so the final
logarithmic statement is well defined.
\end{proof}

The torsion-free additive equality-profile transfer is classical: Ruzsa's
finite Freiman-model theorem gives an $F_h$-isomorphic model in $\mathbb Z$
\cite[Part~II, Lemma~2.3.4]{GeroldingerRuzsa2009}.  Rajagopal likewise
observes that his fixed-$h$ spectrum results extend to infinite torsion-free
abelian groups \cite[Section~5.4]{Rajagopal2026}.  The new refinement in
Corollary~\ref{cor:finite-profile-transfer} is ordered and quantitative: one
integer shadow simultaneously preserves the entire truncated profile, every
comparison, the active rank, and the relation-birth filtration through degree
$h$, at the optimal universal height order $\Theta_k(h^{k-2})$.  For
$G=\mathbb Z^m$, the single-degree equality contains Nathanson's recent
lattice theorem as a special case \cite[Theorem~1]{NathansonLattice2026}.
The transfer concerns attainable
cardinalities; the $[q]$-frequency statements remain one-dimensional, and
normalized diameter has no coordinate-free group analogue.  Corollary~
\ref{cor:defecttwolattice} nevertheless gives an exact $\ell_\infty$
obstruction, in every lattice dimension, for the single value $M_h-2$.
Kova\v{c}evi\'{c}'s recent work on an isodiametric theorem for the type-$A_3$
root lattice, with graph distance (equivalently, half the ambient $\ell_1$
distance), and on diameter-perfect finite-group $B_h$-sets concerns a
different packing-and-tiling extremum in a related relation-lattice setting
\cite{Kovacevic2026}.

For the complementary fixed-curve direction, Hoa and Tien derive regularity
bounds and eventual sumset structure from Ap\'ery data
\cite{HoaTien2026}.  We instead vary the exponent set and count the distinct
degree-$h$ values.

\begin{corollary}[Density of degree-$h$ Hilbert values]\label{cor:hilbert-density}
As $h\to\infty$, let $C_A\subset\mathbb P^3$ range over the projective monomial curves
associated with four distinct integer exponents.  The number of distinct
degree-$h$ Hilbert values $H_{C_A}(h)$ is $\Theta(h^3)$.  Moreover,
$\gg h^3$ such values already occur, inside the window in
\eqref{eq:localmain}, among curves in the family from
Theorem~\ref{thm:family} whose coordinate rings are arithmetically
Cohen--Macaulay of type two.
\end{corollary}

\par\medskip
\begin{proof}
Elias's identity \eqref{eq:eliasbridge} identifies $H_{C_A}(h)$ with
$|hA|$.  The lower bound, including the ACM assertion and the localized
window, is the construction in the proof of Theorem~\ref{thm:main}; the
upper bound follows because every Hilbert value is an integer in
$[1,M_h]$, so there are at most $M_h=O(h^3)$ distinct values.
\end{proof}

Corollary~\ref{cor:lattice-transfer} combines the classical torsion-free
equality-profile transfer with the new ordered, filtered, sharp-height shadow
theorem.  The
Hilbert consequences likewise occur in every projective embedding dimension
$k-1$.  Corollary~
\ref{cor:endpoint-acm-type-two} shows that, for every fixed $k\geq4$, the
endpoint lower bound already counts degree-$h$ Hilbert values of
arithmetically Cohen--Macaulay projective monomial curves in
$\mathbb P^{k-1}$ of Cohen--Macaulay type two.  Their defining ideals are
prime toric almost complete intersections, and their coordinate rings have
the fixed total Betti polynomial $(1+2\xi)(1+\xi)^{k-3}$.  Thus the algebraic chart and its
suspension tower are not merely proof devices: one rigid homological format
already realizes the full ambient polynomial exponent.

It is natural to ask for the finer limiting distribution.

\begin{question}[Limiting density]\label{q:density}
Does
\[
 \frac{|\Rcal(h,4)|}
 {|[3h+1,\binom{h+3}{3}]\cap\Z|}
\]
have a limit?  If so, is the limit equal to $1$?
\end{question}

Our proof produces positive density from a single Cohen--Macaulay chamber
and uses only a coarse energy-to-image inequality.  Determining the image
density of the four-cubic map in \eqref{eq:fourtermdefect}, or combining
several chambers with controlled overlap, may yield a stronger density bound.

Theorem~\ref{thm:rank-frequency} determines the exact quasipolynomial degree
and an effective leading coefficient for every attainable cardinality.  At four
elements, Theorem~\ref{thm:spectral-rank} completely identifies the three
minimum-rank layers: the maximum has rank zero, the universal principal list
has rank one, and every other attainable value has rank two.  What remains
unknown is which individual integers populate that rank-two layer and how
those integers are distributed.  For every fixed $k\geq5$,
Theorem~\ref{thm:maximal-rank-saturation}\textup{(ii)} shows that, at each sufficiently large
$h$, at least one intrinsic rank
$r^\ast_{h,k}\in\{2,\ldots,k-2\}$ supports $h^{k-1-o(1)}$ labels, each admitting
a maximal-rank prime toric ACM endpoint witness.  The selected rank need not
be canonical and may vary with $h$.
Proposition~\ref{prop:stable-factor-barrier} further shows that eliminating
the subpolynomial loss in \eqref{eq:endpoint-local} requires a mechanism
beyond the present stable factorized chart.

\begin{question}[Higher-cardinality density and rank stratification]
\label{q:higher-spectrum}
Fix $k\geq5$.  Is $|\Rcal(h,k)|=\Theta_k(h^{k-1})$?  More sharply, does
\[
 \frac{|\Rcal(h,k)|}
 {|[hk-h+1,\binom{h+k-1}{k-1}]\cap\Z|}
\]
converge, and if so is its limit positive?  Can the subpolynomial factor in
\eqref{eq:endpoint-local} be replaced by $(\log h)^{-O_k(1)}$, or by a
positive constant?  For
$2\leq r\leq k-2$, determine the growth of the minimum-rank strata
$|\Rcal^{[r]}(h,k)|$ as $h\to\infty$.  Is there a fixed
$r\in\{2,\ldots,k-2\}$ for which
$|\Rcal^{[r]}(h,k)|=h^{k-1-o(1)}$ as $h\to\infty$, and if so, which $r$?
\end{question}

The active-relation filtration now has exact descriptions in three
distinguished rank regimes.  Theorem~\ref{thm:optimal-four-mark} gives the
optimal finite ruler at the rank-zero four-mark $B_h$ endpoint.  At rank one,
Theorem~\ref{thm:rankone-multiplicity} shows that a primitive relation
determines the entire representation histogram, while
Corollary~\ref{cor:conditional-collision-phase} gives the conditional
collision phase.  For maximal-rank
representatives, the endpoint amplifier supplies the full spectral exponent
together with rigid ACM structure.  Thus, for every sufficiently large $h$,
the construction and the finite-observation geometry meet in a possibly
$h$-dependent intrinsic minimum-rank stratum.  The construction supplies a
full-exponent family of labels with maximal-rank ACM witnesses, while the
geometry identifies their intrinsic codimension and realization rarity.  The
remaining synthesis problem is to identify or stabilize that rank, determine
every rank stratum, and obtain positive density by a mechanism not confined
to this single fixed-scale full-residue chart.

\bibliographystyle{amsplain}
\enlargethispage{2\baselineskip}
\begin{sloppypar}
\providecommand{\bysame}{\leavevmode\hbox to3em{\hrulefill}\thinspace}
\providecommand{\href}[2]{#2}

\end{sloppypar}

\end{document}